\documentclass[11pt]{amsart}
\usepackage[color=blue!20!white,textsize=tiny]{todonotes}

\usepackage{bm}
\usepackage{multicol}

\usepackage{mathrsfs}
\usepackage{amsmath} 
\usepackage{amsthm}
\usepackage{amssymb} 
\usepackage{amsfonts}
\usepackage{tikz-cd}
\usepackage{graphicx}
\usepackage{xcolor}
\definecolor{darkgreen}{rgb}{0,0.5,0}
\usepackage[normalem]{ulem}
\usepackage{comment}

\usepackage{caption,subcaption,pinlabel}
\usepackage{graphics}
\usepackage[pagebackref=true]{hyperref}
\usepackage{scalerel}
\usepackage{comment}
\usepackage{enumitem}
\usepackage{mathtools}
\usepackage{tikz}
\usetikzlibrary{matrix,arrows,decorations.pathmorphing}
\usepackage{mathabx}
\usetikzlibrary{matrix}
\usepackage{scrextend}

\newtheorem{theorem}{Theorem}[section]
\newtheorem{lemma}[theorem]{Lemma}

\theoremstyle{definition}
\newtheorem{definition}[theorem]{Definition}
\newtheorem{example}[theorem]{Example}

\newtheorem{remark}[theorem]{Remark}

\newcommand{\im}{\textrm{im}}

\newcommand{\ita}{\iota \tau}
\def\Inv{\mathfrak{I}}

\newcommand{\bunderline}[1]{\underline{#1\mkern-2mu}\mkern2mu }

\def\du {\bar{d}}
\def\dl {\bunderline{d}}

\def\H{\mathcal{H}}
\newcommand \bH {\bar{\H}}
\newcommand \brho {\bar{\rho}}

\def\spinc {{\operatorname{spin^c}}}
\def\s{\mathfrak s}
\def\bs{\bar{\mathfrak{\s}}}

\def\x{\mathbf{x}}
\def\CF {\mathit{CF}}
\def\HF {\mathit{HF}}

\newcommand \CFm {\CF^-}
\newcommand \HFm {\HF^-}

\def\HFK {\mathit{HFK}}
\newcommand\HFKhat{\widehat{\HFK}}

\def\CFI {\mathit{CFI}}

\newcommand\alphas{\boldsymbol\alpha}
\newcommand\betas{\boldsymbol\beta}

\newcommand{\Z}{\mathbb{Z}}
\newcommand{\Q}{\mathbb{Q}}

\newcommand{\F}{\mathbb{F}}

\let\int\relax
\newcommand{\int}{\mathring}

\usepackage{accents}

\DeclareMathSymbol{\wtilde}{\mathord}{largesymbols}{"65}

\mathchardef\mhyphen="2D

\newcommand{\obar}[1]{\mkern 1.5mu\overline{\mkern-1.5mu#1\mkern-1.5mu}\mkern 1.5mu}
\newcommand{\ubar}[1]{\mkern 1.5mu\underline{\mkern-1.5mu#1\mkern-1.5mu}\mkern 1.5mu}

\newcommand{\Vtu}{\smash{\obar{V}_0^{\tau}}}
\newcommand{\Vtl}{\smash{\ubar{V}_0^{\tau}}}
\newcommand{\Vitu}{\smash{\obar{V}_0^{\iota\tau}}}
\newcommand{\Vitl}{\smash{\ubar{V}_0^{\iota\tau}}}

\newcommand{\Vsu}{\smash{\obar{V}_0^{\circ}}}
\newcommand{\Vsl}{\smash{\ubar{V}_0^{\circ}}}

\newcommand{\eg}{\smash{\widetilde{g}_4}}
\newcommand{\bg}{\smash{\widetilde{bg}_4}}
\newcommand{\ieg}{\smash{\widetilde{ig}_4}}
\newcommand{\must}{\mu_{\text{st}}}
\newcommand{\mugst}{\mu_{\text{gst}}}

\newcommand{\tausum}{\tau_{\#}}
\newcommand{\tausw}{\tau_{sw}}

\newcommand{\CFK}{\mathcal{CFK}}

\newcommand{\K}{\mathfrak{K}}

\newcommand{\C}{\mathcal{C}}
\newcommand{\eC}{\smash{\widetilde{\C}}}

\newcommand{\gr}{\text{gr}}
\newcommand{\cR}{\mathcal{R}}
\newcommand{\grU}{\text{gr}_U}
\newcommand{\grV}{\text{gr}_V}

\newcommand{\id}{\mathrm{id}}

\newcommand{\Co}{C_0}

\newcommand{\scU}{\mathscr{U}}
\newcommand{\scV}{\mathscr{V}}

\newcommand{\cB}{\mathcal{B}}
\newcommand{\cC}{\mathcal{C}}
\newcommand{\cD}{\mathcal{D}}
\newcommand{\cE}{\mathcal{E}}

\newcommand{\cS}{\mathcal{S}}

\newcommand{\cU}{\mathscr{U}}
\newcommand{\cV}{\mathscr{V}}
\newcommand{\cW}{\mathcal{W}}

\newcommand{\cY}{\mathcal{Y}}

\newcommand{\Tr}{T^{\hspace{0.03cm}r}}

\title{Equivariant knots and knot Floer homology}

\author[I. Dai]{Irving Dai}
\thanks{The first author was partially supported by NSF grant DMS-1902746.}
\address {Department of Mathematics, Stanford University, Palo Alto, CA 94301}
\email{ifdai@stanford.edu}

\author[A. Mallick]{Abhishek Mallick}
\address{Department of Mathematics, Rutgers University, Piscataway, NJ 08854}
\email{abhishek.mallick@rutgers.edu}

\author[M. Stoffregen]{Matthew Stoffregen}
\thanks{The third author was partially supported by NSF grant DMS-1952755.}
\address {Department of Mathematics, Michigan State University, East Lansing, MI 48824}
\email{stoffre1@msu.edu}
\begin{document}
\vspace*{-1cm}
\maketitle
\vspace*{-0.4cm}
\begin{abstract}
We define several equivariant concordance invariants using knot Floer homology. We show that our invariants provide a lower bound for the equivariant slice genus and use this to give a family of strongly invertible slice knots whose equivariant slice genus grows arbitrarily large, answering a question of Boyle and Issa. We also apply our formalism to several seemingly non-equivariant questions. In particular, we show that knot Floer homology can be used to detect exotic pairs of slice disks, recovering an example due to Hayden, and extend a result due to Miller and Powell regarding stabilization distance. Our formalism suggests a possible route towards establishing the non-commutativity of the equivariant concordance group.
\end{abstract}

\section{Introduction}\label{sec:1}
Equivariant knots and concordance have been well-studied historically; see for example \cite{Murasugi, Sakuma, Naik, ChaKo, ND}. Recently, there has been a renewed interest in this topic from the viewpoint of more modern invariants, as evidenced by the works of Watson~\cite{Watson}, Lobb-Watson~\cite{LW} and Boyle-Issa~\cite{BI}. The aim of the present article is to investigate the theory of equivariant knots through the lens of knot Floer homology, an extensive package of invariants introduced independently by Ozsv\'ath-Szab\'o \cite{OSknots} and Rasmussen \cite{Rasmussen}. Our underlying approach is straightforward: given a strongly invertible knot $(K, \tau)$, we show that $\tau$ induces an appropriately well-defined automorphism of the knot Floer complex $\CFK(K)$. Using the induced action of $\tau$, we construct the following suite of numerical invariants:

\begin{theorem}\label{thm:1.1}
Let $(K, \tau)$ be a strongly invertible knot in $S^3$. Associated to $(K, \tau)$, we have four integer-valued equivariant concordance invariants
\[
\Vtu(K) \leq \Vtl(K) \quad \text{and} \quad \Vitu(K) \leq \Vitl(K).
\]
In fact, $\Vsu$ and $\Vsl$ $($where $\circ \in \{\tau, \iota\tau\})$ are invariant under the more general relation of isotopy-equivariant homology concordance.
\end{theorem}
\noindent
Note that $\Vsu$ and $\Vsl$ vanish if $K$ is equivariantly slice. See Definition~\ref{def:isotopyeqconc} for the definition of isotopy-equivariant homology concordance. 

Obstructions to equivariant sliceness have been investigated by several authors, including Sakuma~\cite{Sakuma}, Cha-Ko~\cite{ChaKo}, and Naik-Davis~\cite{ND}. However, understanding the equivariant slice genus has only more recently been studied by Boyle-Issa~\cite{BI}. One of the main results of this paper will be to show that $\Vsu$ and $\Vsl$ provide lower bounds for the equivariant slice genus $\eg(K)$ of $(K, \tau)$. In fact, we show that they bound the \textit{isotopy-equivariant slice genus}; see Definition~\ref{def:isotopyeqgenus}. Using this, we provide a family of strongly invertible slice knots $(K_n, \tau_n)$ whose equivariant slice genus grows arbitrarily large, answering a question posed by Boyle-Issa. Prior to the current article, there were no known examples of strongly invertible knots with $\eg(K) - g_4(K)$ provably greater than one. 

Surprisingly, our invariants also have applications to several (seemingly) non-equivariant questions. We first show that our formalism can be used to detect exotic pairs of slice disks, recovering an example originally due to Hayden \cite{Hayden}. Note that while knot Floer homology has previously been used to detect exotic higher-genus surfaces (see the work of Juh\'asz-Miller-Zemke \cite{JMZ}), the current work represents the first such application of knot Floer homology in the genus-zero case. We also consider the question of bounding the stabilization distance between pairs of disks. Using the work of Juh\'asz-Zemke \cite{JZstabilization}, we show that our examples $K_n$ give a Floer-theoretic re-proof and extension of a result by Miller-Powell~\cite{MP}, which states that for each integer $m$, there is a knot $J_m$ with a pair of slice disks that require at least $m$ stabilizations to become isotopic.

The invariants of Theorem~\ref{thm:1.1} are correction terms derived from the action of $\tau$ on $\CFK(K)$ following the general algebraic program of Hendricks-Manolescu~\cite{HM}. Instead of working with numerical invariants, it is also possible to define a local equivalence group in the style of Hendricks-Manolescu-Zemke~\cite{HMZ} or Zemke~\cite{Zemkeconnected}. This follows the approach taken in Dai-Hedden-Mallick in \cite{DHM} to study cork involutions; and, indeed, the current article is closely related to \cite{DHM}. In this paper, we define the \textit{local equivalence group $\K_{\tau, \iota}$ of $(\tau_K, \iota_K)$-complexes} and show that there is a homomorphism from the equivariant concordance group $\eC$ (defined by Sakuma in \cite{Sakuma}) to $\K_{\tau, \iota}$: 
\[
h_{\tau, \iota} \colon \eC \rightarrow \K_{\tau, \iota}.
\]
Interestingly, it turns out that $\K_{\tau, \iota}$ is not \textit{a priori} abelian. It is an open problem whether $\eC$ is abelian; in principle, our invariants can thus be used to provide a negative answer to this question.\footnote{Recently, Di Prisa has shown that $\eC$ is indeed non-abelian \cite{DiPrisa}; see Remark~\ref{rem:currentdevelopments}.} As far as the authors are aware, this is the first example of a (possibly) non-abelian group arising in the setting of local equivalence. See Section~\ref{sec:2} for background and further discussion.

Although all of the examples in this paper will be strongly invertible, we also establish several analogous results for $2$-periodic knots. We discuss these in Section~\ref{sec:8}.

\subsection{Equivariant slice genus bounds}\label{sec:1.1}
Our first application will be to show that the invariants of Theorem~\ref{thm:1.1} bound the equivariant slice genus $\eg(K)$ of $K$ (see Definition~\ref{def:eqgenus}). In fact, we give a bound for a rather more general quantity, defined as follows. 

Let $(K, \tau)$ be a strongly invertible knot. Let $W$ be a (smooth) homology ball with boundary $S^3$, and consider any (smooth) self-diffeomorphism $\tau_W$ on $W$ which restricts to $\tau$ on $\partial W$. Note that we do not require $\tau_W$ itself to be an involution. We say that a slice surface $\Sigma$ in $W$ with $\partial \Sigma = K$ is an \textit{isotopy-equivariant slice surface} (for the given data) if $\tau_W(\Sigma)$ is isotopic to $\Sigma$ rel $K$. Define the \textit{isotopy-equivariant slice genus} of $(K, \tau)$ by:
\[
\ieg(K) = \min_{\substack{\text{all possible choices of $W$ and $\tau_W$} \\ \text{all isotopy-equivariant slice surfaces $\Sigma$}}} \{g(\Sigma)\}.
\]
Here $\ieg(K)$ depends on $\tau$, but we suppress this from the notation.
The quantity $\ieg(K)$ generalizes the obvious notion of equivariant slice genus in several ways. Firstly, we allow ourselves to consider any homology ball $W$ and any diffeomorphism which extends $\tau$, rather than restricting ourselves to $B^4$. Secondly, we do not require that $\Sigma$ be invariant under the extension of $\tau$, but instead only isotopic to its image. Obviously, 
\[
\ieg(K) \leq \eg(K).
\]
Although the authors do not have an example in which $\ieg(K)$ is distinct from $\eg(K)$, this more general quantity will turn out to be critical for several applications. There is also an obvious accompanying notion of isotopy-equivariant homology concordance; see Definition~\ref{def:isotopyeqconc}.


Although the notion of isotopy equivariance may initially seem rather contrived, a slight shift in perspective demonstrates its usefulness. To see this explicitly, let $(K, \tau)$ be a strongly invertible knot in $S^3$. Let $W$ be any (smooth) homology ball with boundary $S^3$ and $\tau_W$ be any extension of $\tau$ over $W$. If $\Sigma \subseteq W$ is any slice surface for $K$ with $g(\Sigma) < \ieg(K)$, then we may immediately conclude that the two surfaces $\Sigma$ and $\tau_W(\Sigma)$ are not isotopic rel $K$. The calculation of $\ieg(K)$ thus provides an easy method for generating non-isotopic slice surfaces in the presence of a symmetry on $K$. For example, if $K$ is an equivariant slice knot with $\ieg(K) > 0$, then we may take \textit{any} slice disk $\Sigma$ for $K$ and form its image under \textit{any} extension $\tau_W$ of $\tau$ (in \textit{any} homology ball $W$); the resulting pair of slice disks are then automatically non-isotopic rel $K$. We often refer to $\Sigma$ and $\tau_W(\Sigma)$ as a \textit{symmetric pair of slice disks}. This is in marked contrast to the usual approach taken in the literature, where in order to deploy various invariants, one (naturally) has in mind a specific family of slice disks (or surfaces) that are conjectured to be non-isotopic. The situation here is analogous to the notion of a strong cork introduced by Lin-Ruberman-Saveliev in \cite{LRS} and studied in \cite{DHM}.

Following the work of Juh\'asz-Zemke \cite{JZgenus}, we bound $\ieg(K)$ in terms of $\Vsu$ and $\Vsl$:

\begin{theorem}\label{thm:1.2}
Let $(K, \tau)$ be a strongly invertible knot in $S^3$. Then for $\circ \in \{\tau, \iota\tau\}$,
\[
- \left \lceil{\frac{1+ \ieg(K)}{2}} \right \rceil \leq \Vsu(K) \leq \Vsl(K) \leq \left \lceil{\frac{1+ \ieg(K)}{2}} \right \rceil.
\]
\end{theorem}

The computation of $\Vsu(K)$ and $\Vsl(K)$ can thus be used to help construct exotic pairs of slice surfaces for $K$, via the discussion above. In the current paper, we only give the most archetypal instance of this phenomenon; the authors plan to return to the task of finding a systematic range of examples in future work. Note that by Theorem~\ref{thm:1.1}, if $\ieg(K) = 0$ then $\Vsu$ and $\Vsl$ vanish. In the genus-zero case, Theorem~\ref{thm:1.1} thus gives a slightly stronger bound than that of Theorem~\ref{thm:1.2}. This discrepancy is explained in Remark~\ref{rem:weakerbound}.

\subsection{Applications}\label{sec:1.2}

We now give several computations and applications. Our main class of examples is quite straightforward. Let $T_{2n, 2n+1}$ be the right-handed torus knot and select any strong inversion $\tau$ on $T_{2n, 2n+1}$ (in fact, this is unique up to conjugation by \cite[Proposition 3.1]{Sakuma}). As in Figure~\ref{fig:11}, there are two obvious strong inversions on $T_{2n, 2n+1} \# T_{2n, 2n+1}$. On one hand, we may take the equivariant connected sum $\tau_\# = \tau \# \tau$ to obtain an inversion with one fixed point on each summand. On the other, we may consider the strong inversion $\tau_{sw}$ which interchanges the two factors. Strictly speaking, the latter is a strong inversion on $T_{2n, 2n+1} \# \Tr_{2n, 2n+1}$; however, since $T_{2n, 2n+1}$ admits an orientation-reversing symmetry, we will occasionally conflate this with $T_{2n, 2n+1} \# T_{2n, 2n+1}$.

We then consider the further equivariant connected sum
\[
K_n = (T_{2n, 2n+1} \# T_{2n, 2n+1}) \# - (T_{2n, 2n+1} \# T_{2n, 2n+1})
\]
equipped with the strong inversion
\[
\tau_n = \tausum \# -\tau_{sw}.
\]
That is, we consider the strong inversion $\tausum$ on the first copy of $T_{2n, 2n+1} \# T_{2n, 2n+1}$ and take the equivariant connected sum of this with the (orientation-reversed mirror of the) inversion $\tausw$ on $T_{2n, 2n+1} \# T_{2n, 2n+1}$. For a discussion of the equivariant connected sum of two strong inversions, see Section~\ref{sec:2.1}. In general, defining the equivariant connected sum requires some additional data, but the application we have in mind will be insensitive to this subtlety; see Remark~\ref{rem:suminsensitive}. In Figure~\ref{fig:11} we perform the equivariant connected sum by (roughly speaking) stacking successive axes end-to-end. Note that $K_n$ is slice.

\begin{figure}[h!]
\includegraphics[scale = 0.76]{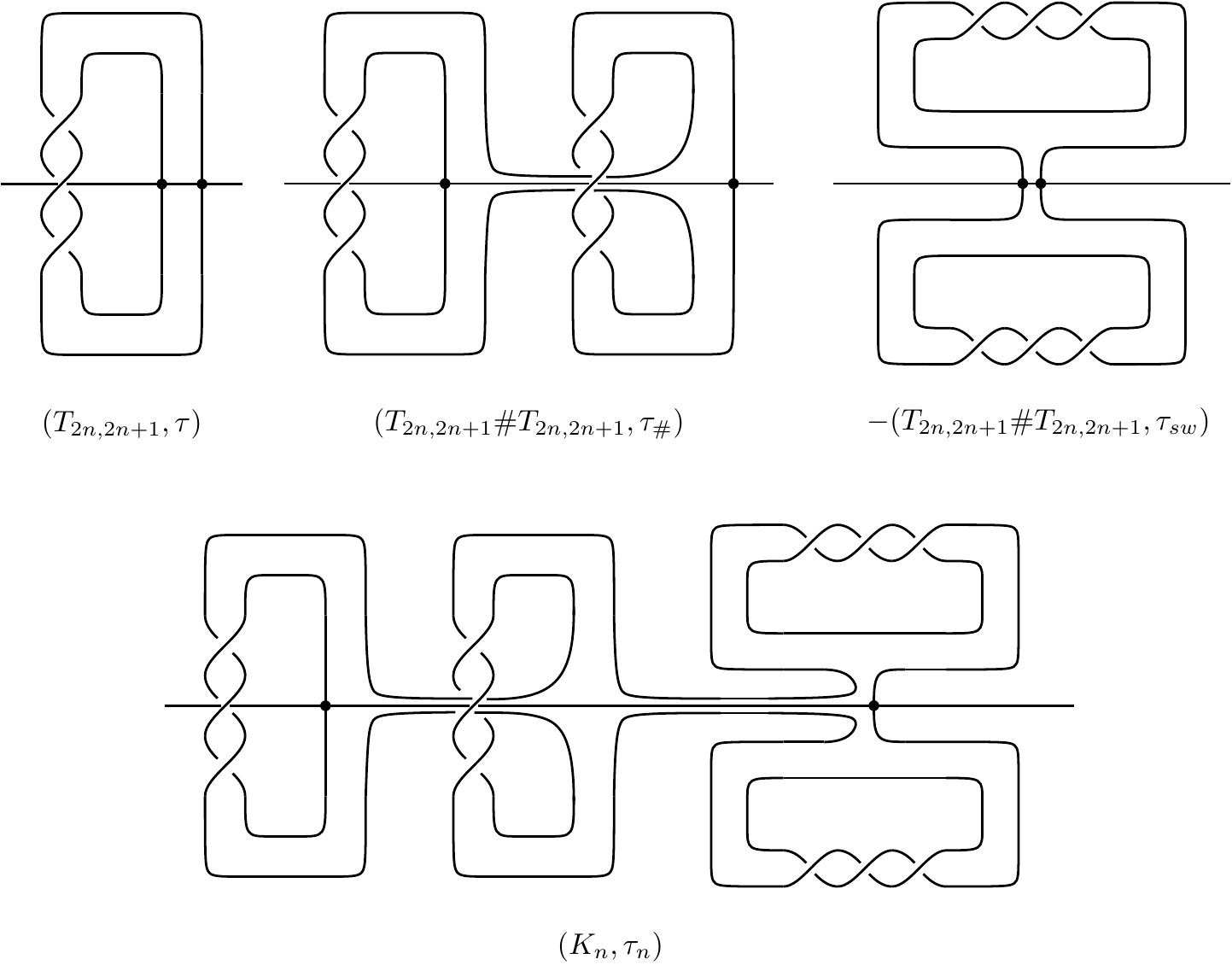}
\caption{Schematic depiction of the case $n = 1$. Top left: a strong inversion on $T_{2n, 2n+1}$. Top middle: the equivariant connected sum inversion $\tau_\#$ on $T_{2n, 2n+1} \# T_{2n, 2n+1}$. Top right: (the mirror of) the strong inversion $\tausw$ on $T_{2n, 2n+1} \# T_{2n, 2n+1}$. Bottom: construction of $K_n$ and $\tau_n$.}\label{fig:11}
\end{figure}

In Section~\ref{sec:6}, we establish the following fundamental calculation:
\begin{theorem}\label{thm:1.4}
For $n$ odd, the pair $(K_n, \tau_n)$ has $\Vtl(K_n) \geq n$.
\end{theorem}
\noindent
Similar knots were investigated by Hendricks-Hom-Stoffregen-Zemke in \cite{HHSZ2} and the proof of Theorem~\ref{thm:1.4} relies on the computations of \cite{HHSZ2}. In fact, we also establish that $\Vtu(K_n) \geq 0$, although this is of limited use, and conjecture that the inequality appearing in Theorem~\ref{thm:1.4} is an equality. However, since we do not need this for any application, we leave the more detailed computation to the reader. 

In \cite[Question 1.1]{BI}, Boyle-Issa asked whether there exists a family of strongly invertible knots for which $\eg(K) - g_4(K)$ becomes arbitrarily large. Applying Theorem~\ref{thm:1.2}, we immediately obtain:

\begin{theorem}\label{thm:1.5}
For $n$ odd, the pair $(K_n, \tau_n)$ has
\[
2n- 2 \leq \ieg(K_n) \leq \eg(K_n).
\]
\end{theorem}
\noindent
Since each $K_n$ is slice, this answers \cite[Question 1.1]{BI} in the affirmative. The topological intuition behind these examples is quite straightforward: the involutions $\tau_\#$ and $\tau_{sw}$ on $T_{2n, 2n+1} \# T_{2n, 2n+1}$ are very different, so one should expect the equivariant slice genus of $(K_n, \tau_n)$ to be large.

We also consider a particular knot $J$ due to Hayden \cite{Hayden}, displayed in Figure~\ref{fig:12}. In \cite{Hayden}, Hayden presents a certain pair of slice disks $D$ and $D'$ for $J$, each with complement having fundamental group $\Z$. By a result of Conway-Powell \cite[Theorem 1.2]{CP}, this implies that $D$ and $D'$ are topologically isotopic. However, in \cite[Section 2.1]{Hayden}, it is shown that $D$ and $D'$ are not smoothly isotopic (or even diffeomorphic) rel boundary. (See also \cite[Theorem 3.2]{HS}.) Note that $J$ admits a strong inversion $\tau$; a crucial part of the argument in \cite{Hayden} relies on the fact that $D$ and $D'$ are related by the obvious extension of $\tau$ over $B^4$. 


\begin{figure}[h!]
\includegraphics[scale = 0.7]{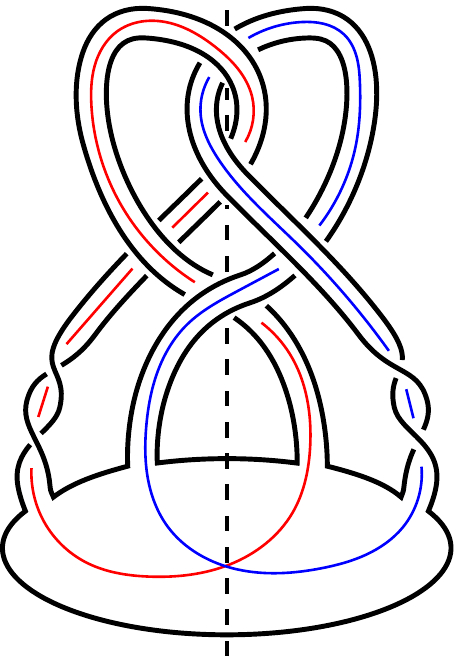}
\caption{An equivariant slice knot $J$ with symmetry $\tau$ given by reflection across the obvious vertical axis. The slice disks $D$ and $D'$ are obtained by compressing along the red and blue curves, respectively.}\label{fig:12}
\end{figure}


In \cite[Section 2.1]{Hayden}, it is noted that $J$ has a close connection to the positron cork $W_0$ of Akbulut-Matveyev~\cite{AkbulutMat}. In \cite[Theorem 1.15]{DHM}, the action of the cork involution on the Heegaard Floer homology of $\partial W_0$ was investigated. Re-casting these computations in the formalism of the current paper yields:



\begin{theorem}\label{thm:1.7}
Let $J$ be as in Figure~\ref{fig:12}. Then $\ieg(J) > 0$. In particular, no pair of symmetric slice disks $\Sigma$ and $\tau_W(\Sigma)$ are (smoothly) isotopic rel $J$. This holds for any (smooth) homology ball $W$ with $\partial W = S^3$ and any extension $\tau_W$ of $\tau$ over $W$.
\end{theorem}

Given the connection between $J$ and $W_0$, it is actually possible to use \cite[Theorem 1.15]{DHM} to provide an immediate proof of Theorem~\ref{thm:1.7}, as we explain in Remark~\ref{rem:immediateproof}. However, going through the proof in the current context explicitly gives:

\begin{theorem}\label{thm:knotfloermaps}
Let $J$ be as in Figure~\ref{fig:12} and let $\Sigma$ and $\tau_W(\Sigma)$ be any pair of symmetric slice disks for $J$. Then
\[
[F_{W, \Sigma}(1)] \neq [F_{W, \tau_W(\Sigma)}(1)]
\]
as elements in either $H_*(\CFK(J))$ or $\HFKhat(J)$. This holds for any (smooth) homology ball $W$ with $\partial W = S^3$ and any extension $\tau_W$ of $\tau$ over $W$.
\end{theorem}
\noindent
Here, $F_{W, \Sigma}$ and $F_{W, \tau_W(\Sigma)}$ are the knot Floer cobordism maps associated to (punctured copies of) $\Sigma$ and $\tau_W(\Sigma)$, respectively. We thus explicitly see that $\Sigma$ and $\tau_W(\Sigma)$ are distinguished by their maps on knot Floer homology. Specializing to $\Sigma = D$, this provides a knot Floer-theoretic analogue of the proof of \cite[Theorem 3.2]{HS}, in which $D$ and $D'$ are distinguished using their induced maps on Khovanov homology. Note that Juh\'asz-Miller-Zemke have used knot Floer homology to detect exotic higher-genus surfaces \cite{JMZ}. (The fact that the surfaces have genus greater than zero is essential to their argument.) However, the current work represents the first instance of knot Floer homology being applied to detect an exotic pair of disks. 

By taking the $n$-fold connected sum $\#_nJ$, it is also straightforward to construct an example of a slice knot with $2^n$ different exotic slice disks, which are distinguished by their concordance maps on $\smash{\HFKhat}$. We establish this in Theorem~\ref{thm:2ndisks}; see \cite[Corollary 6.6]{SundbergSwann} for a similar construction. In Theorem~\ref{thm:exoticfamily}, we extend Theorem~\ref{thm:1.7} to an infinite family of knots with exotic pairs of slice disks, which were likewise considered by Hayden in \cite{Hayden}.

\subsection{Algebraic formalism}\label{sec:1.3}

As discussed previously, our underlying goal will be to show that a strong inversion $\tau$ induces a well-defined action on the knot Floer complex of $K$. We also 
incorporate the involutive knot Floer automorphism of Hendricks-Manolescu \cite{HM} into our formalism, which will allow us to define the invariants $\Vitu$ and $\Vitl$. In order to construct the action of $\tau$, we first fix an orientation on $K$ and an ordered pair of basepoints $(w, z)$ which are interchanged by $\tau$. We refer to this data as a \textit{decoration} on $(K, \tau)$. In Section~\ref{sec:3.2}, we define the action of $\tau$ associated to a decorated strongly invertible knot:

\begin{theorem}\label{thm:1.8}
Let $(K, \tau)$ be a decorated strongly invertible knot. Let $\H$ be any choice of Heegaard data compatible with $(K, w, z)$. Then $\tau$ induces an automorphism
\[
\tau_\H \colon \CFK(\H) \rightarrow \CFK(\H)
\]
with the following properties:
\begin{enumerate}
\item $\tau_\H$ is skew-graded and $\F[\cU, \cV]$-skew-equivariant
\item $\tau_\H^2 \simeq \id$
\item $\tau_\H \circ \iota_\H \simeq \varsigma_\H \circ \iota_\H \circ \tau_\H$
\end{enumerate}
Here, $\iota_\H$ is the Hendricks-Manolescu knot Floer involution on $\CFK(\H)$ and $\varsigma_\H$ is the Sarkar map. Moreover, the homotopy type of the triple $(\CFK(\H), \tau_\H, \iota_\H)$ is independent of the choice of Heegaard data $\H$ for the doubly-based knot $(K, w, z)$.
\end{theorem}
\noindent
This action was originally considered by the second author in the context of establishing a large surgery formula; see \cite{Mallick}. Note that $\tau_\H$ and $\iota_\H$ do not in general commute. This is in contrast to the 3-manifold setting; see \cite[Lemma 4.4]{DHM}.

In view of the last part of Theorem~\ref{thm:1.8}, we may suppress writing $\H$ and unambiguously refer to the homotopy type of $(\CFK(K), \tau_K, \iota_K)$ as an invariant of the decorated knot $(K, \tau)$. In Section~\ref{sec:2.2}, we formalize this algebraic data by defining the notion of an \textit{abstract $(\tau_K, \iota_K)$-complex}. We define an appropriate notion of local equivalence and form the quotient
\[
\K_{\tau, \iota} = \{\text{abstract }(\tau_K, \iota_K)\text{-complexes}\}\ / \ \text{local equivalence}.
\]
See Section~\ref{sec:2.2}.

The role of the decoration on $(K, \tau)$ turns out to be quite subtle. As we will see, this extra choice of data is needed to form the knot Floer complex of $K$ and is critical for discussing the invariance properties of $(\CFK(K), \tau_K, \iota_K)$. In Section~\ref{sec:3.4}, we introduce the notion of a \textit{decorated} isotopy-equivariant homology concordance and show that in the decorated category, we obtain a map
\[
h_{\tau, \iota} \colon \{\text{(decorated) strongly invertible knots}\} \ \big/\ \parbox{14em}{\centering (decorated) isotopy-equivariant \\ homology concordance} \rightarrow \K_{\tau, \iota}.
\]
However, this is (in principle) not quite true if the decorations are discarded: in the undecorated setting, an equivariant knot only defines a $(\tau_K, \iota_K)$-complex up to a certain ambiguity which we refer to as a \textit{twist by $\varsigma_K$}; see Definition~\ref{def:twist}. Nevertheless, we show that $\Vsu$ and $\Vsl$ remain invariants in the undecorated setting.

In Section~\ref{sec:2.2}, we further define a product operation on $\K_{\tau, \iota}$ which makes it into a group. We establish an equivariant connected sum formula in Theorem~\ref{thm:connectedsum}; this will allow us to prove that $h_{\tau, \iota}$ constitutes a homomorphism from the equivariant concordance group $\eC$ to $\K_{\tau, \iota}$.

\begin{theorem}\label{thm:1.9}
We have a homomorphism
\[
h_{\tau, \iota} \colon \eC \rightarrow \K_{\tau, \iota}.
\]
\end{theorem}
\noindent
The equivariant concordance group $\eC$ consists of the set of \textit{directed} strongly invertible knots; see Definition~\ref{def:direqconc}. In Section~\ref{sec:3.5}, we discuss the connection between $\eC$ and decorated isotopy-equivariant concordance. 

Somewhat surprisingly, it turns out that $\K_{\tau, \iota}$ is not \textit{a priori} abelian, although the authors have no explicit example of this. As we discuss in Section~\ref{sec:2.1}, it is currently unknown whether $\eC$ is abelian. Hence in principle Theorem~\ref{thm:1.9} can be used to provide examples demonstrating this claim; we plan to return to this question in future work. As far as the authors are aware, this is the first example of a (possibly) non-abelian group arising in the setting of local equivalence. Note that the $\iota_K$-local equivalence group of Zemke \cite{Zemkeconnected} is abelian.


\subsection{Relation to $3$-manifold invariants} \label{sec:1.4}

If $K$ is an equivariant knot, then any 3-manifold obtained by surgery on $K$ inherits an involution from the symmetry on $K$ (see for example \cite[Lemma 5.2]{DHM}). In \cite{Mallick}, the second author established a large surgery formula relating the action of $\tau_K$ to the corresponding Heegaard Floer action of the 3-manifold involution. This latter action was defined and studied in \cite{DHM} in the context of the theory of corks. It follows immediately from the large surgery formula that (with appropriate normalization) the invariants $\Vsu$ and $\Vsl$ are none other than the numerical involutive correction terms referenced in \cite[Remark 4.5]{DHM}. Explicitly, for $p \geq g_3(K)$, we have
\begin{equation}\label{eq:knotto3manifold}
\begin{split}
&-2\Vsu(K) + \dfrac{p-1}{4} = \du_\circ(S^3_p(K), [0]) \\
&-2\Vsl(K) + \dfrac{p-1}{4} = \dl_\circ(S^3_p(K), [0]).
\end{split}
\end{equation}
for $\circ \in \{\tau, \ita\}$. See \cite[Theorem 1.6]{HM} for the analogous statements concerning the usual involutive numerical invariants $\du$ and $\dl$.

The results of \cite{DHM} easily imply that $\Vsu$ and $\Vsl$ are invariant under equivariant concordance (essentially by surgering along the concordance annulus). Hence it is actually immediate that $\Vsu$ and $\Vsl$ obstruct equivariant sliceness. Indeed, \cite{DHM} already gives several examples of slice knots that are not equivariantly slice, as pointed out in \cite{BI}. The main import of the present paper is thus to show that $\Vsu$ and $\Vsl$ can be used to study higher-genus examples, which were not previously accessible.


In \cite[Theorem 1.5]{DHM}, it was shown that the invariants of \cite[Remark 4.5]{DHM} satisfy certain inequalities in the presence of negative-definite equivariant cobordisms. In our context, this specializes to inequalities of $\Vsu$ and $\Vsl$ involving equivariant crossing changes. In Section~\ref{sec:7}, we consider several kinds of equivariant crossing changes. We prove:

\begin{theorem}\label{thm:1.10}
Let $K$ be strongly invertible knot. Let $K'$ be obtained from $K$ via an equivariant positive-to-negative crossing change (or an equivariant pair of such crossing changes). Then:
\begin{enumerate}
\item If the crossing change is of Type Ia, we have 
\[
\Vtu(K) \geq \Vtu(K') \quad \text{and} \quad \Vtl(K) \geq \Vtl(K').
\]
\item If the crossing change is of Type Ib, we have 
\[
\Vitu(K) \geq \Vitu(K') \quad \text{and} \quad \Vitl(K) \geq \Vitl(K').
\]
\item If we have an equivariant pair of crossing changes (Type II), we have both 
\[
\Vtu(K) \geq \Vtu(K') \quad \text{and} \quad \Vtl(K) \geq \Vtl(K')
\]
and
\[
\Vitu(K) \geq \Vitu(K') \quad \text{and} \quad \Vitl(K) \geq \Vitl(K').
\]
\end{enumerate}
\end{theorem}
\noindent
See Definition~\ref{def:eqcrossingchanges} for a definition of these terms.

A generalization of these ideas will be used to establish Theorem~\ref{thm:1.7}. 


\subsection{Relation to secondary invariants} \label{sec:1.5}
In \cite{JZstabilization}, Juh\'asz and Zemke construct several secondary invariants associated to a pair of slice surfaces $\Sigma$ and $\Sigma'$ for the same knot. These are shown to give lower bounds for various quantities such as the stabilization distance between $\Sigma$ and $\Sigma'$ (see below). Here, we focus on the invariant $V_0(\Sigma, \Sigma')$ of \cite[Section 4.5]{JZstabilization}. It is easy to show:

\begin{theorem}\label{thm:relativeV0}
Let $(K, \tau)$ be a strongly invertible knot in $S^3$. Let $W$ be any (smooth) homology ball with boundary $S^3$, and let $\tau_W$ be any extension of $\tau$ over $W$. If $\Sigma$ is any slice disk for $K$ in $W$, then
\[
\max\{\Vtl(K), \Vitl(K)\} \leq V_0(\Sigma, \tau_W(\Sigma)).
\]
\end{theorem}
\noindent
In \cite{JZstabilization}, $V_0(\Sigma, \Sigma')$ is defined for surfaces in $B^4$, but the extension to general integer homology balls is straightforward. The authors expect further connections with the results of \cite{JZstabilization}, which we intend to investigate in future work.

Let $W$ be a homology ball with $\partial W = S^3$, and let $\Sigma, \Sigma' \subseteq W$ be two slice surfaces for $K$. Recall that the \textit{stabilization distance} $\must(\Sigma, \Sigma')$ is defined to be the minimum of
\[
\max\{g(\Sigma_1), \ldots, g(\Sigma_n)\}
\]
over sequences of slice surfaces $\Sigma_i \subseteq W$ from $\Sigma$ to $\Sigma'$ such that consecutive surfaces are related by either a stabilization/destabilization or an isotopy rel $K$. (We take the 4-manifold $W$ as being implicit in the setup and suppress it from the notation.) In \cite[Theorem 1.1]{JZstabilization}, Juh\'asz and Zemke show that if $\Sigma, \Sigma' \subseteq W$ are two slice disks for the same knot, then
\[
V_0(\Sigma, \Sigma') \leq \left \lceil \dfrac{\must(\Sigma, \Sigma')}{2} \right \rceil.
\]

It follows from this that $\Vsu$ and $\Vsl$ can be used to construct pairs of disks with large stabilization distance. Applying Theorem~\ref{thm:relativeV0} and \cite[Theorem 1.1]{JZstabilization} to the examples $(K_n, \tau_n)$ of Section~\ref{sec:1.2}, we immediately obtain:


\begin{theorem}\label{thm:1.6}
Let $n$ be odd. Let $W$ be any (smooth) homology ball with boundary $S^3$, and let $\tau_W$ be any extension of $\tau_n$ over $W$. Suppose $K_n$ is slice in $W$. Then for any slice disk $\Sigma$,
\[
2n - 1 \leq \must(\Sigma, \tau_W(\Sigma)).
\]
\end{theorem}
\noindent
Since $K_n$ is slice in $B^4$, this shows that for any integer $m$, there is some knot with a pair of slice disks that require at least $m$ stabilizations to become isotopic. This provides an alternate proof of a result of Miller-Powell~\cite[Theorem B]{MP}. In fact, Theorem~\ref{thm:1.6} is slightly stronger, as the stabilization distance between two surfaces can be strictly less than the number of stabilizations needed to make them isotopic. Moreover, the examples of Theorem~\ref{thm:1.6} are inherent to the knots $K_n$, rather than the actual disks: we may start with \textit{any} slice disk for $K_n$ (in $B^4$ or otherwise) and compute its stabilization distance to its reflection (again, under \textit{any} extension of $\tau_n$). 

\begin{remark}
During the completion of this project, Ian Zemke informed us of another family of examples, now independently included in \cite[Section 10.3]{JZstabilization}. Our examples use a similar family of knots as in \cite[Section 10.3]{JZstabilization}, but the slice disks in question are rather different. (In particular, our approach de-emphasizes the construction of the actual disks, and instead requires only that the pair of disks are related by $\tau_W$.)
\end{remark}

\begin{remark}\label{rem:currentdevelopments}
Recently, several related results have emerged which have a strong bearing on the work presented here. Although these appeared some time after the initial posting of this paper, we describe them briefly to provide some context:
\begin{enumerate}
\item Di Prisa \cite{DiPrisa} has shown that the equivariant concordance group is indeed non-abelian. The authors do not believe that knot Floer homology detects these examples; it is still unclear whether $\K_{\tau, \iota}$ is abelian.
\item Building on the Floer-theoretic formalism of the present work (in particular, Theorem~\ref{thm:swapping}), the authors of this paper (in joint work with Kang and Park) have recently shown that the $(2,1)$-cable of the figure-eight is not slice \cite{DKMPS}. This was previously an open question, and as such may provide some motivation for the extensive framework we establish here.
\item Miller and Powell \cite{MPinvertible} have recently provided a second (more topological) proof of \cite[Question 1.1]{BI} by utilizing Blanchfield forms. Curiously, the two approaches do not overlap: the examples presented here are not accessible via the methods of \cite{MPinvertible}; conversely, Floer homology does not give growing genus bounds for the examples of \cite{MPinvertible}. 
\end{enumerate}
\end{remark}

\subsection*{Acknowledgements}
The authors would like to thank Kyle Hayden for bringing the example of Theorem~\ref{thm:1.7} to their attention, and would also like to thank Matthew Hedden and Ian Zemke for helpful conversations. The second author would like to thank the Max Planck Institute of Mathematics for their hospitality and support.

\subsection*{Organization}
In Section~\ref{sec:2}, we give a brief introduction to equivariant concordance and introduce the topological objects that we study in this paper. We then establish the algebraic formalism of local equivalence and define the local equivalence group $\K_{\tau, \iota}$. In Section~\ref{sec:3}, we construct the Floer-theoretic action associated to a strong inversion and prove Theorem~\ref{thm:1.8}. We then define $\Vsu$ and $\Vsl$ and prove Theorem~\ref{thm:1.1}. In Section~\ref{sec:4}, we establish several computational tools involving the action of $\tau$, including a connected sum formula. This leads to the proof of Theorem~\ref{thm:1.9}. We establish the equivariant slice genus bound of Theorem~\ref{thm:1.2} in Section~\ref{sec:5}. Then, in Section~\ref{sec:6}, we carry out the calculation of Theorem~\ref{thm:1.4} regarding the examples $(K_n, \tau_n)$. Finally, in Section~\ref{sec:7}, we explicitly discuss the relation between our invariants and the work of Dai-Hedden-Mallick \cite{DHM} and Juh\'asz-Zemke \cite{JZstabilization}; we prove Theorems~\ref{thm:1.7}, \ref{thm:knotfloermaps}, \ref{thm:1.10}, and \ref{thm:relativeV0}. Section~\ref{sec:8} gives an overview of several analogous results for $2$-periodic knots.


\section{Background and algebraic formalism}\label{sec:2}
 
In this section, we give a brief review of equivariant knots. We then establish the algebraic formalism of local equivalence and define and discuss the group $\K_{\tau, \iota}$. Throughout, we assume a general familiarity with the knot Floer and involutive knot Floer packages; see \cite{HM} and \cite{Zemkeconnected}.
 
\subsection{Equivariant knots} \label{sec:2.1}
Let $K$ be a knot in $S^3$ and let $\tau$ be an orientation-preserving involution on $S^3$ that sends $K$ to itself and has nonempty fixed-point set. By the resolution of the Smith Conjecture, we may assume that the fixed-point set of $\tau$ is an unknot and that $\tau$ is rotation about this axis \cite{Waldhausen}, \cite{MorganBass}. If $\tau$ preserves orientation on $K$, then we say that $(K, \tau)$ is \textit{$2$-periodic} (or often just \textit{periodic}) and refer to $\tau$ as a \textit{periodic involution}. If $\tau$ reverses orientation on $K$, then we say that $(K, \tau)$ is \textit{strongly invertible} and refer to $\tau$ as a \textit{strong inversion}. These correspond to the situations where $\tau$ has zero or two fixed points on $K$, respectively. In this paper we focus on strongly invertible knots; the periodic case is discussed in Section~\ref{sec:8}. Although an \textit{equivariant knot} may refer to either a strongly invertible or a periodic knot $(K, \tau)$, we will often use this terminology with the strongly invertible case in mind.

We say that $(K_1, \tau_1)$ and $(K_2, \tau_2)$ are \textit{equivariantly diffeomorphic} if there is an orientation-preserving diffeomorphism $f \colon S^3 \rightarrow S^3$ which sends $K_1$ to $K_2$ and has $\tau_2 \circ f = f \circ \tau_1$.


\begin{definition}\label{def:eqgenus}
Let $(K, \tau)$ be an equivariant knot. A slice surface $\Sigma \subseteq B^4$ for $K$ is \textit{equivariant} if there exists an involution $\tau_{B^4} \colon B^4 \rightarrow B^4$ which extends $\tau$ and has $\tau_{B^4}(\Sigma) = \Sigma$. The \textit{equivariant slice genus} of $(K, \tau)$ is defined to be the minimum genus over all equivariant slice surfaces for $(K, \tau)$ in $B^4$. We denote this quantity by $\eg(K)$, suppressing the involution $\tau$.
\end{definition}

Equivariant sliceness has been studied by several authors; see for example \cite{Sakuma}, \cite{ChaKo}, and \cite{ND}. Recently, Boyle-Issa~\cite{BI} studied the equivariant slice genus and were able to present several methods for bounding the equivariant slice genus from below. They moreover construct a family of periodic knots for which $\eg(K) - g_4(K)$ becomes arbitrarily large. Prior to the current article, there were no known examples of strongly invertible knots with $\eg(K) - g_4(K)$ provably greater than one. 

There is also an obvious notion of equivariant concordance, given by:

\begin{definition}\label{def:eqconc}
Let $(K_1, \tau_1)$ and $(K_2, \tau_2)$ be two equivariant knots. We say that a concordance $\Sigma \subseteq S^3 \times I$ from $K_1$ to $K_2$ is \textit{equivariant} if there exists an involution $\tau_{S^3 \times I} \colon S^3 \times I \rightarrow S^3 \times I$ which extends $\tau_1$ and $\tau_2$ and has $\tau_{S^3 \times I}(\Sigma) = \Sigma$.
\end{definition}

In the strongly invertible setting, it turns out that it is useful to have a refinement of Definition~\ref{def:eqgenus}, which we now explain. If $(K, \tau)$ is a strongly invertible knot, note that $K$ separates the fixed-point axis of $\tau$ into two halves. 

\begin{definition}\label{def:direction}
A \textit{direction} on $(K, \tau)$ is a choice of half-axis, together with an orientation on this half-axis. 
\end{definition}

\begin{definition}\label{def:eqconnectedsum}
Given two directed strongly invertible knots $(K_1, \tau_1)$ and $(K_2, \tau_2)$, we may form their \textit{equivariant connected sum}, as defined in \cite{Sakuma}. This is another directed strongly invertible knot, constructed as follows.  Place $K_1$ and $K_2$ along the same oriented axis, such that the oriented half-axis for $K_1$ occurs before the oriented half-axis for $K_2$. Attach an equivariant band with one foot at the head of the half-axis for $K_1$ and the other foot at the tail of the half-axis for $K_2$, as in Figure~\ref{fig:21}. Define the oriented half-axis for $K_1 \# K_2$ to run from the tail of the half-axis for $K_1$ to the head of the half-axis for $K_2$.
\end{definition}

\begin{figure}[h!]
\includegraphics[scale = 1]{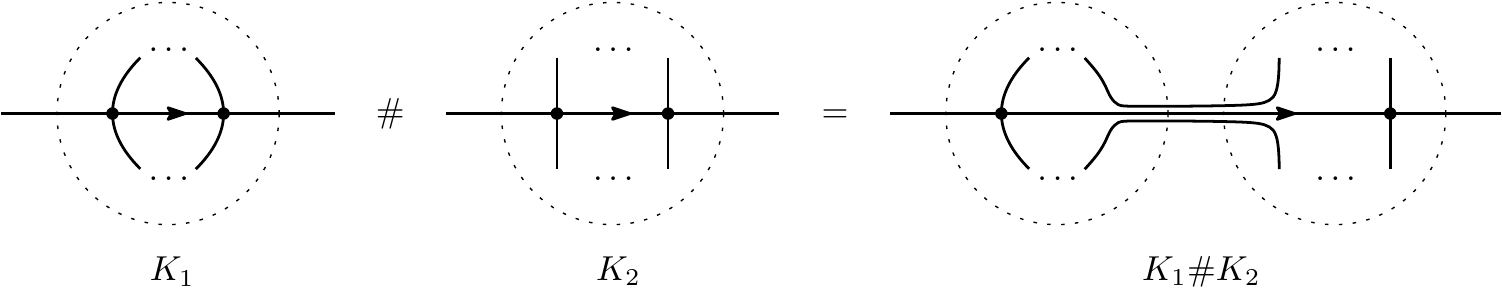}
\caption{Forming the equivariant connected sum. The connected sum band should be thought of as running along the axis of symmetry, and may have other strands of $K_1$ or $K_2$ running over/under it in the knot projection. Alternatively, the ends of the half-axes can be isotoped so that they are the leftmost and rightmost points of each knot.}\label{fig:21}
\end{figure}

We stress that a choice of direction is necessary to define the equivariant connected sum. Moreover, the equivariant connected sum is not a commutative operation: the strongly invertible knots $(K_1 \# K_2, \tau_1 \# \tau_2)$ and $(K_2 \# K_1, \tau_2 \# \tau_1)$ are not usually equivariantly diffeomorphic (even forgetting about the data of the direction). For further discussion, see \cite{Sakuma} or \cite[Section 2]{BI}.

In order to construct an equivariant concordance group, we consider the set of \textit{directed} strongly invertible knots. This necessitates a refinement of Definition~\ref{def:eqconc} which takes into account the extra data of the direction. 

\begin{definition}\cite[Definiton 2.4]{BI}\label{def:direqconc}
Let $(K_1, \tau_1)$ and $(K_2, \tau_2)$ be two directed strongly invertible knots. We say that an equivariant concordance $(\Sigma, \tau_{S^3 \times I})$ between $(K_1, \tau_1)$ and $(K_2, \tau_2)$ is \textit{equivariant in a directed sense} if the orientations of the half-axes induce the same orientation on the fixed-point annulus $F$ of $\tau_{S^3 \times I}$ and the half-axes are contained in the same component of $F - \Sigma$.
\end{definition}

\begin{definition}\cite{Sakuma}\label{def:siconcordancegroup}
The \textit{equivariant concordance group} is formed by quotienting out the set of directed strongly invertible knots by (directed) equivariant concordance. The group operation is given by equivariant connected sum. The inverse of $(K, \tau)$ is constructed by mirroring $K$ and reversing orientation on the (mirrored) half-axis. We denote this group by $\eC$. 
\end{definition}

It is currently an open question whether $\eC$ is abelian. The equivariant concordance group was studied at length by Sakuma~\cite{Sakuma}, who constructed a homomorphism from $\eC$ to the additive group $\Z[t]$ in the form of the $\eta$-polynomial.

\begin{remark}\label{rem:notoriented}
As discussed in \cite{Sakuma} and \cite{BI}, when studying $\eC$, it is often the convention \textit{not} to fix an orientation on $K$, due to the fact that the action of $\tau$ reverses orientation on $K$. In order to follow this convention, we will similarly not require a fixed orientation. However, since most invariants derived from knot Floer homology implicitly \textit{do} require $K$ to be oriented, in each case we will be careful to check whether the choice of orientation is important.
\end{remark}

Finally, we note (as discussed in the introduction) that we bound a rather more general notion than the equivariant slice genus. For completeness, we formally record:

\begin{definition}\label{def:isotopyeqgenus}
Let $(K, \tau)$ be an equivariant knot. Let $W$ be a (smooth) homology ball with boundary $S^3$, and consider any (smooth) self-diffeomorphism $\tau_W$ on $W$ which extends $\tau$. Note that we do not require $\tau_W$ itself to be an involution. We say that a slice surface $\Sigma$ in $W$ with $\partial \Sigma = K$ is an \textit{isotopy-equivariant slice surface} (for the given data) if $\tau_W(\Sigma)$ is isotopic to $\Sigma$ rel $K$. Define the \textit{isotopy-equivariant slice genus} of $(K, \tau)$ by:
\[
\ieg(K) = \min_{\substack{\text{all possible choices of $W$ and $\tau_W$} \\ \text{all isotopy-equivariant slice surfaces $\Sigma$}}} \{g(\Sigma)\}.
\]
Here $\ieg(K)$ depends on $\tau$, but we suppress this from the notation.
\end{definition}

There is also an accompanying notion in the setting of concordance:

\begin{definition}\label{def:isotopyeqconc}
Let $(K_1, \tau_1)$ and $(K_2, \tau_2)$ be two equivariant knots. An \textit{isotopy-equivariant homology concordance} between $(K_1, \tau_1)$ and $(K_2, \tau_2)$ consists of a homology cobordism $W$ from $S^3$ to itself, a (smooth) self-diffemorphism $\tau_W \colon W \rightarrow W$ which extends $\tau_1$ and $\tau_2$, and a concordance $\Sigma \subseteq W$ between $K_1$ and $K_2$ such that $\tau_W(\Sigma)$ is isotopic to $\Sigma$ rel boundary.
\end{definition}




\subsection{Local equivalence and $\K_{\tau, \iota}$}\label{sec:2.2}
We now give an overview of the framework of local equivalence and define $\K_{\tau, \iota}$. We assume the reader has a general familiarity with the ideas of \cite{HM} and \cite{Zemkeconnected}. Let $C$ be a bigraded, free, finitely-generated chain complex over $\cR = \F[\cU, \cV]$ such that 
\begin{enumerate}
\item $\gr(\partial) = (-1, -1)$, $\gr(\cU) = (-2, 0)$, and $\gr(\cV) = (0, -2)$ 
\item $C \otimes \F[\cU, \cV, \cU^{-1}, \cV^{-1}]$ is homotopy equivalent to $\F[\cU, \cV, \cU^{-1}, \cV^{-1}]$
\end{enumerate}
We refer to $C$ as an \textit{abstract knot complex}, occasionally denoting the two components of the grading by $\grU$ and $\grV$. As explained in \cite[Section 3]{Zemkequasistab}, given an abstract knot complex, we can formally differentiate the matrix of $\partial$ with respect to $\cU$ and $\cV$ to obtain $\mathcal{R}$-equivariant maps $\Phi$ and $\Psi$.\footnote{Technically, $\Phi$ and $\Psi$ are only defined after fixing a basis for $C$. However, the homotopy equivalence classes of $\Phi$ and $\Psi$ are well-defined without a choice of basis; see \cite[Corollary 2.9]{Zemkeconnected}.} We define the \textit{Sarkar map} in this context to be $\varsigma_K = \id + \Phi \circ \Psi \simeq \id + \Psi \circ \Phi$. It is a standard fact that $\varsigma_K^2 \simeq \id$.

\begin{definition}\label{def:ticomplex}
A \textit{abstract $(\tau_K, \iota_K)$-complex} is a triple $(C, \tau_K, \iota_K)$ such that:
\begin{enumerate}
\item $C$ is an abstract knot complex
\item $\iota_K \colon C \rightarrow C$ is a skew-graded, $\mathcal{R}$-skew-equivariant chain map such that
\[
\iota_K^2 \simeq \varsigma_K
\]
\item $\tau_K$ is a skew-graded, $\mathcal{R}$-skew-equivariant chain map such that
\[
\tau_K^2 \simeq \id \quad \text{and} \quad \tau_K \circ \iota_K \simeq \varsigma_K \circ \iota_K \circ \tau_K.
\]
\end{enumerate}
\end{definition}
\noindent
Recall that a map $f \colon C \rightarrow C$ is skew-graded if $\gr(f(x)) = (\gr_V(x), \gr_U(x))$ and is $\mathcal{R}$-skew-equivariant if $f(\mathcal{U}^i \mathcal{V}^j x ) = \mathcal{V}^i \mathcal{U}^j x$.

Definition~\ref{def:ticomplex} simply says that the pair $(C, \iota_K)$ is an abstract $\iota_K$-complex in the sense of Zemke \cite[Definition 2.2]{Zemkeconnected}. The conditions on $\tau_K$ are from Theorem~\ref{thm:1.8}. We also note the following extremely important consequence of the commutation condition:

\begin{lemma}\label{lem:sarkarcommutes}
Let $(C, \tau_K, \iota_K)$ be an abstract $(\tau_K, \iota_K)$-complex. Then $\varsigma_K$ commutes with both $\tau_K$ and $\iota_K$ up to homotopy.
\end{lemma}
\begin{proof}
The proof is immediate from the commutation relation between $\tau_K$ and $\iota_K$, the fact that $\iota_K^2 \simeq \varsigma_K$, and the fact that $\varsigma_K^2 \simeq \id$. (In fact, it is possible to show that $\varsigma_K$ commutes with any chain map from $C$ to itself, using the equality $\varsigma_K = \id + \Phi \circ \Psi$.)
\end{proof}

There is a natural notion of homotopy equivalence:

\begin{definition}\label{def:homotopyequivalence}
Two $(\tau_K, \iota_K)$-complexes $(C_1, \tau_{K_1}, \iota_{K_1})$ and $(C_2, \tau_{K_2}, \iota_{K_2})$ are \textit{homotopy equivalent} if there exist graded, $\cR$-equivariant homotopy inverses $f$ and $g$ between $C_1$ and $C_2$ such that
\[
f \circ \tau_{K_1} \simeq \tau_{K_2} \circ f \text{ and } g \circ \tau_{K_2} \simeq \tau_{K_1} \circ g
\]
and 
\[
f \circ \iota_{K_1} \simeq \iota_{K_2} \circ f \text{ and } g \circ \iota_{K_2} \simeq \iota_{K_1} \circ g.
\]
In this case we write $(C_1, \tau_{K_1}, \iota_{K_1}) \simeq (C_2, \tau_{K_2}, \iota_{K_2})$.\footnote{If $f$ and $g$ are graded, $\cR$-equivariant chain maps, we write $f \simeq g$ to mean that $f$ and $g$ are homotopic via an $\cR$-equivariant homotopy. This means that $f + g = \partial H + H \partial$ for some $\cR$-equivariant $H$. If $f$ and $g$ are skew-graded, skew-$\cR$-equivariant chain maps, we again write $f \simeq g$ to mean that $f$ and $g$ are homotopic via a skew-$\cR$-equivariant homotopy. This means that $f + g = \partial H + H \partial$ for some $\cR$-skew-equivariant $H$. Our notation differs slightly from \cite{Zemkeconnected}, where the convention is to write $\mathrel{\rotatebox[origin=c]{-180}{$\simeq$}}$ in the latter case.}
\end{definition}

We also have the analogue of local equivalence from \cite[Definition 2.4]{Zemkeconnected}:

\begin{definition}\label{def:localequivalence}
Two $(\tau_K, \iota_K)$-complexes $(C_1, \tau_{K_1}, \iota_{K_1})$ and $(C_2, \tau_{K_2}, \iota_{K_2})$ are \textit{locally equivalent} if there exist graded, $\cR$-equivariant chain maps $f \colon C_1 \rightarrow C_2$ and $g \colon C_2 \rightarrow C_1$ such that
\[
f \circ \tau_{K_1} \simeq \tau_{K_2} \circ f \text{ and } g \circ \tau_{K_2} \simeq \tau_{K_1} \circ g
\]
and 
\[
f \circ \iota_{K_1} \simeq \iota_{K_2} \circ f \text{ and } g \circ \iota_{K_2} \simeq \iota_{K_1} \circ g.
\]
and $f$ and $g$ induce homotopy equivalences $C_1 \otimes \F[\cU, \cV, \cU^{-1}, \cV^{-1}] \simeq C_2 \otimes \F[\cU, \cV, \cU^{-1}, \cV^{-1}]$. In this case we write $(C_1, \tau_{K_1}, \iota_{K_1}) \sim (C_2, \tau_{K_2}, \iota_{K_2})$. We refer to the maps $f$ and $g$ as \textit{local maps}. 
\end{definition}

Using the notion of local equivalence, we now define:

\begin{definition}\label{def:tigroup}
We define the \textit{$(\tau_K, \iota_K)$-local equivalence group} to be
\[
\K_{\tau, \iota} = \{\text{abstract }(\tau_K, \iota_K)\text{-complexes}\}\ / \ \text{local equivalence}.
\]
The group operation is defined as follows. Given $(C_1, \tau_{K_1}, \iota_{K_1})$ and $(C_2, \tau_{K_2}, \iota_{K_2})$, define automorphisms $\tau_\otimes$ and $\iota_\otimes$ on $C_1 \otimes C_2$ as follows. Let
\[
\tau_\otimes = \tau_{K_1} \otimes \tau_{K_2}
\]
and
\[
\iota_{\otimes} = (\id \otimes \id + \Phi \otimes \Psi) \circ (\iota_{K_1} \otimes \iota_{K_2}).
\]
We define the \textit{product} of two abstract $(\tau_K, \iota_K)$-complexes $(C_1, \tau_{K_1}, \iota_{K_1})$ and $(C_2, \tau_{K_2}, \iota_{K_2})$ to be $(C_1 \otimes C_2, \tau_\otimes, \iota_\otimes)$. This operation gives another abstract $(\tau_K, \iota_K)$-complex and is well-defined with respect to local equivalence. The identity is given by the trivial complex $(\cR, \tau_e, \iota_e)$, where $\tau_e = \iota_e$ is the map on $\cR$ which interchanges $\cU$ and $\cV$. Inverses are given by dualizing with respect to $\cR$; that is, $(C, \tau_K, \iota_K)^\vee = (C^\vee, \tau_K^\vee, \iota_K^\vee)$. See Lemmas~\ref{lem:gpA} and \ref{lem:gpB}.
\end{definition}

It will also be convenient for us to consider the Sarkar map $\varsigma_\otimes$ on the product of two complexes. Note that 
\begin{align*}
\varsigma_\otimes &\simeq \id_\otimes + \Phi_\otimes \Psi_\otimes \\
&\simeq \id \otimes \id + (\id \otimes \Phi + \Phi \otimes \id)(\id \otimes \Psi + \Psi \otimes \id) \\
&\simeq \id \otimes \id + \id \otimes \Phi \Psi + \Phi \Psi \otimes \id + \Phi \otimes \Psi + \Psi \otimes \Phi.
\end{align*}


\begin{lemma}\label{lem:gpA}
The tensor product induces an associative binary operation on $\K_{\tau, \iota}$. 
\end{lemma}
\begin{proof}
We will be brief, since the majority of the claim is immediate from \cite[Section 2.3]{Zemkeconnected}. We first verify that the product complex is an abstract $(\tau_K, \iota_K)$-complex. The only condition which is not either obvious or contained in \cite{Zemkeconnected} is the commutation relation
\[
\tau_\otimes \circ \iota_\otimes \simeq \varsigma_\otimes \circ \iota_\otimes \circ \tau_\otimes.
\]
To see this, let us expand the left-hand side. Supressing the subscripts on $\tau$ and $\iota$, we obtain
\begin{align*}
&\ \ \ \ (\tau \otimes \tau)(\id \otimes \id + \Phi \otimes \Psi)(\iota \otimes \iota) \\
&\simeq (\id \otimes \id + \Psi \otimes \Phi)(\tau \otimes \tau)(\iota \otimes \iota) \\
&\simeq (\id \otimes \id + \Psi \otimes \Phi)(\varsigma \iota \tau \otimes \varsigma \iota \tau) \\
&\simeq (\id \otimes \id + \Psi \otimes \Phi) (\id \otimes \id + \id \otimes \Phi \Psi + \Phi \Psi \otimes \id + \Phi \Psi \otimes \Phi \Psi)(\iota \otimes \iota)(\tau \otimes \tau) \\
&\simeq (\id \otimes \id + \id \otimes \Phi \Psi + \Phi \Psi \otimes \id + \Phi \Psi \otimes \Phi \Psi + \Psi \otimes \Phi)(\iota \otimes \iota)(\tau \otimes \tau).
\end{align*}
Here, in the second line we have used \cite[Lemma 2.8]{Zemkeconnected}, which states that a skew-equivariant map intertwines $\Phi$ and $\Psi$ up to homotopy. In the third line, we have used the commutation property of $\tau$ and $\iota$ in each factor. In the fourth line, we have used the fact that $\varsigma = \id + \Phi \Psi$. Finally, in the last line we have used the fact that $\Phi$ and $\Psi$ homotopy commute and that $\Phi^2 \simeq \Psi^2 \simeq 0$; see \cite[Lemma 2.10]{Zemkeconnected} and \cite[Lemma 2.11]{Zemkeconnected}.

On the other hand, the right-hand side is given by
\begin{align*}
&\ \ \ \ (\id_\otimes + \Phi_\otimes \Psi_\otimes)(\id \otimes \id + \Phi \otimes \Psi)(\iota \otimes \iota)(\tau \otimes \tau) \\
&\simeq (\id \otimes \id + \id \otimes \Phi \Psi + \Phi \Psi \otimes \id + \Phi \otimes \Psi + \Psi \otimes \Phi)(\id \otimes \id + \Phi \otimes \Psi)(\iota \otimes \iota)(\tau \otimes \tau) \\
&\simeq (\id \otimes \id + \id \otimes \Phi \Psi + \Phi \Psi \otimes \id + \Psi \Phi \otimes \Phi \Psi + \Psi \otimes \Phi)(\iota \otimes \iota)(\tau \otimes \tau).
\end{align*}
Here, in the second line we have used the fact that $\Phi_\otimes = \id \otimes \Phi + \Phi \otimes \id$ (and similarly for $\Psi_\otimes$). In the fourth line, we again use the fact that $\Phi$ and $\Psi$ homotopy commute and that $\Phi^2 \simeq \Psi^2 \simeq 0$. The resulting expression is homotopic to the previous. 

Checking associativity is straightforward. Indeed, in \cite[Section 2.3]{Zemkeconnected} it is shown that the obvious identity map from $(C_1 \otimes C_2) \otimes C_3$ to $C_1 \otimes (C_2 \otimes C_3)$ intertwines the $\iota_K$-actions up to homotopy; this clearly intertwines the $\tau_K$-actions. Checking that the tensor product respects local equivalence is likewise immediate.
\end{proof}

\begin{lemma}\label{lem:gpB}
The tensor product operation above gives $\K_{\tau, \iota}$ the structure of a group.
\end{lemma}
\begin{proof}
Again, the majority of the claim is immediate from \cite[Section 2.3]{Zemkeconnected}. The only nontrivial claim is to establish that inverses are given by dualizing, which follows the proof of \cite[Lemma 2.18]{Zemkeconnected}. Zemke shows that the cotrace and trace maps
\[
F \colon \mathcal{R} \rightarrow C \otimes C^\vee \quad \text{and} \quad G \colon C \otimes C^\vee \rightarrow \mathcal{R}
\]
have the requisite behavior with respect to localizing and intertwine the $\iota_K$-actions. We check that $F$ intertwines the actions of $\tau_e$ and $\tau_K \otimes \tau_K^\vee$. Since $\tau_e$ squares to the identity, it suffices to show
\[
(\tau_K \otimes \tau_K^\vee) \circ F \circ \tau_e \simeq F.
\]
But \cite[Equation (14)]{Zemkeconnected} implies
\[
(\tau_K \otimes \tau_K^\vee) \circ F \circ \tau_e \simeq (\tau_K^2 \otimes \id) \circ F \simeq (\id \otimes \id) \circ F = F,
\]
establishing the claim. The proof for $G$ is similar. Hence $(C, \tau_K, \iota_K) \otimes (C, \tau_K, \iota_K)^\vee$ is locally equivalent to the trivial $(\tau_K, \iota_K)$-complex. Checking triviality of the product in the opposite order is similar; use the obvious maps
\[
F^r \colon \mathcal{R} \rightarrow C^\vee \otimes C \quad \text{and} \quad G^r \colon C^\vee \otimes C \rightarrow \mathcal{R}.
\]
This completes the proof.
\end{proof}

Note that instead of considering triples $(C, \tau_K, \iota_K)$, we may forget $\iota_K$ or $\tau_K$ and consider pairs $(C, \tau_K)$ or $(C, \iota_K)$, respectively. The reader will have no trouble in defining appropriate notions of local equivalence for these and forming the analogous local equivalence groups. 

\begin{definition}
We denote the local equivalence group of $\tau_K$-complexes by $\K_\tau$; this consists of pairs $(C, \tau_K)$ such that $\tau_K$ is a skew-graded, $\cR$-equivariant chain map with $\tau_K^2 \simeq \id$. We denote the local equivalence group of $\iota_K$-complexes by $\K_\iota$; this consists of pairs $(C, \iota_K)$ such that $\iota_K$ is a skew-graded, $\cR$-equivariant chain map with $\iota_K^2 \simeq \varsigma$. The latter is just the usual local equivalence group of \cite[Proposition 2.6]{Zemkeconnected}. We obtain forgetful maps from $\K_{\tau, \iota}$ to $\K_\tau$ and $\K_\iota$ by discarding $\iota_K$ and $\tau_K$, respectively.
\end{definition}

\begin{remark}
It is possible to have triples $(C_1, \tau_{K_1}, \iota_{K_1})$ and $(C_2, \tau_{K_2}, \iota_{K_2})$ such that $(C_1, \tau_{K_1}) \sim (C_2, \tau_{K_2})$ and $(C_1, \iota_{K_1}) \sim (C_2, \iota_{K_2})$, but still $(C_1, \tau_{K_1}, \iota_{K_1}) \not\sim (C_2, \tau_{K_2}, \iota_{K_2})$. This is because in Definition~\ref{def:localequivalence}, we require $\tau_{K_i}$ and $\iota_{K_i}$ to be \textit{simultaneously} intertwined by $f$ and $g$. This will be an important distinction which leads to a great deal of (possible) subtlety in the structure of $\K_{\tau, \iota}$. For an explicit example of the above phenomenon, see Example~\ref{ex:stevedore}. Compare \cite[Example 4.7]{DHM}, which establishes a similar phenomenon in the 3-manifold setting.
\end{remark}

\subsection{(Possible) non-commutativity of $\K_{\tau, \iota}$} \label{sec:2.3}
We now discuss some subtleties of $\K_{\tau, \iota}$. The first of these involves a seeming asymmetry in the product operation. Recall that we defined the product $\iota_K$-action on $C_1 \otimes C_2$ to be $\iota_\otimes = \iota_{\otimes, A}$, where
\[
\iota_{\otimes, A} = (\id \otimes \id + \Phi \otimes \Psi) \circ (\iota_{K_1} \otimes \iota_{K_2}).
\]
There is of course a slightly different $\iota_K$-action on $C_1 \otimes C_2$, given by 
\[
\iota_{\otimes, B} = (\id \otimes \id + \Psi \otimes \Phi) \circ (\iota_{K_1} \otimes \iota_{K_2}).
\]
It is straightforward to check that using $\iota_{\otimes, B}$ in Definition~\ref{def:tigroup} also gives a well-defined operation on $\K_{\tau, \iota}$. Rather surprisingly, it turns out that these operations are not \textit{a priori} the same.

\begin{remark}\label{rem:notabelian}
In \cite[Lemma 2.14]{Zemkeconnected}, Zemke considers the map
\[
F = \id \otimes \id + \Psi \otimes \Phi.
\]
This is a homotopy equivalence from $C_1 \otimes C_2$ to itself such that $F \circ \iota_{\otimes, A} \simeq \iota_{\otimes, B} \circ F$. Hence $F$ mediates a homotopy equivalence of pairs $(C_1 \otimes C_2, \iota_{\otimes, A}) \simeq (C_1 \otimes C_2, \iota_{\otimes, B})$. For this reason, $\iota_{\otimes, A}$ and $\iota_{\otimes, B}$ both give the same product structure on $\K_\iota$. However, the map $F$ above does \textit{not} provide a homotopy equivalence between the triples $(C_1 \otimes C_2, \tau_{K_1} \otimes \tau_{K_2}, \iota_{\otimes, A})$ and $(C_1 \otimes C_2, \tau_{K_1} \otimes \tau_{K_2}, \iota_{\otimes, B})$. Indeed, $F$ is not $\tau_{K_1} \otimes \tau_{K_2}$-equivariant. We have
\[
F \circ (\tau_{K_1} \otimes \tau_{K_2}) = (\id \otimes \id + \Psi \otimes \Phi)(\tau_{K_1} \otimes \tau_{K_2})
\]
while
\[
(\tau_{K_1} \otimes \tau_{K_2}) \circ F = (\tau_{K_1} \otimes \tau_{K_2})(\id \otimes \id + \Psi \otimes \Phi) \simeq (\id \otimes \id + \Phi \otimes \Psi)(\tau_{K_1} \otimes \tau_{K_2}).
\]
In general, it is not true that these are chain homotopic maps. 
\end{remark}

This discrepancy is closely related to the possible non-commutativity of $\K_{\tau, \iota}$. Indeed, consider the two products $C_1 \otimes C_2$ and $C_2 \otimes C_1$. There is an obvious isomorphism from $C_1 \otimes C_2$ to $C_2 \otimes C_1$ given by transposition of factors; this clearly intertwines the two $\tau_K$-actions $\tau_{K_1} \otimes \tau_{K_2}$ and $\tau_{K_2} \otimes \tau_{K_1}$. However, it does not intertwine the $\iota_K$-actions: instead, it sends $\iota_{\otimes, A}$ on $C_1 \otimes C_2$ to $\iota_{\otimes, B}$ on $C_2 \otimes C_1$. Hence $\K_{\tau, \iota}$ is not necessarily abelian, and in fact this question is equivalent to whether the operations on $\K_{\tau, \iota}$ induced by $\iota_{\otimes, A}$ and $\iota_{\otimes, B}$ are the same (up to local equivalence). In Theorem~\ref{thm:connectedsum}, we establish a connected sum formula showing that using $\iota_{\otimes, A}$ corresponds to taking the equivariant connected sum as in Definition~\ref{def:eqconnectedsum}. Using $\iota_{\otimes, B}$ thus corresponds to modifying Definition~\ref{def:eqconnectedsum} by placing the half-axis of the first knot above the half-axis of the second.

Unfortunately, the authors do not have an explicit example demonstrating that $\K_{\tau, \iota}$ is not abelian. Indeed, in all of the examples that the authors have tried, it is possible to find an \textit{ad hoc} construction of a local equivalence (in fact, even a homotopy equivalence) between $(C_1 \otimes C_2, \tau_{K_1} \otimes \tau_{K_2}, \iota_{\otimes, A})$ and $(C_1 \otimes C_2, \tau_{K_1} \otimes \tau_{K_2}, \iota_{\otimes, B})$. Note that $\K_{\tau, \iota}$ admits forgetful maps to both $\K_\tau$ and $\K_\iota$, which are both abelian.




\subsection{Twisting by $\varsigma_K$}\label{sec:2.4}

As discussed in the introduction, our goal will be to associate to a strongly invertible knot a $(\tau_K, \iota_K)$-complex which is well-defined up to homotopy equivalence. Moreover, we wish to show that this local equivalence class is invariant under isotopy-equivariant homology concordance. Unfortunately, it turns out that both of these statements are technically only true if we pass to the \textit{decorated} category (see Sections~\ref{sec:3.3} and \ref{sec:3.4}). In order to capture this subtlety, we introduce the following notion:

\begin{definition}\label{def:twist}
Let $(C, \tau_K, \iota_K)$ be an abstract $(\tau_K, \iota_K)$-complex. Compose $\tau_K$, $\iota_K$, or both with any number of copies of $\varsigma_K$. By Lemma~\ref{lem:sarkarcommutes}, this produces another $(\tau_K, \iota_K)$-complex. We refer to this new complex as being obtained from $(C, \tau_K, \iota_K)$ via a \textit{twist by $\varsigma_K$}.
\end{definition}


\begin{lemma}\label{lem:onetwist}
Let $(C, \tau_K, \iota_K)$ be an abstract $(\tau_K, \iota_K)$-complex. Then
\[
(C, \varsigma_K \circ \tau_K, \iota_K) \simeq (C, \tau_K, \varsigma_K \circ \iota_K)
\]
and
\[
(C, \tau_K, \iota_K) \simeq (C, \varsigma_K \circ \tau_K, \varsigma_K \circ \iota_K).
\]
\end{lemma}
\begin{proof}
To prove the second claim, we use the graded, $\cR$-equivariant automorphism $\tau_K \circ \iota_K$. A quick computation using the relation $\tau_K \circ \iota_K \simeq \varsigma_K \circ \iota_K \circ \tau_K$ shows that this constitutes a homotopy equivalence between $(C, \tau_K, \iota_K)$ and $(C, \varsigma_K \circ \tau_K, \varsigma_K \circ \iota_K)$. The first claim follows immediately from the second, noting that $\varsigma_K^2 \simeq \id$.
\end{proof}

Up to homotopy equivalence, a $(\tau_K, \iota_K)$-complex thus has only one twist, which is represented by $(C, \varsigma_K \circ \tau_K, \iota_K) \simeq (C, \tau_K, \varsigma_K \circ \iota_K)$. In general, the authors know of no reason this should be homotopy (or even locally) equivalent to its original, and it is possible that the requisite computation of $\iota_K$ does not currently exist in the literature. Note that Lemma~\ref{lem:onetwist} also implies $(C, \tau_K) \simeq (C, \varsigma_K \circ \tau_K)$ and $(C, \iota_K) \simeq (C, \varsigma_K \circ \iota_K)$ as pairs. Hence the distinction between a complex and its twist is a phenomenon that is only present when considering $\tau_K$ and $\iota_K$ simultaneously.

As we will see in Section~\ref{sec:3.3} and Section~\ref{sec:3.4}, twisted complexes will play an important role when we move from the decorated to the undecorated setting. Roughly speaking, if $(K, \tau)$ is an strongly invertible knot without a choice of decoration, then we can only define the homotopy equivalence class of $(C, \tau_K, \iota_K)$ up to a twist by $\varsigma_K$. Nevertheless, we show presently that the numerical invariants $\Vsu$ and $\Vsl$ are unchanged by twisting.

The distinction between a complex and its twist also has an interpretation in terms of the direction on a strongly invertible knot (see Definition~\ref{def:direction}). In Section~\ref{sec:3.5}, we show that a choice of direction on $(K, \tau)$ can be used to determine a homotopy equivalence class of $(\tau_K, \iota_K)$-complex. Reversing the direction on $(K, \tau)$ corresponds to applying a twist by $\varsigma_K$. In general, reversing the direction on $(K, \tau)$ alters its class in $\eC$. Boyle-Issa \cite{BI} and Alfieri-Boyle \cite{AB} show that several invariants are sensitive to this operation; $\K_{\tau, \iota}$ is thus (in principle) similar, although this fails for the simple examples at our disposal.


\subsection{Extracting numerical invariants}\label{sec:2.5}
We now give a brief review of extracting numerical invariants from the local equivalence class of $(C, \tau_K, \iota_K)$. Recall that given an abstract knot complex $C$, we may form the \textit{large surgery subcomplex}, which we denote by $\Co$. 

\begin{definition}\label{def:largesurgery}
Let $(C, \tau_K, \iota_K)$ be a $(\tau_K, \iota_K)$-complex. The \textit{large surgery subcomplex of $C$} is the subset $\Co$ of $C$ lying in Alexander grading zero; that is, the set of elements $x$ with $\grU(x) = \grV(x)$. (This is often denoted by by $A_0$ elsewhere in the literature.) Strictly speaking, this is not a subcomplex of $C$; although $\Co$ is preserved by $\partial$, it is not a submodule over $\cR$. Instead, we view $\Co$ as a singly-graded complex over the ring $\F[U]$, where
\[
U = \cU \cV.
\]
The Maslov grading of an element is given by $\grU = \grV$. When we write $\Co$, we will mean this singly-graded complex over $\F[U]$.

Note that although $\tau_K$ and $\iota_K$ are skew-graded, the condition $\grU = \grV$ means that $\tau_K$ and $\iota_K$ induce grading-preserving automorphisms on $\Co$, which we also denote by $\tau_K$ and $\iota_K$. Moreover, although $\tau_K$ and $\iota_K$ are $\cR$-skew-equivariant, their actions on $\Co$ are equivariant with respect to $U = \cU \cV$. It follows from \cite[Lemma 3.16]{HHSZ} that as an automorphism of $\Co$, the Sarkar map $\varsigma_K$ is homotopic to the identity. It is then easily checked that
\[
(\Co, \tau_K) \quad \text{and} \quad (\Co, \iota_K \circ \tau_K)
\]
are $\iota$-complexes in the sense of \cite[Definition 8.1]{HMZ}. 
\end{definition}

We now follow the construction of the involutive numerical invariants $\du$ and $\dl$ from \cite{HM}, except that we replace the Heegaard Floer $\iota$-action with the action of $\tau_K$ on $\Co$, where $\Co$ is viewed as a singly-graded complex over $\F[U]$. Explicitly, let
\[
\CFI^{\tau}(\Co) = \text{Cone}\left(\Co \xrightarrow{Q(1 + \tau_K)} Q \cdot \Co\right)
\]
where $Q$ is a formal variable of grading $-1$. Define
\begin{align*}
\dl_\tau(\Co) = \max\{r \ \vert \ &\exists x \in \CFI^\tau_r(\Co) \text{ such that } \\ 
&U^n x \neq 0 \text{ and } U^n x \notin \im(Q) \text{ for all } n\}
\end{align*}
and
\begin{align*}
\du_\tau(\Co) = \max\{r \ \vert \ &\exists x \in \CFI^\tau_r(\Co) \text{ such that } \\
&U^n x \neq 0 \text{ for all } n \text{ and } U^m x \in \im(Q) \text{ for some } m\} + 1.
\end{align*}
We define the mapping cone $\CFI^{\ita}(\Co)$ by replacing $\tau_K$ with $\iota_K \circ \tau_K$, and define the numerical invariants $\dl_{\ita}(\Co)$ and $\du_{\ita}(\Co)$ similarly. Our conventions here are such that if $C$ is the trivial complex $\cR$, then $\dl_\circ = \du_\circ = 0$. We now have:

\begin{definition}\label{def:numericalinvariants}
Let $(C, \tau_K, \iota_K)$ be an abstract $(\tau_K, \iota_K)$-complex. Define
\[
\Vtu(C) = -\frac{1}{2} \du_\tau(\Co) \text{ and } \Vtl(C) = -\frac{1}{2} \dl_\tau(\Co)
\]
and
\[
\Vitu(C) = -\frac{1}{2} \du_{\ita}(\Co) \text{ and } \Vitl(C) = -\frac{1}{2} \dl_{\ita}(\Co).
\]
\end{definition}

\begin{lemma}
The invariants $\Vsu(C)$ and $\Vsl(C)$ are local equivalence invariants; that is, they factor through $\K_{\tau, \iota}$.
\end{lemma}
\begin{proof}
Let $(C_1, \tau_{K_1}, \iota_{K_1})$ and $(C_2, \tau_{K_2}, \iota_{K_2})$ be two $(\tau_K, \iota_K)$-complexes and let $f$ and $g$ be local equivalences between them. Since $f$ and $g$ are graded and $\cR$-equivariant, they induce graded, $\F[U]$-equivariant chain maps between $(C_1)_0$ and $(C_2)_0$, which are easily checked to be local in the sense of \cite[Definition 8.5]{HMZ}. The claim follows. 
\end{proof}


\begin{lemma}\label{lem:twistnochange}
The invariants $\Vsu(C)$ and $\Vsl(C)$ are insensitive to twisting by $\varsigma_K$.
\end{lemma}
\begin{proof}
This follows immediately from the fact that $\varsigma_K$ is homotopic to the identity as a map on $\Co$; see the proof of \cite[Lemma 3.16]{HHSZ}.
\end{proof}
\noindent
Note that the same argument indicates that no additional numerical invariants can be defined by considering (for example) $\tau_K \circ \iota_K$ in place of $\iota_K \circ \tau_K$. 

\subsection{Examples}
We now list the $(\tau_K, \iota_K)$-complexes corresponding to different strong inversions on the figure-eight and the stevedore. These may be derived from the results of \cite[Section 4.2]{DHM} in the following manner. Fix a basis in which the action of $\iota_K$ is standard, as in \cite[Section 8]{HM}. For the pairs $(K, \tau)$ at hand, the 3-manifold action of $\tau$ on $\HFm(S^3_{+1}(K))$ was calculated in \cite[Section 4.2]{DHM}. By the discussion of Section~\ref{sec:7.1}, this determines the action of $\tau_K$ on the homology of $\CFK(K)_0$. We then list all automorphisms $\tau_K$ of $\CFK(K)$ which induce this action and satisfy the axioms of Definition~\ref{def:ticomplex}. (In particular, note that $\tau_K$ is required to satisfy $\tau_K \circ \iota_K \simeq \varsigma_K \circ \iota_K \circ \tau_K$.) It turns out that in each example, the resulting automorphism is unique up to $\iota_K$-equivariant basis change. The proof of this is left to the reader and is an exercise in tedium. Compare \cite[Section 4.2]{DHM}.

\begin{example}\label{ex:figureeight}
There are two strong inversions $\tau$ and $\sigma$ on the figure-eight, which are displayed in Figure~\ref{fig:22}. In Figure~\ref{fig:23}, we display their corresponding actions $\tau_K$ and $\sigma_K$ on $C = \CFK(4_1)$ (calculated in the basis with the indicated action of $\iota_K$). See \cite[Example 4.6]{DHM}.
\end{example}

\begin{figure}[h!]
\includegraphics[scale = 0.65]{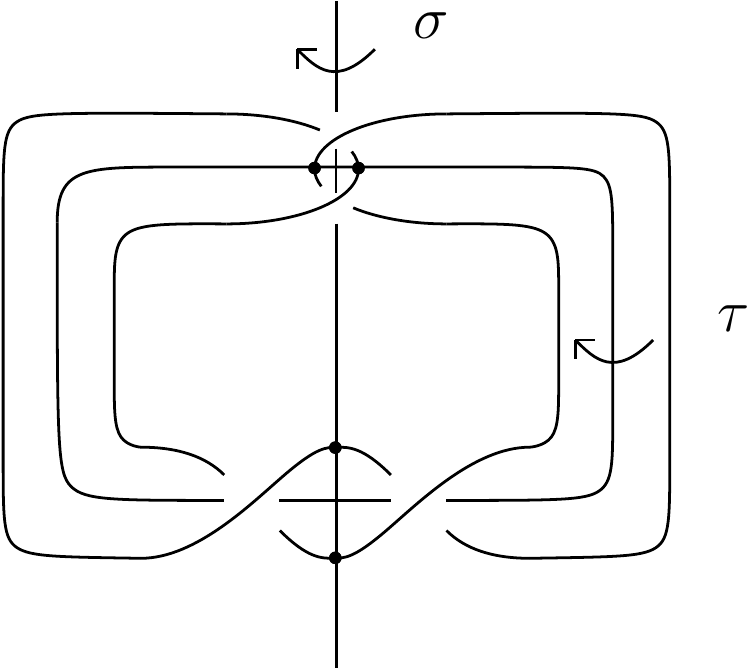}
\caption{The figure-eight $4_1$ with two strong inversions $\tau$ and $\sigma$.}\label{fig:22}
\end{figure}

\begin{figure}[h!]
\includegraphics[scale = 0.95]{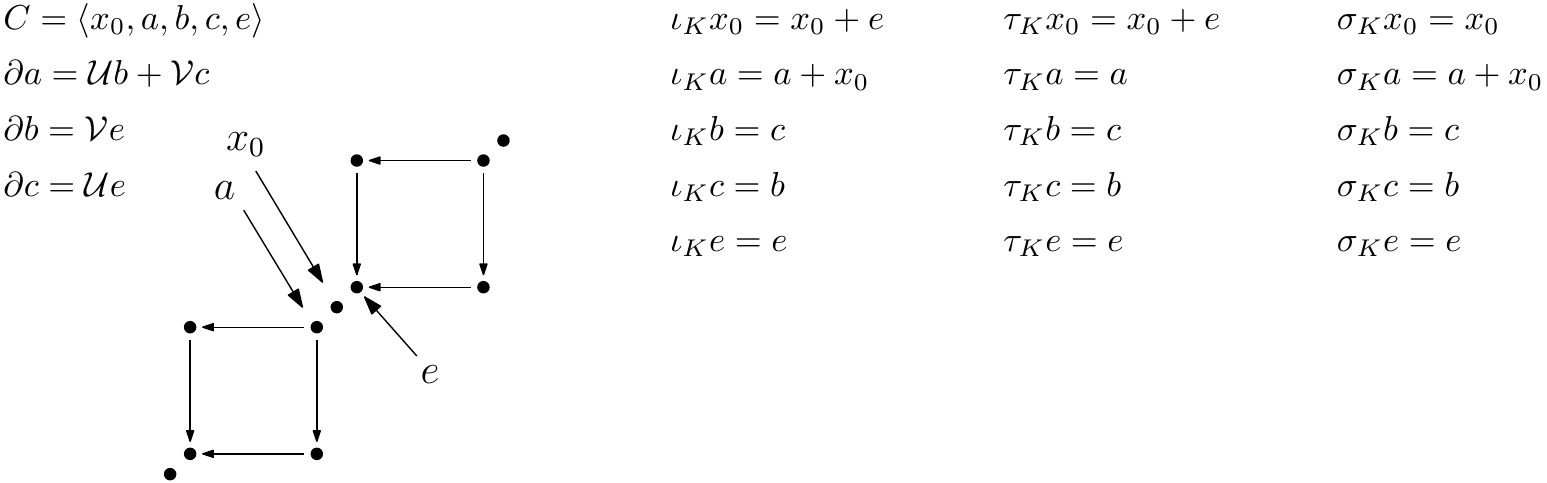}
\caption{The $(\tau_K, \iota_K)$-complexes associated to $\tau$ and $\sigma$ on $4_1$.}\label{fig:23}
\end{figure}

\begin{example}\label{ex:stevedore}
There are two strong inversions $\tau$ and $\sigma$ on the stevedore, which are displayed in Figure~\ref{fig:24}. In Figure~\ref{fig:25}, we display their corresponding actions $\tau_K$ and $\sigma_K$ on $C = \CFK(6_1)$ (calculated in the basis with the indicated action of $\iota_K$). See \cite[Example 4.7]{DHM}. Note that the pairs $(C, \iota_K)$ and $(C, \sigma_K)$ are individually trivial. In both cases, the local map to the trivial complex is given by sending all generators except for $x_0$ to zero. The local map in the other direction has image $x_0$ in the former case but image $x_0 + e_1$ in the latter. However, there is no local map from the trivial complex that simultaneously commutes with both $\sigma_K$ and $\iota_K$ (up to homotopy), so the triple $(C, \sigma_K, \iota_K)$ is nontrivial. This can be checked via exhaustive casework, or by computing $\Vitl(K) = 1$.
\end{example}

\begin{figure}[h!]
\includegraphics[scale = 0.65]{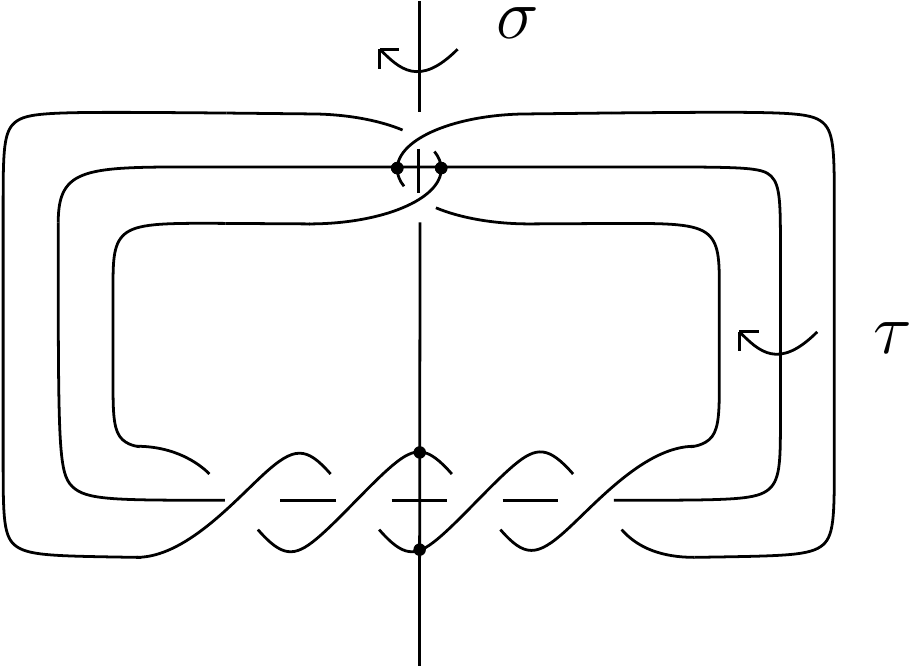}
\caption{The stevedore $6_1$ with two strong inversions $\tau$ and $\sigma$.}\label{fig:24}
\end{figure}

\begin{figure}[h!]
\includegraphics[scale = 0.95]{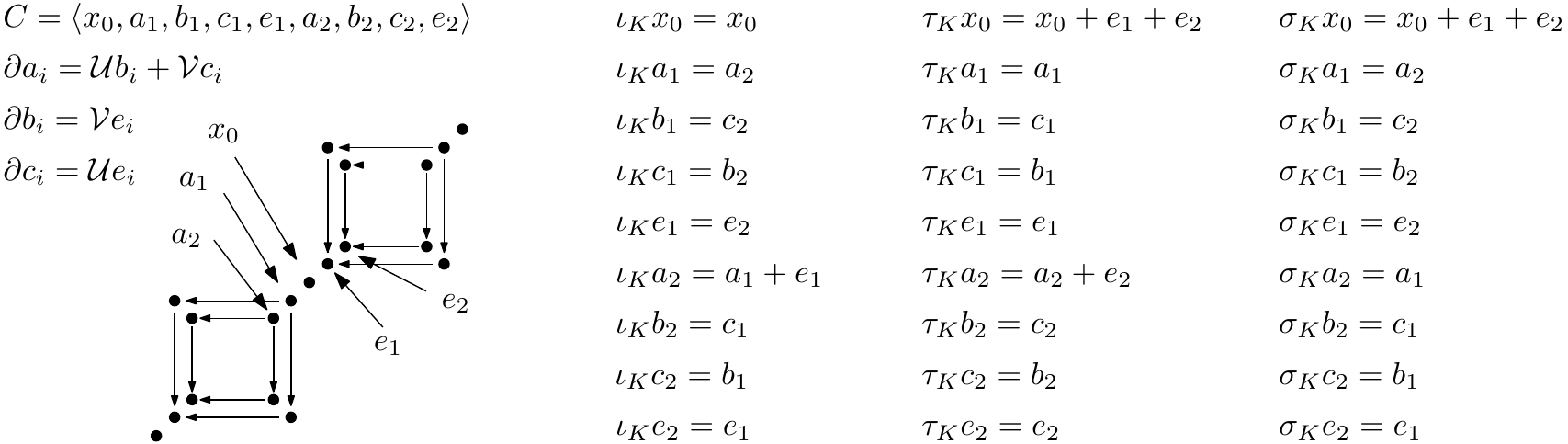}
\caption{The $(\tau_K, \iota_K)$-complexes associated to $\tau$ and $\sigma$ on $6_1$.}\label{fig:25}
\end{figure}

\begin{remark}
The reader may verify that in each of the above examples, performing a twist by $\varsigma_K$ does not change the homotopy equivalence class of the relevant triple. We thus suppress writing a choice of decoration or direction in both Example~\ref{ex:figureeight} and Example~\ref{ex:stevedore}.
\end{remark}

\section{Construction of $\tau_K$ and equivariant concordance}\label{sec:3}

In this section, we construct the action $\tau_K \colon \CFK(K) \rightarrow \CFK(K)$ of $\tau$ on the knot Floer complex of $K$. In order to do this, we first equip $K$ with an orientation and a symmetric pair of basepoints, which we collectively refer to as a \textit{decoration} on $K$. We then explain in what sense $\tau_K$ is independent of the choice of decoration. This turns out to be rather subtle, and will require an extended discussion about identifying different knot Floer complexes for the same knot in the case that the orientation or basepoints are changed. In particular, we show that if $(K, \tau)$ is a decorated strongly invertible knot, then the triple $(\CFK(K), \tau_K , \iota_K)$ is well-defined up to homotopy equivalence of $(\tau_K, \iota_K)$-complexes. If $(K, \tau)$ does not come with a decoration, then the homotopy equivalence class of $(\CFK(K), \tau_K , \iota_K)$ is only defined up to a twist by $\varsigma_K$, although the homotopy equivalence class of the \textit{pair} $(\CFK(K), \tau_K)$ is still well-defined. See Theorems~\ref{thm:3.2A} and \ref{thm:3.2B}.

We then turn to the behavior of $\tau_K$ under equivariant concordance. Here, we similarly modify the notion of an isotopy-equivariant homology concordance to hold in the decorated setting. We show that a decorated equivariant concordance induces a local equivalence of $(\tau_K, \iota_K)$-complexes. In the undecorated setting, this only holds up to a twist applied to one end of the concordance, although we still obtain a local equivalence of $\tau_K$-complexes. See Theorems~\ref{thm:3.3A} and \ref{thm:3.3B}. 

Finally, we discuss the connection between the decorated and directed categories. We show that a choice of direction similarly determines a homotopy equivalence class of $(\tau_K, \iota_K)$-complex, and that a concordance in the directed category again induces a local equivalence. We then put everything together and establish Theorem~\ref{thm:1.1}.

\subsection{Preliminaries}\label{sec:3.1}

Defining the action of $\tau$ will rely on a large number of auxiliary maps. In order to establish notation, we collect these below. We assume that the reader has some familiarity with the ideas of \cite{Zemkelinkcobord} and \cite{Zemkeconnected}.

\begin{definition}\label{def:mapset1}
Let $(K, w, z)$ be an oriented, doubly-based knot. 
\begin{enumerate}
\item Let $f$ be a diffeomorphism moving $(K, w, z)$ into $(f(K), f(w), f(z))$. If $\H$ is any choice of Heegaard data for $(K, w, z)$, then we obtain a pushforward set of Heegaard data $f\H$ for $(f(K), f(w), f(z))$. Moreover, $f$ induces a tautological chain isomorphism
\[
f \colon \CFK(\H) \rightarrow \CFK(f\H).
\]
which by abuse of notation we also denote by $f$. We call this the \textit{tautological pushforward}.
\item If $\H_1$ and $\H_2$ are two choices of Heegaard data for $(K, w, z)$, then there is a preferred homotopy equivalence
\[
\Phi(\H_1, \H_2) \colon \CFK(\H_1) \rightarrow \CFK(\H_2).
\]
This is unique (up to homotopy). We refer to $\Phi(\H_1, \H_2)$ as the \textit{naturality map}. The set of $\Phi$ form a transitive system.
\item Let $\H = ((\Sigma, \alphas, \betas, w, z), J)$ be a choice of Heegaard data for $(K, w, z)$. Then 
\[
\H^r = ((\Sigma, \alphas, \betas, z, w), J)
\]
is a choice of Heegaard data for $(K^r, z, w)$. Note that we interchange the roles of the basepoints $w$ and $z$, but we do not reverse orientation on $\Sigma$ or interchange the roles of $\alphas$ and $\betas$. The resulting diagram describes the knot $K$ with reversed orientation. There is a tautological skew-graded isomorphism
\begin{gather*}
sw : \CFK(\H) \rightarrow \CFK(\H^r)
\end{gather*}
with $\cU^i\cV^j \x \longmapsto \cU^{j}\cV^{i} \x$, given by mapping each intersection tuple to itself and interchanging the roles of $\cU$ and $\cV$. We call $sw$ the \textit{switch map}. 
\item Let $\H =( (\Sigma, \alphas, \betas, w, z), J)$ a choice of Heegaard data for $(K, w, z)$. Then 
\[
\bH = ((-\Sigma, \betas, \alphas, z, w), \bar{J})
\]
is a choice of Heegaard data for $(K, z, w)$. There is a tautological skew-graded isomorphism
\begin{gather*}
\eta \colon \CFK(\H) \rightarrow \CFK(\bH)
\end{gather*}
with $\cU^i\cV^j \x \longmapsto \cU^{j}\cV^{i} \x$, given by mapping each intersection tuple to itself and interchanging the roles of $\cU$ and $\cV$. We call $\eta$ the \textit{involutive conjugation map}. We stress that although $sw$ and $\eta$ might appear to be the same map, their codomains are different: the former represents $(K^r, z, w)$, while the latter represents  $(K, z, w)$.
\end{enumerate}
\end{definition}
\noindent
With the exception of the naturality map, we will usually suppress the data of $\H$ and thus the domain of the map in question.


\begin{lemma}\label{lem:mapset1}
The maps $f$, $\Phi$, $sw$, and $\eta$ all commute up to homotopy. Moreover, if $f$ and $g$ are two diffeomorphisms which commute, then their pushforwards commute up to homotopy.
\end{lemma}
\begin{proof}
Follows from naturality results established by Juh\'asz-Thurston-Zemke \cite{JTZ} and Zemke \cite{Zemkelinkcobord}.
\end{proof}

The maps in Lemma~\ref{lem:mapset1} should be interpreted as having the proper domain(s). For example, when we say that $f$ and $\Phi$ commute, we mean that we have a (homotopy) commutative square

\[\begin{tikzcd}
	{\CFK(\H_1)} && {\CFK(\H_2)} \\
	\\
	{\CFK(f\H_1)} && {\CFK(f\H_2)}
	\arrow["{f}"', from=1-1, to=3-1]
	\arrow["{f}", from=1-3, to=3-3]
	\arrow["{\Phi(\H_1, \H_2)}", from=1-1, to=1-3]
	\arrow["{\Phi(f\H_1, f\H_2)}", from=3-1, to=3-3]
\end{tikzcd}\]

\noindent
and similarly for the other maps. We thus write (for instance) $f \circ \Phi(\H_1, \H_2) \simeq \Phi(f\H_1, f\H_2) \circ f$ with the understanding that the two instances of $f$ have different domains. Note that implicitly, we are also claiming these operations commute when applied to Heegaard diagrams. For example, when we write $f \circ sw \simeq sw \circ f$, we are necessarily claiming that $f \H^r = (f \H)^r$, so that the codomains of both sides may be identified.

There are two other important maps that are derived from those in Definition~\ref{def:mapset1}:
\begin{definition}\label{def:mapset2}
Let $(K, w, z)$ be an oriented, doubly-based knot. 
\begin{enumerate}
\item Let $\H$ be any choice of Heegaard data for $(K, w, z)$. Then
\[
\eta \circ sw = sw \circ \eta \colon \CFK(\H) \rightarrow \CFK(\bH^r)
\]
provides a filtered isomorphism between $\CFK(\H)$ and $\CFK(\bH^r)$. Note that $\CFK(\bH^r)$ is a choice of Heegaard data for $(K^r, w, z)$; this has the reversed orientation but the same pair of basepoints. We call $\eta \circ sw$ the \textit{orientation-reversal map}. 
\item Let $\H$ be any choice of Heegaard data for $(K, w, z)$. Let $\rho$ be the half Dehn twist along the orientation of $K$ which moves $w$ into $z$ and $z$ into $w$. This induces a tautological pushforward
\[
\rho \colon \CFK(\H) \rightarrow \CFK(\rho\H).
\]
Note that $\rho \H$ represents the doubly-based knot $(K, z, w)$. 
We denote the half Dehn twist \textit{against} the orientation of $K$ by $\brho$, and denote the induced pushforward similarly. 
\end{enumerate}
\end{definition}

Since the definition of $\rho$ depends on the choice of orientation on $K$, the commutation relations for $\rho$ are slightly more subtle than those in Lemma~\ref{lem:mapset1}. In particular, since $sw$ reverses orientation on $K$, we have the following:

\begin{lemma}\label{lem:mapset2}
The map $\rho$ commutes with $\Phi$ and $\eta$ up to homotopy. However, the maps $\rho$ and $sw$ do not (in general) commute. Instead, we have
\[
(\rho \H)^r = \bar{\rho} \H^r \quad \text{and} \quad (\brho \H)^r = \rho \H^r
\]
and
\[
sw \circ \rho \simeq \brho \circ sw \quad \text{and} \quad sw \circ \brho \simeq \rho \circ sw.
\]
\end{lemma}
\begin{proof}
Follows from naturality results established by Juh\'asz-Thurston-Zemke \cite{JTZ} and Zemke \cite{Zemkelinkcobord}.
\end{proof}

Finally, we will often employ the following:

\begin{lemma}\label{lem:isotopicmaps}
Let $(K, w, z)$ be a doubly-based knot. Let $f$ and $g$ be two diffeomorphisms of $S^3$ such that $f(w) = g(w)$ and $f(z) = g(z)$, and suppose that $f$ and $g$ are isotopic rel $\{w, z\}$.\footnote{That is, there is an isotopy $H_t$ which sends $w$ to $f(w) = g(w)$ and $z$ to $f(z) = g(z)$ for all $t$.} Let $\H$ be any Heegaard data for $(K, w, z)$ and let $\H'$ be any choice of Heegaard data for $(f(K), f(w), f(z)) \simeq (g(K), g(w), g(z))$. Then
\[
\Phi(f\H, \H') \circ f \simeq \Phi(g\H, \H') \circ g.
\]
\end{lemma}
\begin{proof}
Follows from naturality results established by Juh\'asz-Thurston-Zemke \cite{JTZ} and Zemke \cite{Zemkelinkcobord}.
\end{proof}

\subsection{Construction of $\tau_K$}\label{sec:3.2} 

We now construct $\tau_K$. As usual, we begin by defining $\tau_K$ with respect to a fixed choice of Heegaard data for $K$. 

\begin{definition}\label{def:decoration}
Let $(K, \tau)$ be a strongly invertible knot. A \textit{decoration} on $(K, \tau)$ is a choice of orientation for $K$, together with an ordered pair of distinct basepoints $(w, z)$ on $K$ which are interchanged by $\tau$. Following the usual notation for a doubly-based knot, we denote this data by $(K, w, z)$. Here we introduce a slight abuse of notation, in that $K$ is not considered to have a fixed orientation as part of $(K, \tau)$, but is considered to have a fixed orientation as part of the data $(K, w, z)$.
\end{definition}
\noindent
A decorated knot is just an oriented, doubly-based knot in the usual sense, with the caveat that $w$ and $z$ are symmetric under the action of $\tau$. However, because the choice of extra data will be important, we formally emphasize this in Definition~\ref{def:decoration}.

Once a decoration for $(K, \tau)$ is chosen, we may select any set of Heegaard data $\H$ for $(K, w, z)$. Define an automorphism
\[
\tau_{\H} \colon \CFK(\H) \rightarrow \CFK(\H)
\]
as follows. We first apply the tautological pushforward
\[
t \colon \CFK(\H) \rightarrow \CFK(\tau\H).
\]
Here, we denote the pushforward by $t$ so as to not create confusion with the overall action $\tau_K$. Note that $\tau \H$ represents $(K^r, z, w)$, since $\tau$ is an orientation-reversing involution on $K$ and interchanges $w$ and $z$. Since $\H^r$ also represents $(K^r, z, w)$, we have a naturality map
\[
\Phi(\tau \H, \H^r) \colon \CFK(\tau \H) \rightarrow CFK(\H^r).
\]
Finally, we apply the switch map
\[
sw : \CFK(\H^r) \rightarrow \CFK(\H).
\]
\begin{definition}
The action $\tau_{\H} \colon \CFK(\H) \rightarrow \CFK(\H)$ is given by the composition
\[
\tau_{\H}: \CFK(\mathcal{H}) \xrightarrow{t} \CFK(\tau \mathcal{H}) \xrightarrow{\Phi} \CFK(\mathcal{H}^r) \xrightarrow{sw} \CFK(\mathcal{H}).
\]
Technically, the middle map $\Phi$ is only defined up to homotopy, but this clearly does not affect the homotopy class of $\tau_\H$.
\end{definition}

Theorem~\ref{thm:1.8} summarizes the salient features of $\tau_{\H}$:

\begin{proof}[Proof of Theorem~\ref{thm:1.8}] For $(1)$, applying Lemma~\ref{lem:mapset1} and keeping track of the appropriate domains gives the following chain of homotopies:
\begin{align*}
\tau^{2}_{\H}&= (sw \circ \Phi(\tau \H, \H^r) \circ t) \circ (sw \circ \Phi(\tau \H, \H^r) \circ t) \\
&= \Phi(\tau \H^r, \H) \circ \Phi(\H, \tau \H^r) \circ sw \circ t \circ sw \circ t \\
&= \Phi(\tau \H^r, \H) \circ \Phi(\H, \tau \H^r) \circ sw \circ sw \circ t \circ t \\
&\simeq \id.
\end{align*}
Here, in the last line we have used the fact that $t^2 = sw^2 = \id$, together with the fact that the set of $\Phi$ form a transitive system. Claim (2) follows from observing that the pushforward map $t$ and the naturality map $\Phi(\tau \H, \H^r)$ are graded and $\cR$-equivariant, whereas the map $sw$ is skew-graded and $\cR$-skew-equivariant. For (3), we apply Lemmas~\ref{lem:mapset1} and \ref{lem:mapset2} to $\tau_\H \circ \iota_\H$. We move all of the naturality maps to the left, simplify, and then collect the pushforward maps together:
\begin{align*}
\tau_\H \circ \iota_\H &= (sw \circ \Phi(\tau \H, \H^r) \circ t) \circ (\Phi(\rho \bH, \H) \circ \rho \circ \eta) \\
&\simeq \Phi(\tau \H^r, \H) \circ \Phi((\tau \rho \bH)^r, \tau \H^r) \circ sw \circ t \circ \rho \circ \eta \\
&\simeq \Phi(\tau \H^r, \H) \circ \Phi(\tau \brho \bH^r, \tau \H^r) \circ t \circ \brho \circ sw \circ \eta \\
&\simeq \Phi(\tau \brho \bH^r, \H) \circ t \circ \brho \circ sw \circ \eta.
\end{align*}
See \cite[Section 6.1]{HM} for the definition of $\iota_\H$. It will be convenient for us to replace $\varsigma_\H$ with $\varsigma_\H^{-1}$; this is allowed since $\varsigma_{\H}^2 \simeq \id$. Note that $\varsigma_\H^{-1}$ is represented by the basepoint-moving map against the orientation of $K$. Doing this, we obtain
\begin{align*}
\varsigma_\H^{-1} \circ \iota_\H \circ \tau_\H &= (\Phi(\brho^2 \H, \H) \circ \brho^2) \circ (\Phi(\rho\bH, \H) \circ \rho \circ \eta) \circ (sw \circ \Phi(\tau \H, \H^r) \circ t) \\
& \simeq \Phi(\brho^2 \H, \H)\circ \Phi(\brho^2\rho\bH, \brho^2\H) \circ \Phi(\brho^2\rho \tau \bH^r, \brho^2\rho \bH) \circ \brho^2 \circ \rho \circ \eta \circ sw \circ t \\
& \simeq \Phi(\brho^2 \rho \tau \bH^r, \H) \circ \brho^2 \circ \rho \circ \eta \circ sw \circ t \\
& \simeq \Phi(\tau \brho^2 \rho \bH^r, \H) \circ t \circ \brho^2 \circ \rho \circ sw \circ \eta. 
\end{align*}
The claim then follows from Lemma~\ref{lem:isotopicmaps} and the fact that $\tau \circ \brho^2 \circ \rho \simeq \tau \circ \brho$. Finally, the last part of the theorem follows from the fact that the naturality maps commute with each of the factors used in the definitions of $\tau_\H$ and $\iota_\H$.
\end{proof}

\subsection{Naturality of $\tau_K$}\label{sec:3.3}
Theorem~\ref{thm:1.8} shows that $(\CFK(K), \tau_\H, \iota_\H)$ is a $(\tau_K, \iota_K)$-complex whose homotopy type is independent of the choice of Heegaard data for the oriented, doubly-based knot $(K, w, z)$. Moreover, the homotopy equivalences between such triples are precisely the naturality maps of Definition~\ref{def:mapset1}. It thus remains to show that $\tau_{\H}$ is independent of the choice of decoration on $K$. The reason we have separated this from the claim of Theorem~\ref{thm:1.8} is that in general, there is no canonical identification between two knot Floer complexes for $K$ in the case that the orientation on $K$ is reversed or the basepoints are changed. For example, although one can write down complexes for $K$ and $K^r$ which are isomorphic, such an isomorphism is not via a naturality map $\Phi$.

We begin with the choice of orientation on $K$. Let $\H$ be any choice of Heegaard data for $(K, w, z)$. As discussed previously, we have the orientation-reversing isomorphism
\[
\eta \circ sw \colon \CFK(\H) \rightarrow \CFK(\bH^r).
\]
Note that the right-hand side represents $(K^r, w, z)$. We now have:

\begin{lemma}\label{lem:changeorientation}
Let $\H$ be any choice of Heegaard data for $(K, w, z)$ and $\bH^r$ be the corresponding Heegaard data for $(K^r, w, z)$. Then:
\begin{enumerate}
\item $(\eta \circ sw) \circ \tau_{\H} \simeq \tau_{\bH^r} \circ (\eta \circ sw)$
\item $(\eta \circ sw) \circ \iota_{\H} \simeq \varsigma_{\bH^r} \circ \iota_{\bH^r} \circ (\eta \circ sw)$
\end{enumerate}
That is, $\eta \circ sw$ provides a homotopy equivalence
\[
(\CFK(\H), \tau_\H, \iota_\H) \simeq (\CFK(\bH^r), \tau_{\bH^r}, \varsigma_{\bH^r} \circ \iota_{\bH^r}).
\]
\end{lemma}
\begin{proof}
Claim (1) follows immediately from Lemma~\ref{lem:mapset1}, as both $\eta$ and $sw$ commute with all of the individual factors of $\tau_\H$. To see Claim (2), it is more convenient to replace $\varsigma_{\bH^r}$ with $\varsigma_{\bH^r}^{-1}$. Applying Lemmas~\ref{lem:mapset1} and \ref{lem:mapset2}, we obtain
\begin{align*}
(\eta \circ sw) \circ \iota_{\H} &= (\eta \circ sw) \circ (\Phi(\rho \bH, \H) \circ \rho \circ \eta) \\
&\simeq \Phi((\rho \H)^r, \bH^r) \circ \eta \circ sw \circ \rho \circ \eta \\
&\simeq \Phi(\brho \H^r, \bH^r) \circ \brho \circ sw
\end{align*}
and
\begin{align*}
\varsigma_{\bH^r}^{-1} \circ \iota_{\bH^r} \circ (\eta \circ sw)&= (\Phi(\brho^2 \bH^r, \bH^r) \circ \brho^2) \circ (\Phi(\rho \H^r, \bH^r) \circ \rho \circ \eta) \circ (\eta \circ sw) \\
&\simeq \Phi(\brho^2 \bH^r, \bH^r) \circ \Phi (\brho^2\rho \H^r,  \brho^2\bH^r) \circ  \brho^2 \circ \rho \circ \eta \circ \eta \circ sw \\
&\simeq \Phi(\brho^2\rho \H^r, \bH^r) \circ \brho^2 \circ \rho \circ sw.
\end{align*}
The claim then follows from the fact $\brho^2 \circ \rho \simeq \brho$.
\end{proof}

\noindent
Lemma~\ref{lem:changeorientation} thus says that the homotopy equivalence class of $\tau_\H$ is independent of the choice of orientation on $K$. However, note that the homotopy equivalence used in Lemma~\ref{lem:changeorientation} does \textit{not} not establish this for $\iota_{\H}$: the graded isomorphism $\eta \circ sw$ between $\CFK(\H)$ and $\CFK(\bH^r)$ intertwines $\iota_\H$ and $\varsigma_{\bH^r} \circ \iota_{\bH^r}$. We thus obtain a homotopy equivalence between the $(\tau_K, \iota_K)$-triple associated to $(K, w, z)$, and the \textit{twist} of the $(\tau_K, \iota_K)$-triple associated to $(K^r, w, z)$.

We now investigate the dependence of $\tau_\H$ on the choice of basepoints. Let $(w, z)$ and $(w', z')$ be two symmetric pairs of basepoints for $(K, \tau)$. The fixed-point axis of $\tau$ separates $K$ into two arcs, both of which contain a single basepoint from each pair. There are two possibilities: either $w$ and $w'$ lie in the same subarc of $K$, or they lie in opposite subarcs. If $w$ and $w'$ lie in the same subarc, then there is an obvious equivariant diffeomorphism of $S^3$ which moves $w$ into $w'$ and $z$ into $z'$; this is formed by pushing $w$ and $z$ along $K$ in a symmetric fashion, as shown in Figure~\ref{fig:31}.
\begin{figure}[h!]
\includegraphics[scale = 1]{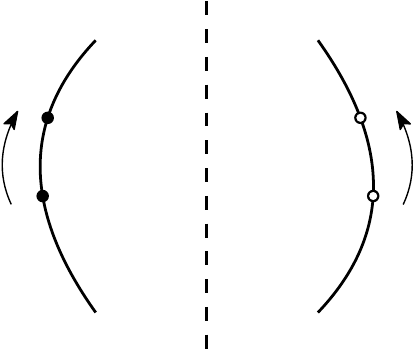}
\caption{Equivariant basepoint-pushing diffeomorphism.}\label{fig:31}
\end{figure}
\noindent
The desired naturality statement in this case is then subsumed by a more general claim regarding equivariant diffeomorphisms of $S^3$. In general, if $f \colon S^3 \rightarrow S^3$ is an equivariant diffeomorphism, then the image $(f(K), \tau)$ of $(K, \tau)$ is another strongly invertible knot. We have:

\begin{lemma}\label{lem:equivariantdiffeomorphism}
Let $f \colon S^3 \rightarrow S^3$ be an equivariant diffeomorphism. Let $\H$ be any choice of Heegaard data for $(K, w, z)$ and $f\H$ be the corresponding pushforward data for $(f(K), f(w), f(z))$. Then
\begin{enumerate}
\item $f \circ \tau_{\H} \simeq \tau_{f\H} \circ f$
\item $f \circ \iota_{\H} \simeq \iota_{f\H} \circ f$
\end{enumerate}
That is, $f$ provides a homotopy equivalence of triples
\[
(\CFK(\H), \tau_\H, \iota_\H) \simeq (\CFK(f\H), \tau_{f\H}, \iota_{f\H}).
\]
\end{lemma}
\begin{proof}
This follows from the fact that $f$ commutes with each of the components of $\tau_\H$ and $\iota_\H$.
\end{proof}
\noindent
Lemma~\ref{lem:equivariantdiffeomorphism} says that up to homotopy equivalence, the triple $(\CFK(K), \tau_\H, \iota_\H)$ is a well-defined invariant up to equivariant diffeomorphism (in the decorated setting). In particular, by using Figure~\ref{fig:31} we may move $(w, z)$ to any other symmetric pair $(w', z')$ so long as $w$ and $w'$ lie in the same subarc of $K$. 

Now consider the case in which $w'$ is chosen to lie in the opposite subarc from $w$. Due to our analysis of the previous case, we may in fact assume that $w' = z$ and $z' = w$. There is then an obvious diffeomorphism which moves $(K, w, z)$ into $(K, z, w)$; namely, the half Dehn twist $\rho$ along the oriented knot $K$. However, $\rho$ does \textit{not} commute with all the components of $\tau_\H$. We instead have:

\begin{lemma}\label{lem:changebasepoints}
Let $\H$ be any choice of Heegaard data for $(K, w, z)$ and $\rho \H$ be the corresponding pushforward data for $(K, z, w)$ under the half Dehn twist $\rho$. Then 
\begin{enumerate}
\item $\rho \circ \tau_{\H} \simeq \varsigma_{\rho\H} \circ \tau_{\rho\H} \circ \rho$
\item $\rho \circ \iota_{\H} \simeq \iota_{\rho\H} \circ \rho$
\end{enumerate}
That is, $\rho$ provides a homotopy equivalence of triples
\[
(\CFK(\H), \tau_\H, \iota_\H) \simeq (\CFK(\rho\H), \varsigma_{\rho\H} \circ \tau_{\rho\H}, \iota_{\rho\H}).
\]
\end{lemma}
\begin{proof}
To prove the first claim, we compute
\begin{align*}
\rho \circ \tau_{\H} &= \rho \circ (sw \circ \Phi(\tau\H, \H^r) \circ t) \\
&\simeq \Phi(\rho\tau\H^r, \rho\H) \circ \rho \circ sw \circ t \\
&\simeq \Phi(\rho\tau\H^r, \rho\H) \circ \rho \circ t \circ sw
\end{align*}
and
\begin{align*}
\varsigma_{\rho\H} \circ \tau_{\rho\H} \circ \rho &= (\Phi(\rho^2\rho\H, \rho\H) \circ \rho^2) \circ (sw \circ \Phi(\tau\rho\H, (\rho\H)^r) \circ t) \circ \rho \\
&\simeq \Phi(\rho^2\rho\H, \rho\H) \circ \Phi(\rho^2(\tau\rho\H)^r, \rho^2\rho\H) \circ \rho^2 \circ sw \circ t \circ \rho \\
&\simeq \Phi(\rho^2(\tau\rho\H)^r, \rho\H) \circ \rho^2 \circ sw \circ t \circ \rho \\
&\simeq \Phi(\rho^2 \brho \tau \H^r, \rho\H) \circ \rho^2 \circ \brho \circ t \circ sw.
\end{align*}
The claim then follows from Lemma~\ref{lem:isotopicmaps} and the fact that $\rho^2 \circ \brho \circ \tau \simeq \rho \circ \tau$. The second assertion of the lemma follows from the fact that $\rho$ commutes with all the individual components of $\iota_\H$.
\end{proof}

Lemma~\ref{lem:changebasepoints} might seem to imply that the homotopy equivalence class of $\tau_\H$ is dependent on the order of the basepoints $w$ and $z$. Indeed, without a choice of decoration, it initially appears that $\tau_\H$ is only well-defined up to composition with the Sarkar map. This is a reasonable heuristic, but not quite correct: it is important to stress that there is no canonical way to compare two knot Floer complexes for $K$ with different pairs of basepoints. Lemma~\ref{lem:changebasepoints} should thus be interpreted as a statement specifically regarding the choice of homotopy equivalence $\rho$ between a choice of Heegaard data for $(K, w, z)$ and a choice of Heegaard data for $(K, z, w)$. \textit{A priori}, it is possible that a different choice of homotopy equivalence might intertwine $\tau_\H$ and $\tau_{\rho\H}$. Indeed, recall from Lemma~\ref{lem:onetwist} that $\varsigma_{\H} \circ \tau_\H$ and $\tau_\H$ are conjugate up to homotopy. More precisely,
\[
(\CFK(\H), \tau_\H, \iota_\H) \simeq (\CFK(\H), \varsigma_{\H} \circ \tau_\H, \varsigma_{\H} \circ \iota_\H).
\]
\noindent
Hence Lemma~\ref{lem:changebasepoints} combined with Lemma~\ref{lem:onetwist} shows that the homotopy equivalence class of $\tau_\H$ \textit{is} invariant under exchanging the roles of $w$ and $z$, while the homotopy equivalence class of the \textit{triple} $(\CFK(\H), \tau_\H, \iota_\H)$ is not, at least \textit{a priori}. Instead, we see that $(\CFK(\H), \tau_\H, \iota_\H)$ is homotopy equivalent to either of the classes
\[
(\CFK(\rho\H), \varsigma_{\rho \H} \circ \tau_{\rho\H}, \iota_{\rho\H}) \simeq (\CFK(\rho\H), \tau_{\rho\H}, \varsigma_{\rho\H} \circ  \iota_{\rho\H}).
\]


The situation is summarized in the following pair of theorems:

\begin{theorem}\label{thm:3.2A}
Let $(K, \tau)$ be a decorated strongly invertible knot. The triple $(\CFK(\H), \tau_\H, \iota_\H)$ is independent, up to homotopy equivalence, of the choice of $\H$ so long as $\H$ is compatible with the chosen decoration; moreover, it is an invariant of $(K, \tau)$ up to equivariant diffeomorphism, interpreted in the decorated setting.
\end{theorem}
\begin{proof}
Follows from Theorem~\ref{thm:1.8} and Lemma~\ref{lem:equivariantdiffeomorphism}.
\end{proof}

In the decorated setting, we thus suppress the choice of Heegaard data and refer to the homotopy equivalence class of the triple $(\CFK(K), \tau_K, \iota_K)$ unambiguously. In the undecorated setting, we instead have the following:

\begin{theorem}\label{thm:3.2B}
The homotopy equivalence class of $(\CFK(\H), \tau_\H)$ is independent of the choice of decoration on $(K, \tau)$. Reversing orientation or interchanging the basepoints each alters the homotopy equivalence class of $(\CFK(\H), \tau_\H, \iota_\H)$ by a twist.
\end{theorem}
\begin{proof}
Follows from Lemma~\ref{lem:changeorientation} and Lemma~\ref{lem:changebasepoints}.
\end{proof}

In the undecorated setting, we thus refer to $(\CFK(K), \tau_K)$ unambiguously, although this is not entirely natural. However, we must take care when discussing $(\CFK(K), \tau_K, \iota_K)$ in the undecorated setting. Explicitly, we have constructed homotopy equivalences:

\begin{itemize}
\item $(\CFK(K, w, z), \tau_K, \iota_K) \simeq (\CFK(K^r, w, z), \tau_{K^r}, \varsigma \iota_{K^r})$ via Lemma~\ref{lem:changeorientation}
\item $(\CFK(K, w, z), \tau_K, \iota_K) \simeq (\CFK(K, z, w), \varsigma \tau_{K}, \iota_K)$ via Lemma~\ref{lem:changebasepoints}
\item $(\CFK(K, w, z), \tau_K, \iota_K) \simeq (\CFK(K, w, z), \varsigma \tau_{K}, \varsigma \iota_{K})$ via Lemma~\ref{lem:onetwist}.
\end{itemize}
\noindent
Again, however, note that these should not be treated as canonical.

\subsection{Equivariant concordance}\label{sec:3.4}
We now turn to the behavior of $\tau_K$ under equivariant concordance. As in the previous section, we first need to define a notion of equivariant concordance in the decorated setting. 

\begin{definition}\label{def:decoratedeqconc}
Let $(K_1, \tau_1)$ and $(K_2, \tau_2)$ be two decorated strongly invertible knots and let $(W, \tau_W, \Sigma)$ be an isotopy-equivariant homology concordance between them. We say that $(W, \tau_W, \Sigma)$ \textit{respects the decorations} (alternatively, \textit{is equivariant in the decorated sense}) if:
\begin{enumerate}
\item $\Sigma$ is an oriented knot concordance; and,
\item\label{itm:2} We can find a pair of properly embedded arcs $\gamma_1, \gamma_2 \subseteq \Sigma$ such that:
\begin{enumerate}
\item Each $\gamma_i$ has one end point on $K_1$ and one endpoint on $K_2$, and these endpoints are fixed by $\tau_1$ and $\tau_2$, respectively.
\item We have an isotopy (rel boundary) moving $(\tau_W(\Sigma), \tau_W(\gamma_1), \tau_W(\gamma_2))$ into $(\Sigma, \gamma_1, \gamma_2)$.
\item The arcs divide $\Sigma$ into two rectangular regions, one of which contains both $w_1$ and $w_2$ (we call this the \textit{black region}), and the other of which contains both $z_1$ and $z_2$ (we call this the \textit{white region}).
\end{enumerate}
\end{enumerate}
\end{definition}

\noindent
When the context is clear, we refer to such a $\Sigma$ as a \textit{decorated isotopy-equivariant concordance}. Note that $\Sigma$ is just an isotopy-equivariant cobordism for which we can find an appropriate set of isotopy-equivariant dividing curves, in the sense of \cite{Zemkelinkcobord}.

\begin{theorem}\label{thm:3.3A}
Let $(K_1, \tau_1)$ and $(K_2, \tau_2)$ be two decorated strongly invertible knots. A decorated isotopy-equivariant concordance between $(K_1, \tau_1)$ and $(K_2, \tau_2)$ induces a local equivalence
\[
(\CFK(K_1), \tau_{K_1}, \iota_{K_1}) \sim (\CFK(K_2), \tau_{K_2}, \iota_{K_2}).
\]
\end{theorem}
\begin{proof}
By work of Zemke \cite{Zemkelinkcobord}, we obtain a concordance map 
\[
F_{W, \mathcal{F}}: \CFK(K_1) \rightarrow \CFK(K_2).
\]
Here, $\mathcal{F}$ represents the concordance $\Sigma$ with the choice of dividing curves $\gamma_1$ and $\gamma_2$. It is standard that $F_{W, \mathcal{F}}$ is grading-preserving and has the requisite behavior under localization. In \cite[Section 4.5]{HM} and \cite[Theorem 1.5]{Zemkeconnected}, it is shown that $F_{W, \mathcal{F}}$ is $\iota_K$-equivariant (up to homotopy). It thus suffices to show that it is $\tau_K$-equivariant.

Consider the diagram:

\[\begin{tikzcd}
	{\mathcal{CFK}(K_1, w_1, z_1) } && {\mathcal{CFK}(K_2, w_2, z_2)} \\
	\\
	{\mathcal{CFK}(K_1^r, z_1, w_1) } && {\mathcal{CFK}(K_2^r, z_2, w_2)}\\
	\\
	{\mathcal{CFK}(K_1, w_1, z_1) } && {\mathcal{CFK}(K_2, w_2, z_2)}
	\arrow["{F_{W,\mathcal{F}}}", from=1-1, to=1-3]
	\arrow["{t}"', from=1-1, to=3-1]
	\arrow["{t}", from=1-3, to=3-3]
	\arrow["{F_{W,\tau_W(\mathcal{F})}}", from=3-1, to=3-3]
	\arrow["sw"', from=3-1, to=5-1]
	\arrow["sw", from=3-3, to=5-3]
	\arrow["{F_{W,sw(\tau_W(\mathcal{F}))}}", from=5-1, to=5-3]
\end{tikzcd}\]
\noindent
Here, by $\mathcal{CFK}(K_1, w_1, z_1)$, we mean any representative for the complex of $(K_1, w_1, z_1)$ in the transitive system of complexes for doubly-basepointed knots. (Similarly for the other entries in the diagram; we have thus suppressed writing the naturality maps $\Phi$ as part of the vertical arrows.) 

The first square of this diagram commutes due to the diffeomorphism invariance of link cobordisms \cite[Section 1.1]{Zemkelinkcobord}. By $\tau_W(\mathcal{F})$, we mean the image of the decoration of $\mathcal{F}$ under $\tau_W$. The second square of the diagram also tautologically commutes; here, $sw(\tau_W(\mathcal{F}))$ is obtained from $\tau_W(\mathcal{F})$ by interchanging the roles of the black and white regions on $\tau_W(\mathcal{F})$ and reversing orientation. The fact that our concordance is equivariant in the decorated sense shows that $sw(\tau_W(\mathcal{F}))$ is isotopic to $\mathcal{F}$ rel boundary, including the dividing curves and coloring of regions on $\Sigma$. The isotopy invariance of link cobordisms then implies that
\[
F_{W,sw(\tau_W(\mathcal{F}))} \simeq F_{W, \mathcal{F}}.
\]
This shows that $F_{W, \mathcal{F}}$ homotopy commutes with $\tau_K$ and hence constitutes a local map from $(\CFK(K_1), \tau_{K_1}, \iota_{K_1})$ to $(\CFK(K_2), \tau_{K_2}, \iota_{K_2})$. Turning the concordance around gives the local map in the other direction and completes the proof.
\end{proof}

In the decorated setting, the local equivalence class of the triple $(\CFK(K), \tau_K, \iota_K)$ is thus an invariant of isotopy-equivariant concordance. If $(K_1, \tau_1)$ and $(K_2, \tau_2)$ do not come equipped with decorations, then (according to Theorem~\ref{thm:3.2B}) we may still unambiguously speak of the homotopy equivalence classes of $(\CFK(K_1), \tau_{K_1})$ and $(\CFK(K_2), \tau_{K_2})$. We claim that in the presence of an (undecorated) isotopy-equivariant concordance (as in Definition~\ref{def:isotopyeqconc}), these are again guaranteed to be locally equivalent:

\begin{theorem}\label{thm:3.3B}
Let $(K_1, \tau_1)$ and $(K_2, \tau_2)$ be two strongly invertible knots. An isotopy-equivariant concordance between $(K_1, \tau_1)$ and $(K_2, \tau_2)$ gives a local equivalence of pairs
\[
(\CFK(K_1), \tau_{K_1}) \simeq (\CFK(K_2), \tau_{K_2}).
\]
Moreover, suppose we equip $(K_1, \tau_1)$ and $(K_2, \tau_2)$ with decorations, so that the homotopy equivalence classes of their associated $(\tau_K, \iota_K)$-complexes are defined. Then $(\CFK(K_1), \tau_{K_1}, \iota_{K_1})$ is locally equivalent to either $(\CFK(K_2), \tau_{K_2}, \iota_{K_2})$ or the twist of $(\CFK(K_2), \tau_{K_2}, \iota_{K_2})$.
\end{theorem}

\begin{proof}
Because of the discussion following Lemma~\ref{lem:onetwist}, the first claim follows from the second. Thus, let $(K_1, \tau_1)$ and $(K_2, \tau_2)$ be two decorated strongly invertible knots. Let $(W, \tau_W, \Sigma)$ be an equivariant concordance between them which may not be equivariant in the decorated sense. Due to Lemma~\ref{lem:changeorientation}, up to twisting the $(\tau_K, \iota_K)$-complexes at either end, we may assume that $\Sigma$ is an oriented concordance. 

Now choose any pair of properly embedded arcs $\gamma_1, \gamma_2 \subseteq \Sigma$ satisfying (a) and (c) of Definition~\ref{def:decoratedeqconc}(\ref{itm:2}). That is, each $\gamma_i$ has one end point on $K_1$ and one endpoint on $K_2$, and these endpoints are fixed by $\tau_1$ and $\tau_2$, respectively. Moreover, the curves $\gamma_1$ and $\gamma_2$ divide $\Sigma$ into two rectangular regions, one of which contains the $w_i$ basepoints and the other of which contains the $z_i$ basepoints. Let $\mathcal{F}$ denote this concordance with the choice of dividing arcs $\gamma_1$ and $\gamma_2$. As usual, $F_{W, \mathcal{F}}$ commutes with $\iota_K$ (up to homotopy). Following the proof of Theorem~\ref{thm:3.3A}, we see that since $\Sigma$ may not be isotopy equivariant in the decorated sense, we no longer have that $sw(\tau_W(\mathcal{F}))$ is isotopic to $\mathcal{F}$. Instead, $sw(\tau_W(\mathcal{F}))$ is necessarily isotopic to a decorated concordance obtained by applying some number of Dehn twists to $\mathcal{F}$.

The concordance map associated to this altered decoration is given by precomposing the concordance map for $sw(\tau_W(\mathcal{F}))$ with a power of the Sarkar map. Following the proof of Theorem~\ref{thm:3.3A}, we thus see that
\[
\tau_{K_2} \circ F_{W, \mathcal{F}} \simeq F_{W, \mathcal{F}} \circ (\varsigma_{K_1}^n \circ\tau_{K_1}).
\]
Hence $F_{W, \mathcal{F}}$ intertwines $\tau_{K_1}$ and $\tau_{K_2}$ up to composition with some power of the Sarkar map. As the Sarkar map is a homotopy involution, the claim follows.
\end{proof}

In the undecorated setting, an equivariant concordance thus only induces a local equivalence of $(\tau_K, \iota_K)$-triples up to twist. (Of course, note that if our knots are not decorated, then these triples are only defined up to twist anyway.) However, we still obtain a local equivalence between their $\tau_K$-complexes.

Having established the necessary naturality results, we now conclude with the construction of the numerical invariants of Theorem~\ref{thm:1.1}:

\begin{definition}\label{def:Vinvariants}
Let $(K, \tau)$ be a strongly invertible knot, which may be neither directed nor decorated. Fix any decoration on $(K, \tau)$ and consider the resulting $(\tau_K, \iota_K)$-complex $(\CFK(K), \tau_K, \iota_K)$. Following Section~\ref{sec:2.5}, define:
\[
\Vtu(K) = \Vtu(\CFK(K), \tau_K, \iota_K) \text{ and } \Vtl(K) = \Vtl(\CFK(K), \tau_K, \iota_K)
\]
and
\[
\Vitu(K) = \Vitu(\CFK(K), \tau_K, \iota_K) \text{ and } \Vitl(K) = \Vitl(\CFK(K), \tau_K, \iota_K).
\]
This is independent of the choice of decoration. Indeed, according to Theorem~\ref{thm:3.2B}, changing the decoration on $(K, \tau)$ corresponds to applying a twist by $\varsigma_K$. However, due to Lemma~\ref{lem:twistnochange}, our numerical invariants are not altered by this operation.
\end{definition}

Putting everything together, we obtain:

\begin{proof}[Proof of Theorem~\ref{thm:1.1}]
By Theorem~\ref{thm:3.3B}, an isotopy-equivariant homology concordance (in the undecorated category) induces a local equivalence of $(\tau_K, \iota_K)$-complexes, up to a twist by $\varsigma_K$. By Lemma~\ref{lem:twistnochange}, this leaves our numerical invariants unchanged.
\end{proof}

\subsection{Directed knots}\label{sec:3.5}

We now turn to the connection between the decorated and directed settings.

\begin{definition}\label{def:compatibledir}
Let $(K, \tau)$ be a directed strongly invertible knot, in the sense of Definition~\ref{def:direction}. We say that a decoration $(K, w, z)$ is \textit{compatible} with this choice of direction if the oriented subarc of $K$ containing the $z$-basepoint induces the same orientation on its boundary as the chosen half-axis. See Figure~\ref{fig:32}.
\end{definition}
\begin{figure}[h!]
\includegraphics[scale = 0.8]{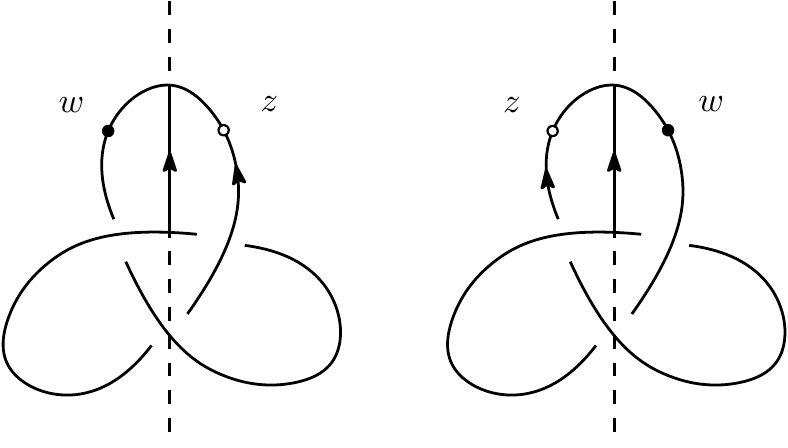}
\caption{Two decorations compatible with a fixed choice of direction.}\label{fig:32}
\end{figure}

If $(K, \tau)$ is a directed strongly invertible knot, then (up to orientation-preserving equivariant diffeomorphism) there are two compatible decorations on $(K, \tau)$. These are related to each other by simultaneously reversing orientation on $K$ and interchanging the roles of $w$ and $z$. See Figure~\ref{fig:32}. Note that as discussed in Remark~\ref{rem:notoriented}, a directed strongly invertible knot does not generally come with a specified orientation. Conversely, suppose that $(K, \tau)$ is a decorated strongly invertible knot. Then there are two possible choices of direction which are compatible with this decoration; they are related by simultaneously switching the half-axis and reversing the axis orientation. 

\begin{theorem}\label{thm:3.5A}
We have a well-defined set map from the directed equivariant concordance group to the local equivalence group of $(\tau_K, \iota_K)$-complexes.
\[
h \colon \eC \rightarrow \K_{\tau, \iota}
\]
\end{theorem}
\begin{proof}
Let $(K, \tau)$ be a directed strongly invertible knot. As explained above, the choice of direction determines two compatible decorations on $(K, \tau)$, which are related to each other by simultaneously reversing orientation on $K$ and interchanging the roles of $w$ and $z$. By Theorem~\ref{thm:3.2B}, applying both of these operations in succession does not change the homotopy type of the associated $(\tau_K, \iota_K)$-complex. Hence using the convention of Definition~\ref{def:compatibledir}, we may unambiguously talk of the $(\tau_K, \iota_K)$-complex of a directed knot.

Moreover, suppose that we have a directed equivariant concordance $(\Sigma, \tau_{S^3 \times I})$ from $(K_1, \tau_1)$ to $(K_2, \tau_2)$. Definition~\ref{def:direqconc} implies that we can find a pair of arcs $\gamma_1$ and $\gamma_2$ which run along the length of $\Sigma$ and are fixed by $\tau_{S^3 \times I}$. We may choose our compatible decorations on $K_1$ and $K_2$ such that $\Sigma$ is an oriented concordance. Then $\Sigma$ (with the arcs $\gamma_1$ and $\gamma_2)$ forms a decorated concordance in the sense of Definition~\ref{def:decoratedeqconc}.
\end{proof}


\section{Connected sums}\label{sec:4}

In this section, we establish further fundamental results regarding the action of $\tau_K$ and show that the map $h$ from Theorem~\ref{thm:3.5A} is a group homomorphism.

\subsection{Connected sums}\label{sec:4.1}
We begin with the connected sum formula. Let $(K_1, \tau_1)$ and $(K_2, \tau_2)$ be two directed strongly invertible knots. As discussed in Section~\ref{sec:2.1}, we may form the equivariant connected sum $(K_1 \# K_2, \tau_1 \# \tau_2)$, which is another directed strongly invertible knot. Note that according to Theorem~\ref{thm:3.5A}, we have a well-defined (up to homotopy equivalence) $(\tau_K, \iota_K)$-complex for each of the directed pairs $(K_1, \tau_1)$, $(K_2, \tau_2)$, and $(K_1 \# K_2, \tau_1 \# \tau_2)$, obtained by choosing a compatible decoration in each case.

\begin{theorem}\label{thm:connectedsum}
Let $(K_1, \tau_1)$ and $(K_2, \tau_2)$ be directed strongly invertible knots and $K_1 \# K_2$ be their equivariant connected sum. Then
\[
(\CFK(K_1 \# K_2), \tau_{K_1 \# K_2}, \iota_{K_1 \# K_2}) \quad \text{and} \quad (\CFK(K_1) \otimes \CFK(K_2), \tau_\otimes, \iota_\otimes)
\]
are homotopy equivalent, where
\[
\tau_\otimes = \tau_{K_1} \otimes \tau_{K_2}
\]
and
\[
\iota_\otimes = (\id \otimes \id + \Phi \otimes \Psi)(\iota_{K_1} \otimes \iota_{K_2}).
\]
\end{theorem}

\begin{proof}
Define an equivariant cobordism from $(S^{3}, K_1, \tau_1) \sqcup (S^{3}, K_2, \tau_2)$ to $(S^3, K_1 \# K_2, \tau_1 \# \tau_2)$ by attaching a 1-handle and then a band in the obvious manner. Denote this by $(W, \tau_W, \Sigma)$; the surface $\Sigma$ is schematically depicted in Figure~\ref{fig:41}. The knots $K_1$ and $K_2$ are represented by the two inner circles and have half-axes given by their respective horizontal diameters (oriented from left-to-right). Their connected sum $K_1 \# K_2$ is represented by the outer ellipse and has half-axis defined similarly. We place $w$- and $z$-basepoints on $K_1$, $K_2$, and $K_1 \# K_2$ as indicated in Figure~\ref{fig:41}; note that these are compatible with each of the chosen directions. Let $\mathcal{F}$ be the set of dividing arcs on $\Sigma$ consisting of the three indicated horizontal arcs. This makes $\Sigma$ into a cobordism which is equivariant in the decorated sense.
\begin{figure}[h!]
\includegraphics[scale = 1]{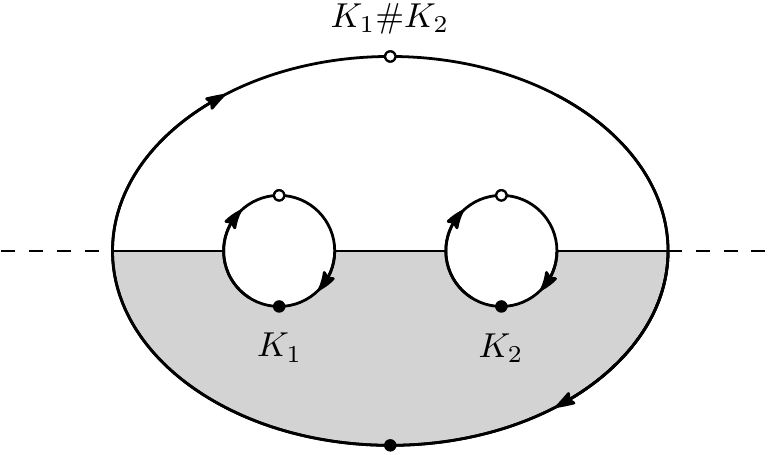}
\caption{Decorated equivariant cobordism from $K_1 \sqcup K_2$ to $K_1 \# K_2$. Black dots represent $w$-basepoints; white dots represent $z$-basepoints. The action of $\tau_W$ is given by reflection across the horizontal axis. See \cite[Figure 6.1]{Zemkeconnected}.}\label{fig:41}
\end{figure}

In \cite[Theorem 1.1]{Zemkeconnected}, Zemke shows that the map 
\[
F_{W, \mathcal{F}} \colon \CFK(K_1) \otimes \CFK(K_2) \rightarrow \CFK(K_1 \# K_2)
\]
defined by the link cobordism with decoration $\mathcal{F}$ is a homotopy equivalence, together with the map in the other direction constructed by turning the cobordism around. (Indeed, Figure~\ref{fig:41} is just \cite[Figure 6.1]{Zemkeconnected}, which corresponds to the map $G_1$ in \cite[Theorem 1.1]{Zemkeconnected}.) Moreover, according to \cite[Theorem 1.1]{Zemkeconnected}, this homotopy equivalence intertwines $(\id \otimes \id + \Phi \otimes \Psi)(\iota_{K_1} \otimes \iota_{K_2})$ on the incoming end with the connected sum involution $\iota_{K_1 \# K_2}$ on the outgoing end. We thus simply need to show that $F_{W, \mathcal{F}}$ intertwines $\tau_{K_1} \otimes \tau_{K_2}$ with $\tau_{K_1 \# K_2}$. This follows from the same argument as in Theorem~\ref{thm:3.3A}. We have the commutative diagram:

\[\begin{tikzcd}
	\CFK(K_1) \otimes \CFK(K_2) && \CFK(K_1 \# K_2) \\
	\\
	\CFK(K_1^r) \otimes \CFK(K_2^r) && \CFK(K_1^r \# K_2^r) \\
	\\
	\CFK(K_1) \otimes \CFK(K_2) && \CFK(K_1 \# K_2)
	\arrow["F_{W, \mathcal{F}}", from=1-1, to=1-3]
	\arrow["t \otimes t"', from=1-1, to=3-1]
	\arrow["t", from=1-3, to=3-3]
	\arrow["{sw \otimes sw}"', from=3-1, to=5-1]
	\arrow["sw", from=3-3, to=5-3]
	\arrow["F_{W, sw(\tau_W(\mathcal{F}))}", from=5-1, to=5-3]
	\arrow["{F_{W, \tau_W(\mathcal{F})}}", from=3-1, to=3-3]
\end{tikzcd}\]
\noindent
Each of the two squares commutes tautologically. It is clear from Figure~\ref{fig:41} that $sw(\tau_W(\mathcal{F}))$ coincides with $\mathcal{F}$; hence $F_{W, \mathcal{F}}$ intertwines $\tau_{K_1} \otimes \tau_{K_2}$ with $\tau_{K_1 \# K_2}$. The proof for the reversed cobordism map is similar. 
\end{proof}

\begin{remark}
Note that the above proof does not allow us to use $\Psi \otimes \Phi$ in place of $\Phi \otimes \Psi$ in the statement of Theorem~\ref{thm:connectedsum}, unless the conventions of Definition~\ref{def:eqconnectedsum} are also changed. This asymmetry is due to the fact that we have specifically used the map $G_1$ in \cite[Theorem 1.1]{Zemkeconnected}. The map $G_2$  in \cite[Theorem 1.1]{Zemkeconnected} intertwines $(\id \otimes \id + \Psi \otimes \Phi)(\iota_{K_1} \otimes \iota_{K_2})$ with $\iota_{K_1 \# K_2}$. However, $G_2$ does \textit{not} correspond to a decoration which is geometrically equivariant; see \cite[Figure 5.1]{Zemkeconnected}. See the discussion in Section~\ref{sec:2.3}.
\end{remark}

This completes the proof of Theorem~\ref{thm:1.9}:

\begin{proof}[Proof of Theorem~\ref{thm:1.9}]
Follows from Theorem~\ref{thm:3.5A} and Theorem~\ref{thm:connectedsum}.
\end{proof}


\subsection{The swapping involution}\label{sec:4.2}

We now compute the action of the swapping involution described in Section~\ref{sec:1.2}. In general, given a knot $K$ in $S^3$, we can form the connected sum $K \# K^r$ as in Figure~\ref{fig:42}. As discussed in \cite[Section 2]{BI}, this admits an obvious strong inversion. In fact, as discussed in \cite[Section 2]{BI}, we obtain a homomorphism from the usual concordance group to $\eC$ by choosing the half-axis depicted in Figure~\ref{fig:42}. We call this the \textit{swapping involution} on $K \# K^r$ and denote it by $\tau_{sw}$. Our goal will be to calculate the $(\tau_K, \iota_K)$-complex of $(K \# K^r, \tau_{sw})$ (with this choice of direction). 
\begin{figure}[h!]
\includegraphics[scale = 1.5]{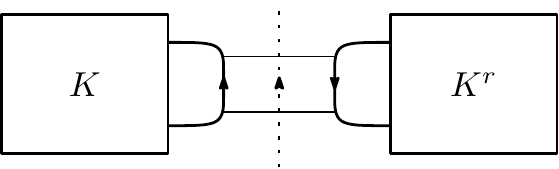}
\caption{The connected sum $K \# K^r$. The half-axis runs vertically across the band and is oriented to coincide with the orientation on $K$. See \cite[Figure 5]{BI}.}\label{fig:42}
\end{figure}

In order to compute the action of $\tau_{sw}$, we need to discuss the construction of $K \# K^r$ more precisely. Assume that $(K, w, z)$ is an oriented, doubly-based knot in $S^3$. We think of the projection of $K$ as lying entirely to the the left of a vertical axis. Denote 180-degree rotation about this axis by $\tau$. We obtain another doubly-based knot $(\tau K, \tau w, \tau z)$ by applying $\tau$ to $K$. Although this can of course be identified with $K$, it will be helpful for us to emphasize the second copy of $K$ as being the image of the first under $\tau$; we thus henceforth write $\tau K$ rather than $K$. We moreover modify the decoration on $\tau K$ by applying $sw$; this gives $(\tau K^r, \tau z, \tau w)$. 

As in Figure~\ref{fig:42}, we now attach a $\tau$-equivariant band to form the connected sum of $K$ and $\tau K^r$. It will be convenient for us to assume that this band has a particular arrangement with respect to the basepoints on $K$ and $\tau K^r$. Specifically, we require the foot of our band on $K$ to lie on the oriented subarc of $K$ running from $z$ to $w$, and the foot of our band on $\tau K^r$ to lie on the oriented subarc running from $\tau w$ to $\tau z$. We furthermore place a pair of symmetric basepoints $w'$ and $z'$ on $K \# \tau K^r$ in such a way so that $w'$ lies on $K$ and $z'$ lies on $\tau K^r$. See Figure~\ref{fig:43}. Note that this makes $(K \# \tau K^r, w', z')$ into a decorated strongly invertible knot, and this choice of decoration is compatible with the direction chosen in Figure~\ref{fig:42}.

\begin{figure}[h!]
\includegraphics[scale = 1]{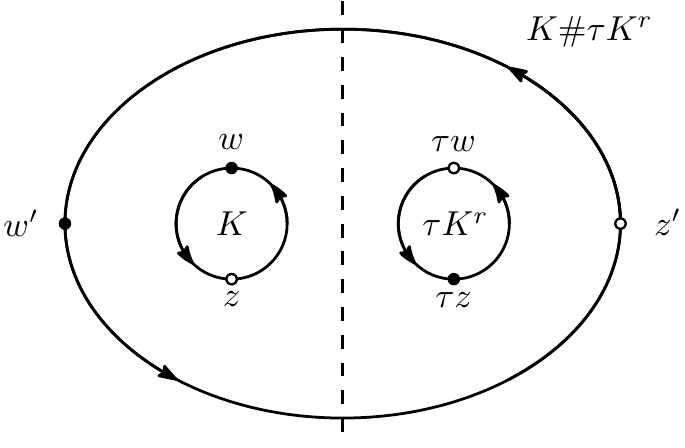}
\caption{Schematic depiction of $(K, w, z) \sqcup (\tau K^r, \tau z, \tau w)$ and $(K \# \tau K^r, w', z')$, together with a pair-of-pants cobordism between them. The actions of $\tau$, $\tau_{sw}$, and $\tau_W$ on the pair-of-pants is given by reflection across the vertical axis.}\label{fig:43}
\end{figure}

Before proceeding further, we first construct the induced action of $\tau$ on the disjoint union $(K, w, z) \sqcup (\tau K^r, \tau z, \tau w)$. Define a chain map
\[
\tau_{exch}: \CFK(K,w,z) \otimes \CFK(\tau K^{r}, \tau z, \tau w)  \rightarrow \CFK(K,w,z) \otimes \CFK(\tau K^{r},\tau z, \tau w)
\]
as follows. First apply the tautological pushforward associated to $\tau$. This induces an isomorphism from $\CFK(K, w, z)$ to $\CFK(\tau K, \tau w, \tau z)$, and also an isomorphism from $\CFK(\tau K^r, \tau z, \tau w)$ to $\CFK(K^r, z, w)$. We thus obtain a map
\[
t:\CFK(K,w,z) \otimes \CFK(\tau K^r,\tau z, \tau w)  \rightarrow \CFK(K^{r},z,w) \otimes \CFK(\tau K, \tau w, \tau z)
\]
which sends the first factor on the left isomorphically onto the second factor on the right, and the second factor on the left isomorphically onto the first factor on the right. We then apply the map $sw$ from Definition~\ref{def:mapset1} in each factor:
\[
sw \otimes sw: \CFK(K^{r},z,w) \otimes \CFK(\tau K,\tau w,\tau z) \rightarrow \CFK(K,w,z) \otimes \CFK(\tau K^{r},\tau z,\tau w).
\]
The action of $\tau_{exch}$ is thus defined by the composition 
\[
\tau_{exch} = (sw \otimes sw) \circ t.
\]
Note that this is just the action of $\tau_K$ defined in Section~\ref{sec:3.2}, generalized to the symmetric link $(K, w, z) \sqcup (\tau K^r, \tau z, \tau w)$.

We now establish the main theorem of this subsection:
\begin{theorem}\label{thm:swapping}
Denote the induced action of $\tau_{sw}$ also by $\tau_{sw}$. Then
\[
(\CFK(K \# \tau K^r), \tau_{sw}, \iota_{K \# \tau K^r}) \quad \text{and} \quad (\CFK(K) \otimes \CFK(\tau K^r), \tau_\otimes, \iota_\otimes)
\]
are homotopy equivalent, where
\[
\tau_\otimes = (\id \otimes \id + \Psi \otimes \Phi) \circ \tau_{exch}
\]
and
\[
\iota_\otimes = \varsigma_\otimes \circ (\id \otimes \id + \Psi \otimes \Phi) \circ (\iota_K \otimes \iota_{\tau K^r}).
\]
\end{theorem}

\begin{proof}
As in the proof of Theorem~\ref{thm:connectedsum}, we consider the pair-of-pants cobordism $(W, \tau_W, \Sigma)$ from $K \sqcup \tau K^{r}$ to $K \# \tau K^{r}$. Decorate $\Sigma$ with the set $\mathcal{F}$ of dividing curves depicted in Figure~\ref{fig:44}. The involution $\tau_W$ on this cobordism restricts to $\tau$ on the incoming end and $\tau_{sw}$ on the outgoing end. As in Figure~\ref{fig:43}, this is given by reflection across the vertical axis.

\begin{figure}[h!]
\includegraphics[scale = 1]{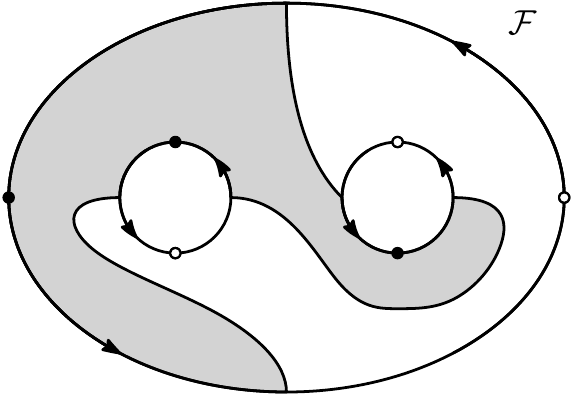}
\caption{The decoration $\mathcal{F}$ on $\Sigma$.}\label{fig:44}
\end{figure}

Now, the decoration $\mathcal{F}$ is not equivariant with respect to $\tau_W$. Nevertheless, we have the following homotopy-commutative diagram:
\[\begin{tikzcd}
	{\mathcal{CFK}(K,w,z) \otimes \mathcal{CFK}(\tau K^{r},\tau z,\tau w)} && {\mathcal{CFK}(K \# \tau K^{r}, w^{\prime}, z^{\prime})} \\
	\\
	{\mathcal{CFK}(K^{r},z,w) \otimes \mathcal{CFK}(\tau K,\tau w,\tau z)} && {\mathcal{CFK}(K^{r} \# \tau K,  z^{\prime}, w^{\prime})} \\
	\\
	{\mathcal{CFK}(K,w,z) \otimes \mathcal{CFK}(\tau K^{r},\tau z,\tau w)} && {\mathcal{CFK}(K \# \tau K^{r}, w^{\prime}, z^{\prime})}
	\arrow["t"', from=1-1, to=3-1]
	\arrow["{sw \otimes sw}"', from=3-1, to=5-1]
	\arrow["{F_{W,\mathcal{F}}}", from=1-1, to=1-3]
	\arrow["t", from=1-3, to=3-3]
	\arrow["{F_{W, \tau_W (\mathcal{F})}}", from=3-1, to=3-3]
	\arrow["sw", from=3-3, to=5-3]
	\arrow["{F_{W,sw(\tau_W (\mathcal{F}))}}", from=5-1, to=5-3]
\end{tikzcd}\]
Here, the decoration $sw(\tau_W (\mathcal{F}))$ is obtained by switching the designation of white and black regions in $\tau_W(\mathcal{F})$ and reversing orientation. Hence we obtain
\[
\tau_{sw} \circ F_{W,\mathcal{F}} \simeq  F_{W,sw(\tau_W (\mathcal{F}))} \circ \tau_{exch}.
\]
We now claim that
\begin{equation}\label{eq:swapequation}
F_{W,sw(\tau_W (\mathcal{F}))} \simeq F_{W, \mathcal{F}} \circ (\id \otimes \id + \Psi \otimes \Phi).
\end{equation}
To see this, we use the bypass relation for link cobordism maps established in \cite[Lemma 1.4]{Zemkeconnected}. A schematic outline of the bypass relation is given in Figure~\ref{fig:bypass}. 
\begin{figure}[h!]
\includegraphics[scale = 1]{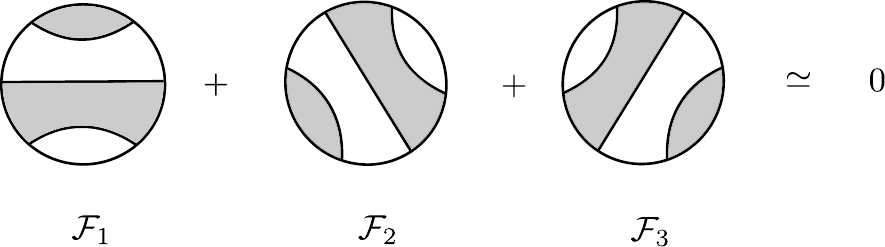}
\caption{The bypass relation, taken from \cite[Figure 1.2]{Zemkeconnected}. See \cite[Section 1.3]{Zemkeconnected} for discussion.}\label{fig:bypass}
\end{figure}

In our case, we apply the bypass relation to the dotted disk in the top-left of Figure~\ref{fig:45}. The effect of applying the bypass relation is also depicted in Figure~\ref{fig:45} and yields the claim. It follows that $F_{W, \mathcal{F}}$ intertwines $\tau_{sw}$ with $(\id \otimes \id + \Psi \otimes \Phi) \circ \tau_{exch}$.
\begin{figure}[h!]
\includegraphics[scale = 0.85]{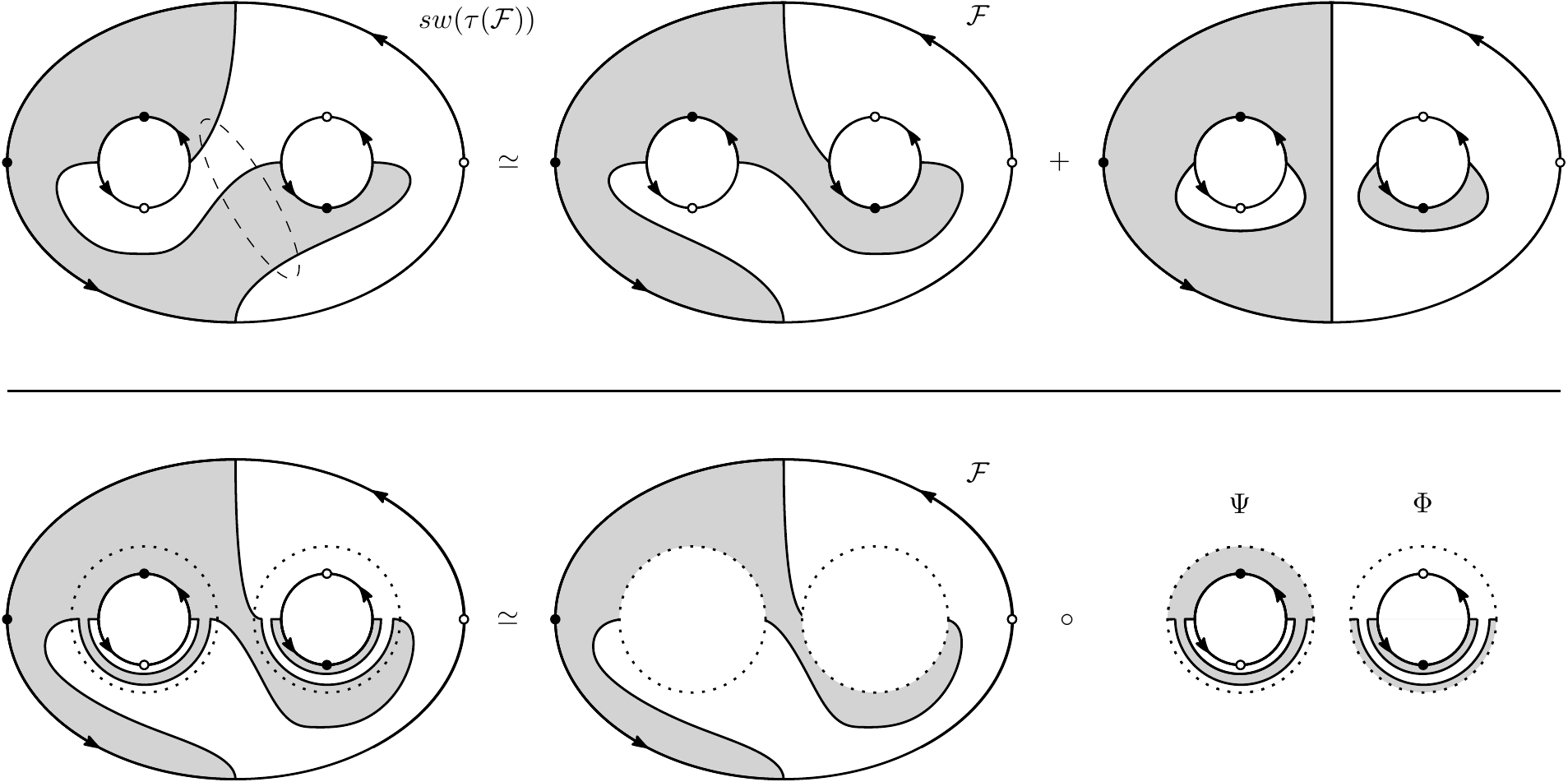}
\caption{Above: applying the bypass relation to the decoration $sw(\tau_W(\mathcal{F}))$. Below: the map induced by the rightmost decoration in the first line is homotopic to the composition $F_{W, \mathcal{F}} \circ (\Psi \otimes \Phi)$.}\label{fig:45}
\end{figure}

We now consider the behavior of $F_{W, \mathcal{F}}$ with respect to $\iota_K$. Note that $F_{W, \mathcal{F}}$ is not the same as the map used in the connected sum formula of Theorem~\ref{thm:connectedsum}, and thus does not necessarily intertwine $\iota_{K \# \tau K^r}$ and $(\id \otimes \id + \Phi \otimes \Psi) \circ (\iota_K \otimes \iota_{\tau K^r})$. Instead, we have the following commutative diagram:
\[\begin{tikzcd}
	{\mathcal{CFK}(K,w,z) \otimes \mathcal{CFK}(\tau K^{r},\tau z,\tau w)\quad\quad} && {\quad\quad\mathcal{CFK}(K \# \tau K^{r}, w^{\prime}, z^{\prime})} \\
	\\
	{\mathcal{CFK}(K,z,w) \otimes \mathcal{CFK}(\tau K^{r},\tau w,\tau z)\quad\quad} && {\quad\quad\mathcal{CFK}(K \# \tau K^{r},  z^{\prime}, w^{\prime})} \\
	\\
	{\mathcal{CFK}(K,w,z) \otimes \mathcal{CFK}(\tau K^{r},\tau z,\tau w)\quad\quad} && {\quad\quad\mathcal{CFK}(K \# \tau K^{r}, w^{\prime}, z^{\prime})}
	\arrow["\eta \otimes \eta"', from=1-1, to=3-1]
	\arrow["{\rho \otimes \rho}"', from=3-1, to=5-1]
	\arrow["{F_{W,\mathcal{F}}}", from=1-1, to=1-3]
	\arrow["\eta", from=1-3, to=3-3]
	\arrow["{F_{W, \eta(\mathcal{F})}}", from=3-1, to=3-3]
	\arrow["\rho", from=3-3, to=5-3]
	\arrow[inner ysep=0.85ex, "\rho \circ {F_{W, \eta(\mathcal{F})}} \circ (\brho \otimes \brho)", from=5-1, to=5-3]
\end{tikzcd}\]
Here, $\eta(\mathcal{F})$ is obtained from $\mathcal{F}$ by interchanging the roles of the white and black regions of $\mathcal{F}$, but not reversing orientation. Hence we obtain
\[
\iota_{K \# \tau K^r} \circ F_{W,\mathcal{F}} \simeq \left(\rho \circ {F_{W, \eta(\mathcal{F})}} \circ (\brho \otimes \brho)\right) \circ (\iota_{K} \otimes \iota_{\tau K^r}).
\]
We now claim that
\[
\rho \circ {F_{W, \eta(\mathcal{F})}} \circ (\brho \otimes \brho) \simeq F_{W, \mathcal{F}} \circ \varsigma_\otimes \circ (\id \otimes \id + \Psi \otimes \Phi).
\]
Indeed, the reader should check that applying an oppositely-oriented half-Dehn twist to each end of $\eta(\mathcal{F})$ gives a decoration isotopic to $sw(\tau_W (\mathcal{F}))$. Hence 
\[
\brho \circ {F_{W, \eta(\mathcal{F})}} \circ (\brho \otimes \brho) \simeq F_{W,sw(\tau_W (\mathcal{F}))}. 
\]
Applying formula (\ref{eq:swapequation}) for $F_{W,sw(\tau_W (\mathcal{F}))}$ and using the fact that $\rho$ and $\brho$ differ by an application of the Sarkar map, we obtain
\[
\rho \circ {F_{W, \eta(\mathcal{F})}} \circ (\brho \otimes \brho) \simeq \varsigma_{\#} \circ F_{W, \mathcal{F}} \circ (\id \otimes \id + \Psi \otimes \Phi).
\]
Here, $\varsigma_\#$ is the Sarkar map on the connected sum $K \# \tau K^r$. The fact that the Sarkar map can be computed algebraically shows that $\varsigma_\# \circ  F_{W, \mathcal{F}} \simeq F_{W, \mathcal{F}} \circ \varsigma_{\otimes}$, since $F_{W, \mathcal{F}}$ is an explicit homotopy equivalence which identifies $\CFK(K \# \tau K^r)$ with the tensor product $\CFK(K) \otimes \CFK(\tau K^r)$. For completeness, however, we include a more concrete topological proof in Lemma~\ref{lem:sarkarproduct} below. The desired claim follows. 

Finally, we show that turning $F_{W, \mathcal{F}}$ around constitutes a homotopy inverse to $F_{W, \mathcal{F}}$. To see that these are homotopy inverses, note that 
\[
F_{W,\mathcal{F}} \simeq q_K \circ F_{W, \mathcal{F}^{\prime}} \circ (\id \otimes \rho).
\]
Here, $q$ is a quarter-Dehn twist and $\mathcal{F}'$ is the decoration in Figure~\ref{fig:46}. 

\begin{figure}[h!]
\includegraphics[scale = 1]{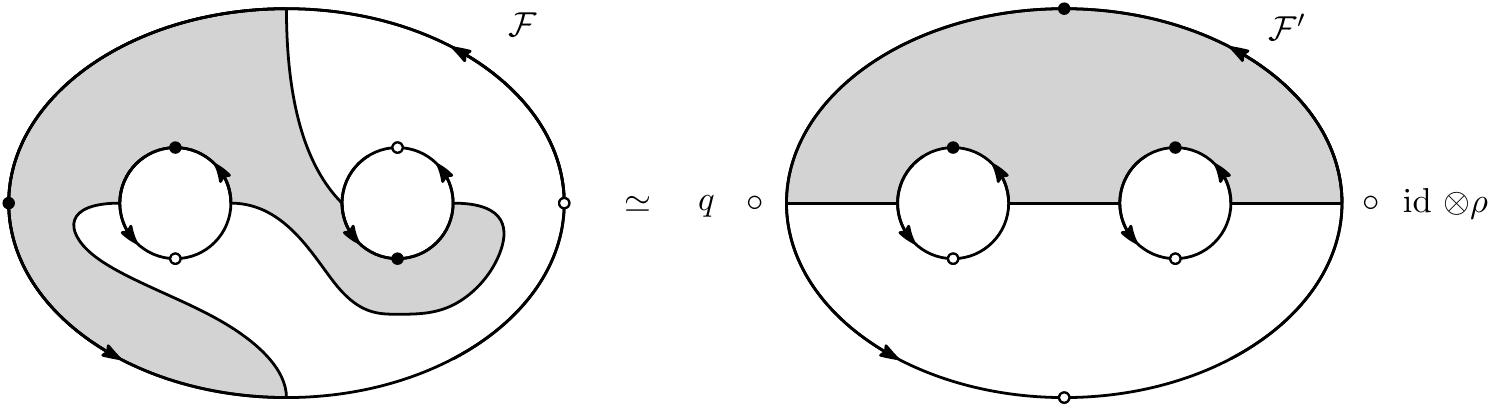}
\caption{Writing $\mathcal{F}$ in terms of $\mathcal{F}'$.}\label{fig:46}
\end{figure}

Writing $\bar{F}_{W,\mathcal{F}}$ for the reverse of $F_{W, \mathcal{F}}$, we thus have
\[
\bar{F}_{W,\mathcal{F}} \simeq (\id \otimes \brho) \circ \bar{F}_{W, \mathcal{F}^\prime} \circ \bar{q}_K.
\]
As in the proof of Theorem~\ref{thm:connectedsum}, $F_{W, \mathcal{F}^\prime}$ and $\bar{F}_{W, \mathcal{F}^\prime}$ are homotopy inverses. The claim follows.
\end{proof}

\begin{lemma}\label{lem:sarkarproduct}
With $\mathcal{F}$ as in the proof of Theorem~\ref{thm:swapping}, we have $\varsigma_{\#} \circ F_{W,\mathcal{F}} \simeq F_{W,\mathcal{F}} \circ \varsigma_{\otimes}$.
\end{lemma}

\begin{proof}
As in the proof of Theorem~\ref{thm:swapping}, write $F_{W,\mathcal{F}} \simeq q \circ F_{W,\mathcal{F^{\prime}}} \circ (\mathrm{id} \otimes \rho)$. Using the fact that $q$ and $\rho$ are induced by orientation-preserving diffeomorphisms, it is straightforward to check that $q \circ \varsigma_\# \simeq \varsigma_\# \circ q$ and $(\mathrm{id} \otimes \rho) \circ \varsigma_{\otimes} \simeq \varsigma_{\otimes} \circ (\mathrm{id} \otimes \rho)$. (For the latter, simply note that $\Phi$ and $\Psi$ commute with all such pushforward maps.) It thus suffices to prove the lemma with the decoration $\mathcal{F}'$ in place of $\mathcal{F}$. Applying the definition of $\varsigma_{\otimes}$, this reduces to showing
\begin{gather*}\label{sarkar_relation_2}
\Psi \Phi \circ F_{W,\mathcal{F}^{\prime}} \simeq F_{W,\mathcal{F}^{\prime}} \circ (\mathrm{id} \otimes \Phi\Psi + \Phi \Psi \otimes \mathrm{id} + \Phi \otimes \Psi + \Psi \otimes \Phi).
\end{gather*}

We repeatedly apply suitable bypass relations. First note that $\Psi \Phi \circ F_{W,\mathcal{F}^{\prime}}$ is the map associated to the decoration on the left in Figure~\ref{fig:47}. Applying the bypass relation to the disk on the left-hand side gives the two decorations shown on the right. We denote these by $\mathcal{F}_1$ and $\mathcal{F}_2$, respectively.

\begin{figure}[h!]
\includegraphics[scale = 0.35]{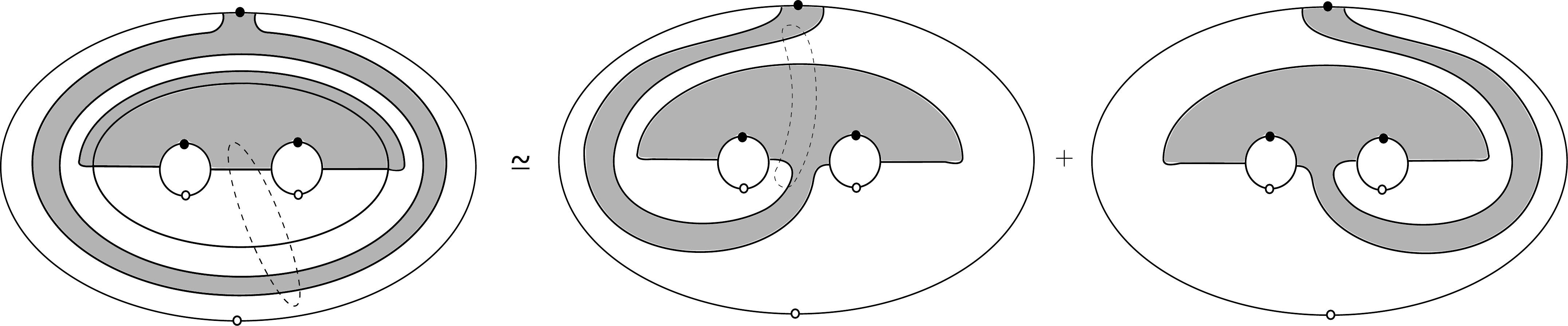}
\caption{Applying a bypass relation to the decoration associated to $\Psi \Phi \circ F_{W, \mathcal{F}^{\prime}}$ gives $\mathcal{F}_1$ (left) and $\mathcal{F}_2$ (right).}\label{fig:47}
\end{figure}

We then further apply a bypass relation to $\mathcal{F}_1$. Doing this for the disk on the right-hand side of Figure~\ref{fig:47} gives the two decorations shown in Figure~\ref{fig:48}, which we denote by $\mathcal{F}_3$ and $\mathcal{F}_4$. Note that $F_{W,\mathcal{F}_4} \simeq F_{W, \mathcal{F}^{\prime}} \circ (\Psi \Phi \otimes \mathrm{id})$. 

\begin{figure}[h!]
\includegraphics[scale = 0.4]{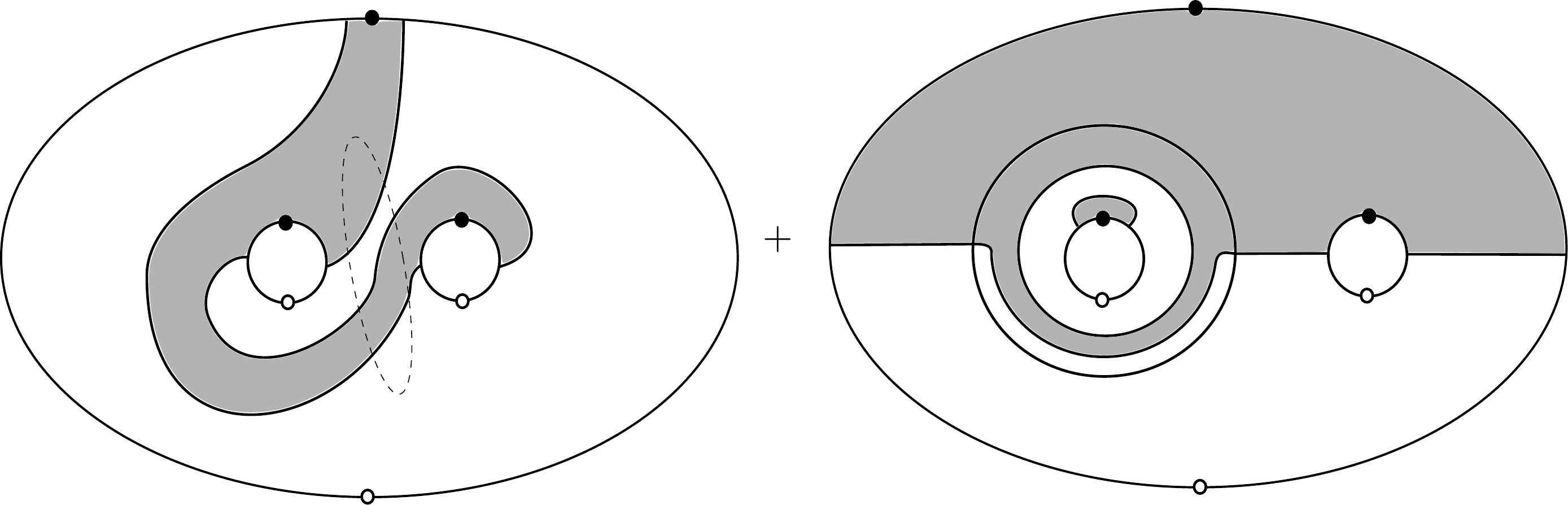}
\caption{Applying a bypass relation to $\mathcal{F}_1$ gives $\mathcal{F}_3$ (left) and $\mathcal{F}_4$ (right).}\label{fig:48}
\end{figure}

We now apply a final bypass relation to $\mathcal{F}_3$. Doing this for the disk indicated in Figure~\ref{fig:48} gives the two decorations shown in Figure~\ref{fig:49}, which we denote by $\mathcal{F}_5$ and $\mathcal{F}_6$. Note that $F_{W,\mathcal{F}_5} \simeq F_{W,\mathcal{F}^{\prime}} \circ (\Psi \otimes \Phi)$, while $\mathcal{F}_6$ is just $\mathcal{F}'$.

\begin{figure}[h!]
\includegraphics[scale = 0.4]{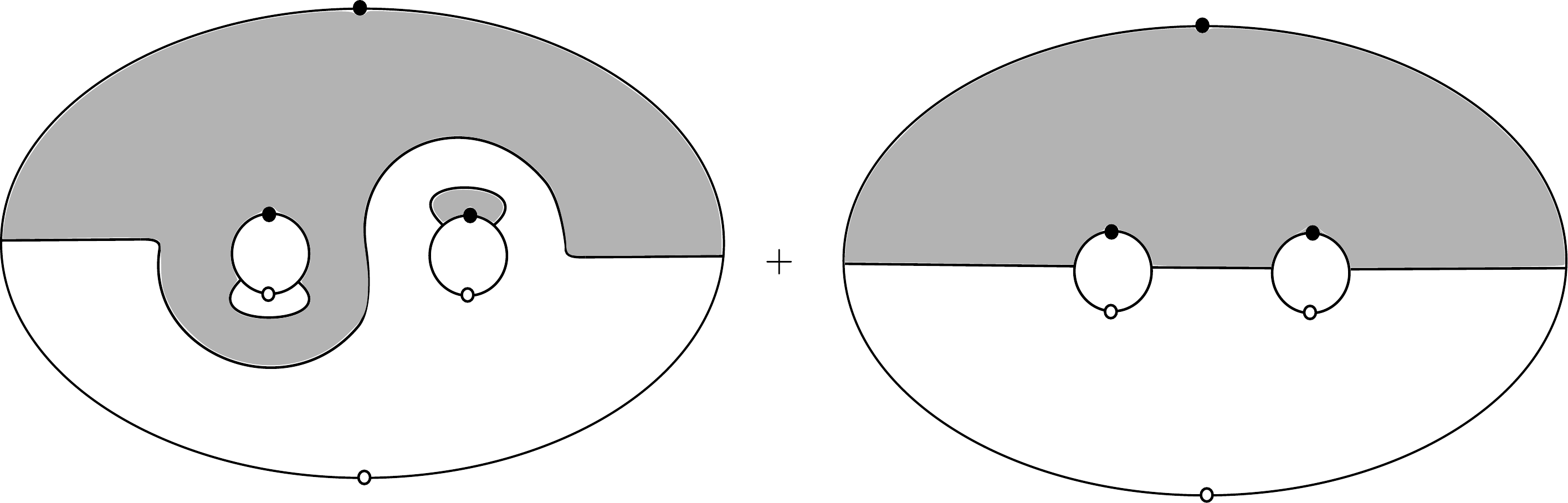}
\caption{Applying a bypass relation to $\mathcal{F}_3$ gives $\mathcal{F}_5$ (left) and $\mathcal{F}_6$ (right).}\label{fig:49}
\end{figure}

Putting the results of Figures~\ref{fig:48} and \ref{fig:49} together, we have that
\begin{gather*}\label{sarkar_relation_3}
F_{W,\mathcal{F}_1} \simeq  F_{W,\mathcal{F}^{\prime}} \circ (\Psi \otimes \Phi) + F_{W,\mathcal{F}^{\prime}} + F_{W, \mathcal{F}^{\prime}} \circ (\Psi \Phi \otimes \mathrm{id}).
\end{gather*}
Now, note that in Figure~\ref{fig:47}, the decorations $F_{W,\mathcal{F}_1}$ and $F_{W,\mathcal{F}_2}$ are related by reflection across the vertical axis. By applying similar bypass relations to $\mathcal{F}_2$ (using the reflections of the disks for $\mathcal{F}_1$) we obtain
\begin{gather*}\label{sarkar_relation_4}
F_{W,\mathcal{F}_2} \simeq  F_{W,\mathcal{F}^{\prime}} \circ (\Phi \otimes \Psi)  + F_{W,\mathcal{F}^{\prime}} + F_{W, \mathcal{F}^{\prime}} \circ ( \mathrm{id} \otimes \Psi \Phi).
\end{gather*}
Adding these two relations together and using the fact that $\Psi$ and $\Phi$ homotopy commute gives the desired result.
\end{proof}


\section{Equivariant slice genus bounds}\label{sec:5}

We now prove Theorem~\ref{thm:1.2}. This closely follows \cite[Theorem 1.7]{JZgenus}.

\begin{proof}[Proof of Theorem~\ref{thm:1.2}]
Let $(K, \tau)$ be a strongly invertible knot, which may be neither directed nor decorated. Let $\Sigma$ be an isotopy-equivariant slice surface for $(K, \tau)$ in some homology ball $(W, \tau_W)$. We may assume $\tau_W$ acts as $\tau \times \id$ on some collar neighborhood $(\partial W) \times I$. We may furthermore assume that the isotopy from $\tau_W (\Sigma)$ to $\Sigma$ does not move this collar neighborhood of $\partial W$ and that $\Sigma$ is exactly equivariant near $\partial W$. Hence we can puncture $\Sigma$ at some fixed point of $\tau_W$ near $\partial W$ and treat $\Sigma$ and $\tau_W(\Sigma)$ as isotopy-equivariant knot cobordisms from the unknot (with the obvious strong inversion) to $(K, \tau)$. See Figure~\ref{fig:51}.

For reasons that will be clear presently, it will be convenient for us to stabilize $\Sigma$ a certain number of times. If the genus of $\Sigma$ is even, then we stabilize $\Sigma$ twice; if the genus of $\Sigma$ is odd, then we stabilize $\Sigma$ once. We denote the stabilized surface by $\Sigma'$; note that the genus of $\Sigma'$ is even. We carry out the stabilization equivariantly near $\partial W$, so that $\Sigma'$ is still isotopic to $\tau_W(\Sigma')$ rel $K$. See Figure~\ref{fig:51}.

\begin{figure}[h!]
\includegraphics[scale = 0.8]{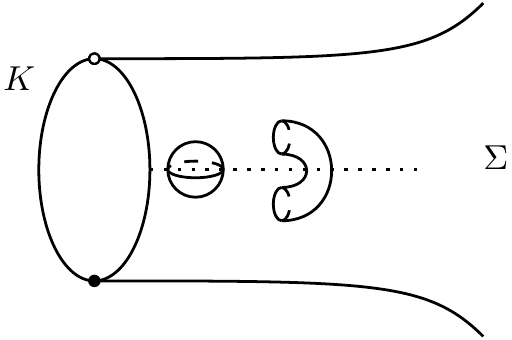}
\caption{We may assume that $\Sigma$ is exactly equivariant near $\partial \Sigma = K \subseteq \partial W$. We have chosen an arc of fixed points lying on $\Sigma$; this is represented by the dotted line. Puncturing $W$ at a point on this dotted line gives an isotopy-equivariant knot cobordism from the unknot to $K$. This is schematically represented by cutting out the sphere indicated in the figure. Stabilizing $\Sigma$ is represented by the 1-handle with feet near the dotted line.}\label{fig:51}
\end{figure}

Now fix any pair of dividing arcs $\mathcal{F}$ on $\Sigma'$ such that the resulting black and white regions have equal genus. Note that this implicitly fixes a decoration on the ends of $\Sigma'$, but due to Lemma~\ref{lem:twistnochange} this choice of decoration on $K$ does not affect the statement of the theorem. Consider the knot cobordism map $F_{W, \mathcal{F}}$. We have the usual commutative diagram
\[\begin{tikzcd}
	{\mathcal{CFK}(U) } && {\mathcal{CFK}(K)} \\
	\\
	{\mathcal{CFK}(U^r) } && {\mathcal{CFK}(K^r)} \\
	\\
	{\mathcal{CFK}(U) } && {\mathcal{CFK}(K)}
	\arrow["{F_{W,\mathcal{F}}}", from=1-1, to=1-3]
	\arrow["{t}"', from=1-1, to=3-1]
	\arrow["{t}", from=1-3, to=3-3]
	\arrow["{F_{W,\tau_W(\mathcal{F})}}", from=3-1, to=3-3]
	\arrow["sw"', from=3-1, to=5-1]
	\arrow["sw", from=3-3, to=5-3]
	\arrow["{F_{W,sw(\tau_W(\mathcal{F}))}}", from=5-1, to=5-3]
\end{tikzcd}\]
where we have suppressed the choice of basepoints. Importantly, we have \textit{not} assumed that $\Sigma$ (or $\Sigma'$) is isotopy-equivariant in the decorated sense. Hence although $\tau_W(\Sigma')$ is isotopic to $\Sigma'$, it is not true that the image of the decoration $\tau_W(\mathcal{F})$ under this isotopy must coincide with the decoration $sw(\mathcal{F})$. Indeed, in general we might obtain a completely different decoration on $\Sigma'$. We thus instead invoke \cite[Proposition 5.5]{JZgenus}. This states that if $\Sigma'$ is any stabilized surface and $\mathcal{F}^A$ and $\mathcal{F}^B$ are any two sets of dividing curves (each consisting of a pair of dividing arcs) on $\Sigma'$ with $\chi(\mathcal{F}^A_w) = \chi(\mathcal{F}^B_w)$ and $\chi(\mathcal{F}^A_z) = \chi(\mathcal{F}^B_z)$, then
\[
[F_{W, \mathcal{F}^A}(1)] = [F_{W, \mathcal{F}^B}(1)].
\]
Hence in our case $F_{W, \mathcal{F}}$ and $F_{W,sw(\tau_W(\mathcal{F}))}$ are chain homotopic. This shows that $F_{W, \mathcal{F}}$ induces a $\tau_K$-equivariant map from the trivial complex of the unknot to $\CFK(K)$. Our argument here is almost identical to that of \cite[Theorem 1.7]{JZgenus}; there, the authors show that $F_{W, \mathcal{F}}$ is $\iota_K$-equivariant (up to homotopy).

The map $F_{W, \mathcal{F}}$ has grading shift $(-g(\Sigma'), -g(\Sigma'))$. We thus obtain a map from the trivial complex $\F[U]$ to the large surgery complex $\Co$ of $\CFK(K)$. This lowers grading by $g(\Sigma')$ and is a homotopy equivalence after inverting $U$. It follows that 
\[
\underline{d}_{\tau}(\Co) \geq -g(\Sigma').
\]
This gives the inequality 
\[
\Vtl(K) \leq \left \lceil{\frac{1+ g(\Sigma)}{2}} \right \rceil
\]
keeping in mind the number of stabilizations relating $\Sigma$ to $\Sigma'$. The same argument, together with the fact that $F_{W, \mathcal{F}}$ also homotopy commutes with $\iota_K$, gives the desired inequality for $\Vitl(K)$. The other claims of the theorem are obtained by turning the cobordism around and reversing orientation.
\end{proof}

\begin{remark}\label{rem:weakerbound}
The reader may be confused as to why Theorem~\ref{thm:1.2} is weaker than Theorem~\ref{thm:1.1} in the genus-zero case. This is because in the proof of Theorem~\ref{thm:1.2}, we stabilize in order to deal with the possible non-equivariance of the dividing curves. 
However, in the genus-zero case, no stabilizations are actually necessary. Following the more specialized proof of Theorem~\ref{thm:1.2} in this situation gives the same conclusion as from Theorem~\ref{thm:1.1}.
\end{remark}

\section{Torus knots}\label{sec:6}

We now bound the invariants $\Vtu$ and $\Vtl$ for the strongly invertible knots $(K_n, \tau_n)$ from the introduction. Recall that $K_n$ is constructed by taking the connected sum of $(T_{2n, 2n+1} \# T_{2n, 2n+1}, \tau_\#)$ with the mirror of $(T_{2n, 2n+1} \# T_{2n, 2n+1}, \tau_{sw})$. We deal with each one of these two factors in turn. Throughout, let $n$ be odd.

\subsection{The connected sum involution} Using the connected sum formula, we first compute the $\tau_K$-complex of $(T_{2n, 2n+1} \# T_{2n, 2n+1}, \tau_\#)$. For this, we need to know the $\tau_K$-complex of $(T_{2n, 2n+1}, \tau)$, where $\tau$ is the unique strong inversion on $T_{2n, 2n+1}$. 


\begin{definition}\label{def:Cn}
Let $\cC_n$ be the staircase complex associated to the parameter sequence
\[
(c_{-2n+1},c_{-2n+2},\dots, c_{2n-2},c_{2n-1}) = (1, 2n-1, 2, 2n-2, \dots, 2n-2, 2, 2n-1, 1).  
\]
This is displayed in Figure~\ref{fig:Cn}. Explicitly, $\cC_n$ is generated by the elements
\begin{align*}
    & x_k \quad \text{for } -2n+2 \leq k \leq 2n-2 \text{ and } k \text{ even; and} \\
    & y_\ell \quad \text{for } -2n+1 \leq \ell \leq 2n-1 \text{ and } \ell \text{ odd}
\end{align*}
and has nonzero differentials given by
\[
\partial x_{k} = \scV^{c_{k-1}}y_{k-1} + \scU^{c_{k+1}}y_{k+1}.
\]

\begin{figure}[h!]
	\begin{tikzcd}[column sep={1cm,between origins}, labels=description, row sep=1cm]
	y_{-2n+1}&& y_{-2n+3}&& \cdots&& y_{-1}&& y_{1}&&\cdots&& y_{2n-3}&&y_{2n-1}\\
	&x_{-2n+2}\ar[ul, "\scV"]\ar[ur, "\scU^{2n-1}"]&&x_{-2n+4}\ar[ul, "\scV^2"] \ar[ur, "\scU^{2n-2}"]&& \cdots\ar[ur] \ar[ul]  && x_{0} \ar[ul, "\scV^{n}"] \ar[ur, "\scU^{n}"]&& \cdots \ar[ur] \ar[ul]  && x_{2n-4} \ar[ul, "\scV^{2n-2}"] \ar[ur, "\scU^{2}"]&& x_{2n-2} \ar[ul, "\scV^{2n-1}"] \ar[ur, "\scU"]
	\end{tikzcd}
	\caption{The complex $\cC_n$. See \cite[Figure 3.1]{HHSZ2}.}\label{fig:Cn}
\end{figure}
\noindent
Together with the definition of $\partial$, the convention that $\grV(y_{-2n+1}) = \grU(y_{2n-1}) = 0$ determines the gradings of all of the generators of $\cC_n$. It will be helpful for us to explicitly record:
\begin{align}\label{eq:gradingsI}
\begin{split}
\grU(y_{2n - 1 - 2i}) &= - 2(1 + 2 + \cdots + i) \\
\grV(y_{-2n + 1 + 2i}) &= - 2(1 + 2 + \cdots + i).
\end{split}
\end{align}
There is a unique skew-graded homotopy involution on $\cC_n$, which is given by
\begin{align*}
    &\tau(x_k) = x_{-k} \\
    &\tau(y_\ell) = y_{-\ell}.
\end{align*}
\end{definition}

In \cite[Proposition 3.1]{HHSZ2}, it is shown that the knot Floer complex of $T_{2n, 2n+1}$ is homotopy equivalent to $\cC_n$. By Theorem~\ref{thm:1.8}, we know that $\tau_K$ is a skew-graded homotopy involution on $\CFK(T_{2n, 2n+1})$. Thus the $\tau_K$-complex of $(T_{2n, 2n+1}, \tau)$ is given by $(\cC_n, \tau)$, with $\tau$ as above. Applying Theorem~\ref{thm:connectedsum}, we conclude that the $\tau_K$-complex of $(T_{2n, 2n+1} \# T_{2n, 2n+1}, \tau_{\#})$ is homotopy equivalent to $(\cC_n \otimes \cC_n, \tau \otimes \tau)$. The goal for this subsection will be to extract a usable representative of this local equivalence class. Our computations here are similar to those of \cite[Section 3]{HHSZ2}. 

\begin{definition}\label{def:Dn}
Let $\cD_n$ be the staircase complex associated to the parameter sequence
\begin{align*}
&(d_{-4n+2}, d_{-4n+3}, \ldots,d_{4n-3}, d_{4n-2}) = \\ 
&\quad (1,2n-1,1,2n-1,2,2n-2,2,2n-2,3,\dots,2n-2,2,2n-2,2,2n-1,1,2n-1,1).
\end{align*}
This is displayed in the upper half of Figure~\ref{fig:En}. Explicitly, $\cD_n$ is generated by the elements
\begin{align*}
    & w_k \quad \text{for } -4n+3 \leq k \leq 4n-3 \text{ and } k \text{ odd; and} \\
    & z_\ell \quad \text{for } -4n+2 \leq \ell \leq 4n-2 \text{ and } \ell \text{ even}
\end{align*}
and has nonzero differentials given by
\[
\partial w_k = \scV^{d_{k-1}}z_{k-1}+\scU^{d_{k+1}}z_{k+1}.
\]
Like $\cC_n$, the complex $\cD_n$ has a unique skew-graded homotopy involution, which we again denote by $\tau$. It will also be useful to consider the square complex $\cS_n$, which is displayed in the lower half of Figure~\ref{fig:En}. 

\begin{figure}[h]
	\[
	\begin{tikzcd}[column sep={1.05cm,between origins}, labels=description, row sep=1cm]
	z_{-4n+2} && z_{-4n+4}&& z_{-4n+6}&& z_{-4n+8}&& \cdots&& z_{-4}&& z_{-2}&&\,\\
	& w_{-4n+3} \ar[ul, "\scV"]\ar[ur, "\scU^{2n-1}"]&& w_{-4n+5}\ar[ul, "\scV"]\ar[ur, "\scU^{2n-1}"]&& w_{-4n+7} \ar[ul, "\scV^2"]\ar[ur, "\scU^{2n-2}"]&& \cdots \ar[ur] \ar[ul] && w_{-5}\ar[ul]\ar[ur, "\scU^{n+1}"] && w_{-3} \ar[ul, "\scV^{n-1}"]\ar[ur, "\scU^{n+1}"] && w_{-1} \ar[ul, "\scV^n"] \ar[ur, "\scU^n"]
	\end{tikzcd} \cdots
	\]
	\[
	\begin{tikzcd}[column sep={1cm,between origins}, labels=description, row sep=1cm]
	z_0&& z_2&& z_4&& z_6&& \cdots&& z_{4n-4}&& z_{4n-2}&&\,\\
	& w_{1} \ar[ul, "\scV^n"]\ar[ur, "\scU^n"]&& w_{3}\ar[ul, "\scV^{n+1}"]\ar[ur, "\scU^{n-1}"]&&w_{5}  \ar[ul, "\scV^{n+1}"]\ar[ur, "\scU^{n-1}"]&& \cdots \ar[ul] \ar[ur] && w_{4n-5}\ar[ul]\ar[ur, "\scU"] && w_{4n-3} \ar[ul, "\scV^{2n-1}"] \ar[ur, "\scU"]
	\end{tikzcd}
	\]
	\[
	\begin{tikzcd}[column sep={1.55cm,between origins}, labels=description, row sep=1cm]
	&t&\,\\
	r_{-1} \ar[ur, "\scU^n"]&& r_1 \ar[ul, "\scV^n"]\\
	& r_0\ar[ul, "\scV^n"]\ar[ur, "\scU^n"]
	\end{tikzcd}
	\]
	\caption{The complex $\cD_n \oplus \cS_n$.} \label{fig:En}
\end{figure}
\noindent
This is generated by $\{r_0,r_{-1},r_1,t\}$, with differential
\[
\partial r_0=\scV^n r_{-1} +\scU^n r_1, \quad \partial r_{-1}=\scU^n t, \quad \partial r_1=\scV^n t ,\quad \text{and} \quad \partial t=0.
\]
Using the fact that $\grV(z_{-4n+2}) = \grU(z_{4n-2}) = 0$, we again have (for example)
\begin{align}\label{eq:gradingsII}
\begin{split}
&\grU(z_{4n - 2 - 2i}) = 
\begin{cases}
- 4(1 + 2 + \cdots + i/2) \text{ for } i  \text{ even} \\
- 4(1 + 2 + \cdots + (i-1)/2) - (i+1) \text{ for } i \text{ odd} 
\end{cases}\\
&\grV(z_{-4n + 2 + 2i}) = 
\begin{cases}
- 4(1 + 2 + \cdots + i/2) \text{ for } i  \text{ even} \\
- 4(1 + 2 + \cdots + (i-1)/2) - (i+1) \text{ for } i \text{ odd}.
\end{cases}
\end{split}
\end{align}
The grading on $\cS_n$ is such that $\gr(t) = \gr(z_0)$. Note that $\gr(r_{-1}) = \gr(w_{-1})$ and $\gr(r_1) = \gr(w_1)$. 
\end{definition}

\begin{definition}\label{def:En}
Let $\cE_n = \cD_n \oplus \cS_n$. Define an involution $\tau$ on $\cE_n$ as follows. On $\cD_n$, we define $\tau$ to be almost the same as in Definition~\ref{def:Dn}, but slightly different on $w_{-1}$, $w_1$, and $z_0$. On $\cS_n$, we define $\tau$ to be the obvious reflection map.
\begin{equation*}
\begin{split}
&\tau(w_k) = w_{-k} \quad \text{for } k \neq -1, 1 \\
&\tau(w_{-1}) = w_1 + r_1 \\
&\tau(w_1) = w_{-1} + r_{-1} \\
&\tau(z_\ell) = z_{-\ell} \quad \text{for } \ell \neq 0 \\ 
&\tau(z_0) = z_0 + t.
\end{split}
\quad\quad\quad
\begin{split}
&\tau(r_i) = r_{-i} \\
&\tau(t) = t.
\end{split}
\end{equation*}
Roughly speaking, $\tau$ acts as reflection on the staircase and the square, but additionally maps some of the staircase generators to (sums of staircase generators with) square generators. Unlike the action of $\iota_K$ in \cite[Section 3.2.1]{HHSZ2}, however, none of the square generators are mapped to staircase generators.
\end{definition}

The main claim of this subsection is that $(\cE_n, \tau)$ is locally equivalent to $(\cC_n \otimes \cC_n, \tau \otimes \tau)$. We show this by constructing local maps in both directions. The forward direction is straightforward from the work of \cite{HHSZ2}. In what follows, we write $\leq$ to indicate the presence of a local map from one $\tau_K$-complex to another; see also the discussion of Section~\ref{sec:7.2}.


\begin{lemma}\label{lem:CnCnI}
We have $(\cE_n, \tau) \leq (\cC_n \otimes \cC_n, \tau \otimes \tau)$.
\end{lemma}
\begin{proof}
In \cite[Section 3.2.1]{HHSZ2}, it is shown that $\cC_n \otimes \cC_n$ admits the subcomplex $\cY_n$ displayed in Figure~\ref{fig:Yn}. Explicitly, $\cY_n$ is spanned by 
\[
\{y_{i}y_{i}\}\cup \{y_{i}y_{i+2}\}_{i\leq -3}\cup \{y_{i+2}y_{i}\}_{i\geq -1},
\]
together with
\[
\{y_ix_{i+1}\}_{i\leq -3} \cup \{ x_iy_{i+1}\}_{i\leq -2} \cup \{x_iy_{i-1}\}_{i\geq 0} \cup \{ y_ix_{i-1}\}_{i\geq 1},
\]
and
\[
y_1y_{-1}+y_{-1}y_1, \quad y_{-1}x_0+x_0y_{-1}, \quad y_1x_0+x_0y_1,\quad \text{and} \quad x_0x_0.
\]
The first two collections of generators span a staircase complex, while the last four generators span a square complex.

\begin{figure}[h]
	\[
	\begin{tikzcd}[column sep={1.05cm,between origins}, labels=description, row sep=1cm]
	y_{1-2n}y_{1-2n}&& y_{1-2n}y_{3-2n}&& y_{3-2n}y_{3-2n}&& y_{3-2n}y_{5-2n}&& \cdots&& y_{-3}y_{-1}&& y_{-1} y_{-1}&&\,\\
	& y_{1-2n} x_{2-2n} \ar[ul, "\scV"]\ar[ur, "\scU^{2n-1}"]&& x_{2-2n}y_{3-2n}\ar[ul, "\scV"]\ar[ur, "\scU^{2n-1}"]&& y_{3-2n}x_{4-2n} \ar[ul, "\scV^2"]\ar[ur, "\scU^{2n-2}"]&& \cdots \ar[ur] \ar[ul] && y_{-3}x_{-2}\ar[ul]\ar[ur, "\scU^{n+1}"] && x_{-2} y_{-1} \ar[ul, "\scV^{n-1}"]\ar[ur, "\scU^{n+1}"] && x_{0}y_{-1} \ar[ul, "\scV^n"] \ar[ur, "\scU^n"]
	\end{tikzcd} \cdots
	\]
	\[
	\begin{tikzcd}[column sep={1.05cm,between origins}, labels=description, row sep=1cm]
	y_{1}y_{-1}&& y_{1}y_{1}&& y_{3}y_{1}&& y_{3}y_{3}&& \cdots&& y_{2n-1}y_{2n-3}&& y_{2n-1}y_{2n-1}&&\,\\
	& y_{1}x_{0} \ar[ul, "\scV^n"]\ar[ur, "\scU^n"]&& x_{2}y_{1}\ar[ul, "\scV^{n+1}"]\ar[ur, "\scU^{n-1}"]&&y_{3}x_{2}  \ar[ul, "\scV^{n+1}"]\ar[ur, "\scU^{n-1}"]&& \cdots \ar[ul] \ar[ur] && x_{2n-2}y_{2n-3}\ar[ul]\ar[ur, "\scU"] && y_{2n-1}x_{2n-2} \ar[ul, "\scV^{2n-1}"] \ar[ur, "\scU"]
	\end{tikzcd}
	\]
	\[
	\begin{tikzcd}[column sep={1.55cm,between origins}, labels=description, row sep=1cm]
	&y_{1}y_{-1}+y_{-1}y_1&\,\\
	y_{-1}x_{0}+x_{0}y_{-1} \ar[ur, "\scU^n"]&& y_1x_{0}+x_{0}y_1 \ar[ul, "\scV^n"]\\
	& x_{0}x_{0}\ar[ul, "\scV^n"]\ar[ur, "\scU^n"]
	\end{tikzcd}
	\]
	\caption{The subcomplex $\cY_n$. Note that the top two rows form a staircase complex, such that $\partial(x_0y_{-1})=\scV^n y_{-1}y_{-1}+\scU^n y_1y_{-1}$. See \cite[Figure 3.3]{HHSZ2}.} \label{fig:Yn}
\end{figure}

There is an obvious map $\varphi\colon \cE_n\to \cC_n\otimes \cC_n$ given by mapping $\cE_n$ isomorphically onto $\cY_n$; compare Figures~\ref{fig:En} and \ref{fig:Yn}. It is straightforward to check that this has the requisite behavior under localization: observe that $y_{1-2n}y_{1-2n}$ is nontorsion. To check equivariance, recall that $\tau(x_k) = x_{-k}$ and $\tau(y_\ell) = y_{-\ell}$. An examination of Figure~\ref{fig:Yn} shows that $\tau \otimes \tau$ acts as reflection on $\cY_n$, except at the generators
\begin{align*}
&(\tau \otimes \tau)(x_0y_{-1}) = x_0y_1 = y_1x_0 + (y_1x_0 + x_0y_1) \\
&(\tau \otimes \tau)(y_1x_0) = y_{-1}x_0 = x_0y_{-1} + (y_{-1}x_0 + x_0y_{-1}) \\
&(\tau \otimes \tau)(y_1y_{-1}) = y_{-1}y_1 = y_1y_{-1} + (y_1y_{-1} + y_{-1}y_1).
\end{align*}
This coincides exactly with the action of $\tau$ on $\cE_n$.
\end{proof}	

We now construct a map from $\cE_n^\vee$ to $\cC_n^\vee \otimes \cC_n^\vee$. It will be helpful for us to first discuss some auxiliary lemmas regarding the dual staircase complexes $\cC_n^\vee$ and $\cD_n^\vee$. Our first lemma concerns elements in $\cC_n^\vee \otimes \cC_n^\vee$ of the form $x_p^\vee \otimes y_q^\vee$. Roughly speaking, we claim that if the value of $p + q$ is fixed, then the grading of $x_p^\vee \otimes y_q^\vee$ is minimized when the difference $|p - q|$ is minimized. Similar statements hold for elements of the form $y_p^\vee \otimes y_q^\vee$. We make this more explicit by introducing the following terminology:
\begin{definition}\label{def:minimizing}
Let $k$ be odd and let $p + q = k$ with $p$ even and $q$ odd. We call $(p, q)$ \textit{difference-minimizing} in the following situations:
\begin{enumerate}
\item $k \equiv 1 \bmod 4$: we require $p = (k-1)/2$ and $q = (k+1)/2$.
\item $k \equiv 3 \bmod 4$: we require $p = (k+1)/2$ and $q = (k-1)/2$.
\end{enumerate}
Let $\ell$ be even and let $p + q = \ell$ with $p$ and $q$ both odd. We call $(p, q)$ \textit{difference-minimizing} in the following situations:
\begin{enumerate}
\item $\ell \equiv 2 \bmod 4$: we require $p = q = \ell/2$.
\item $\ell \equiv 0 \bmod 4$: we require $p = (\ell-1)/2$ and $q = (\ell+1)/2$, or vice-versa.
\end{enumerate}
In each case, note that the difference $|p -q|$ is minimized, subject to the constraints on the parity of $p$ and $q$ and the condition that the value of $p + q$ is fixed. The distinction between $k$ and $\ell$ is due to our choice of notation for the generators of $\cD_n$, and will become clear presently.
\end{definition}

\begin{lemma}\label{lem:minimizing}
Let $k$ be odd. Then
\[
\min_{\substack{p + q = k \\ p \text{ even, }q \text{ odd}}}\{\grU(x^\vee_p \otimes y^\vee_q)\} \quad \text{and} \quad \min_{\substack{p + q = k \\ p \text{ even, }q \text{ odd}}}\{\grV(x^\vee_p \otimes y^\vee_q)\}
\]
both occur when $(p, q)$ is difference-minimizing. Similarly, let $\ell$ be even. Then 
\[
\min_{\substack{p + q = \ell \\ p \text{ odd, }q \text{ odd}}}\{\grU(y^\vee_p \otimes y^\vee_q)\} \quad \text{and} \quad \min_{\substack{p + q = \ell \\ p \text{ odd, }q \text{ odd}}}\{\grV(y^\vee_p \otimes y^\vee_q)\}
\]
both occur when $(p, q)$ is difference-minimizing.
\end{lemma}
\begin{proof}
First note that for any $i$, we have:
	\begin{align*}
	\grU(x_{i+1}^\vee)-\grU(x_{i-1}^\vee)&=-2n-1+i\\
		\grV(x^\vee_{i+1})-\grV(x^\vee_{i-1})&=2n+1+i,\end{align*}
	and
	\begin{align*}
	\grU(y_{i+2}^\vee)-\grU(y^\vee_i)&=-2n+1+i\\
\grV(y_{i+2}^\vee)-\grV(y^\vee_i)&=2n+1+i.
\end{align*}
These claims are verified using the differentials in the definition of $\cC_n$; the reader may find it helpful to consult Figure~\ref{fig:Cn}.

Consider the first claim of the lemma. Observe
\begin{align*}
\grU(x^\vee_{p+2} \otimes y^\vee_{q - 2}) &= \grU(x^\vee_{p} \otimes y^\vee_{q}) + (-2n-1+ (p+1)) - (-2n + 1 + (q-2)) \\
&= \grU(x^\vee_{p} \otimes y^\vee_{q}) + p - q + 1.
\end{align*}
Similarly, we have
\begin{align*}
\grV(x^\vee_{p+2} \otimes y^\vee_{q - 2}) &= \grV(x^\vee_{p} \otimes y^\vee_{q}) + (2n+1+ (p+1)) - (2n + 1 + (q-2)) \\
&= \grU(x^\vee_{p} \otimes y^\vee_{q}) + p - q + 3.
\end{align*}
Note that due to the parity constraints on $p$ and $q$ and the fact that $p + q = k$, the value of $p - q$ is fixed modulo $4$. Treating both of the above as finite-difference equations, it is clear that to minimize both $\grU(x^\vee_{p} \otimes y^\vee_{q})$ and $\grV(x^\vee_{p} \otimes y^\vee_{q})$ we are searching for $(p, q)$ such that $p - q + 1$ and $p - q + 3$ are both in $[0, 4]$. An examination of Definition~\ref{def:minimizing} gives the claim for $k$ odd. The claim for $\ell$ even is established in an analogous manner.
\end{proof}

The following lemma relates the gradings of elements of $\cD_n^\vee$ and elements of $\cC_n^\vee \otimes \cC_n^\vee$, and will be important for constructing a map from the former into the latter.

\begin{lemma}\label{lem:samegrading}
Let $k$ be odd and let $p + q = k$ with $p$ even and $q$ odd. If $(p, q)$ is difference-minimizing, then
\[
\gr(w^\vee_k) = \gr(x^\vee_p \otimes y^\vee_q) = \gr(y^\vee_q \otimes x^\vee_p).
\]
Let $\ell$ be even and let $p + q = \ell$ with $p$ and $q$ both odd. If $(p, q)$ is difference-minimizing, then
\[
\gr(z^\vee_\ell) = \gr(y^\vee_p \otimes y^\vee_q) = \gr(y^\vee_q \otimes y^\vee_p).
\]
\end{lemma}
\begin{proof}
We prove the second claim and leave the first to the reader. Assume $(p, q)$ is difference-minimizing. Write $\ell = 4n - 2 - 2i$, $p = 2n - 1 - 2r$, and $q = 2n - 1- 2s$; note that $r + s = i$. From (\ref{eq:gradingsI}) and (\ref{eq:gradingsII}), we have
\[
\grU(z^\vee_\ell) = - \grU(z_\ell)
\begin{cases}
4(1 + 2 + \cdots + i/2) \text{ for } i  \text{ even} \\
4(1 + 2 + \cdots + (i-1)/2) + (i+1) \text{ for } i \text{ odd} 
\end{cases}
\]
and
\[
\grU(y_p^\vee) = - \grU(y_p) = 2(1 + 2 + \cdots + r) \quad \text{and} \quad \grU(y_q^\vee) = - \grU(y_q) = 2(1 + 2 + \cdots + s).
\]
Suppose $\ell \equiv 2 \bmod 4$. Then $i$ is even, and an examination of Definition~\ref{def:minimizing} shows $r = s = i/2$. In this case, we clearly have $\grU(z^\vee_\ell) = \grU(y_p^\vee) + \grU(y_q^\vee) = \grU(y_p^\vee \otimes y_q^\vee)$. Suppose $\ell \equiv 0 \bmod 4$. Then $i$ is odd, and we have $r = (i+1)/2$ and $s = (i-1)/2$ (or vice-versa). An inspection of the equalities above once again gives the claim. An analogous argument for $\grV$ completes the proof.
\end{proof}

We now establish the major claim of this subsection:

\begin{lemma}\label{lem:CnCnII}
We have $(\cE_n^\vee, \tau^\vee) \leq (\cC_n^\vee \otimes \cC_n^\vee, \tau^\vee \otimes \tau^\vee)$.
\end{lemma}

\begin{proof}
We define a grading-preserving map $\psi\colon \cE^\vee_n \to \cC_n^\vee \otimes \cC_n^\vee$ as follows. For any $\ell$ even, let
\[
\psi(z_\ell^\vee) = \sum_{i+j=\ell} \scU^* \scV^* y_i^\vee\otimes y_j^\vee.
\]
Here, the right-hand side is formed by considering all possible products $y_i^\vee \otimes y_j^\vee$ with $i$ and $j$ odd and $i + j = \ell$. Each term is multiplied by powers of $\scU$ and $\scV$ so that the resulting grading is equal to that of $z_\ell^\vee$. Note that this is possible due to Lemmas~\ref{lem:minimizing} and \ref{lem:samegrading}. Indeed, by Lemma~\ref{lem:minimizing}, $\gr(y_i^\vee \otimes y_j^\vee)$ is minimized when $(i, j)$ is difference-minimizing. We then multiply every other term on the right-hand side by powers of $\scU$ and $\scV$ so as to have grading equal to this minimal grading. But by Lemma~\ref{lem:samegrading}, the minimal grading is none other than $\gr(z_\ell^\vee)$. 

We similarly define:
	\begin{align*}
	&\psi(w_k^\vee) = 
	\sum_{i+j=k} \scU^*\scV^*(x_i^\vee\otimes y_j^\vee+y_j^\vee \otimes x_i^\vee)\\
	&\psi(r_{-1}^\vee) = \sum_{i+j=0, \ i<j}\scU^*\scV^* x_{i-1}^\vee \otimes y_j^\vee + \scU^* \scV^* y_i^\vee \otimes x_{j-1}^\vee\\
	&\psi(r_{1}^\vee) = \sum_{i+j=0,\ i<j} \scU^*\scV^* x_{i+1}^\vee\otimes y_j^\vee+ \scU^* \scV^* y_i^\vee\otimes x_{j+1}^\vee\\
	&\psi(r_0^\vee) = \vphantom{\sum_{i+j=0,\ i<j}} x_0^\vee\otimes x_0^\vee \\
	&\psi(t^\vee) = \sum_{i+j=0,\ i<j} \scU^* \scV^* y_i^\vee\otimes y_j^\vee
	\end{align*}

\noindent
As before, Lemmas~\ref{lem:minimizing} and \ref{lem:samegrading} guarantee that in each of the above equations, there exist unique powers of $\scU$ and $\scV$ which make $\psi$ grading-preserving. Note that $\gr(t^\vee) = \gr(z_0^\vee)$, while $\gr(r_{-1}^\vee) = \gr(w_{-1}^\vee)$ and $\gr(r_{1}^\vee) = \gr(w_{1}^\vee)$.

	We claim that $\psi$ is a chain map. Because both sides of the equation $\partial \psi = \psi \partial$ are homogenous, it suffices to prove this in the quotient where we set $\scU = \scV = 1$. (The reader who is unconvinced of this fact may consult \cite[Section 2.4]{DaiStoffregen}, in which an analogous situation is discussed.) The claim is then straightforward from the definitions; the only subtle cases are to verify $\partial \psi = \psi \partial$ on $r_{-1}^\vee$ and $r_1^\vee$. For the former, we have
\begin{align*}
\partial \psi(r_{-1}^\vee) &= \sum_{i+j=0, \ i<j} \partial( x_{i-1}^\vee \otimes y_j^\vee ) + \partial( y_i^\vee \otimes x_{j-1}^\vee)\\
&= \sum_{i+j=0, \ i<j} x_{i-1}^\vee \otimes x_{j-1}^\vee + x_{i-1}^\vee \otimes x_{j+1}^\vee + x_{i-1}^\vee \otimes x_{j-1}^\vee + x_{i+1}^\vee \otimes x_{j-1}^\vee \\
&= \sum_{i+j=0, \ i<j} x_{i-1}^\vee \otimes x_{j+1}^\vee + x_{i+1}^\vee \otimes x_{j-1}^\vee.
\end{align*}
Identifying this as a telescoping series shows that it is equal to $x_0^\vee \otimes x_0^\vee$ (note that $x_i^\vee = 0$ for $i < -2n + 1$). A similar computation holds for $\partial \psi(r_1^\vee)$. 

Checking that $\psi$ has the requisite behavior under localization and is $\tau$-equivariant is straightforward; the only subtle cases are (again setting $\scU = \scV = 1$):
\[
\tau \psi(t^\vee) = \sum_{i + j = 0, \ i<j} y_{-i}^\vee \otimes y_{-j}^\vee = \psi(t^\vee) + \psi(z_0^\vee)
\]
together with
\[
\tau \psi(r_{-1}^\vee) = \sum_{i+j=0, \ i<j} x_{-i+1}^\vee \otimes y_{-j}^\vee + y_{-i}^\vee \otimes x_{-j+1}^\vee = \psi(r_1^\vee) + \psi(w_1^\vee)
\]
and
\[
\tau \psi(r_{1}^\vee) = \sum_{i+j=0, \ i<j} x_{-i-1}^\vee \otimes y_{-j}^\vee + y_{-i}^\vee \otimes x_{-j-1}^\vee = \psi(r_{-1}^\vee) + \psi(w_{-1}^\vee).
\]
This completes the proof.	
\end{proof}

We thus obtain the overall computation:
\begin{lemma}\label{lem:tau-sum-torus}
For $n$ odd, we have $(\CFK(T_{2n, 2n+1} \# T_{2n, 2n+1}), \tau_{\#}) \sim (\cE_n, \tau)$.
\end{lemma}
\begin{proof}
Follows from Lemmas~\ref{lem:CnCnI} and \ref{lem:CnCnII}.
\end{proof}

\begin{remark}
In \cite[Section 3.2.1]{HHSZ2}, the local equivalence class of $(T_{2n, 2n+1} \# T_{2n, 2n+1}, \iota_{\#})$ was similarly identified with $(\cE_n, \iota_K)$ for a certain involution $\iota_K$ on $\cE_n$. In fact, the map of Lemma~\ref{lem:CnCnI} is both $\tau$- and $\iota_K$-equivariant. However, the map of Lemma~\ref{lem:CnCnII} is \textit{not} $\iota_K$-equivariant. We thus do not determine the $(\tau_K, \iota_K)$-class of $T_{2n, 2n+1} \# T_{2n, 2n+1}$ in this paper; only the $\tau_K$-class.
\end{remark}

\subsection{The swapping involution}
We now turn to the $\tau_K$-class of $(T_{2n, 2n+1} \#T_{2n, 2n+1}, \tau_{sw})$. Although the full local equivalence class turns out to be difficult to compute, for our purposes it will suffice to establish an inequality. Let $\cD_n$ be the staircase complex equipped with the unique skew-graded involution of Definition~\ref{def:Dn}. Then we claim:

\begin{lemma}\label{lem:tau-swap-torus}
We have $(\cD_n, \tau) \leq (\CFK(T_{2n, 2n+1} \#T_{2n, 2n+1}), \tau_{sw})$.
\end{lemma}
\begin{proof}
Consider the subcomplex $\cW_n$ of $\CFK(T_{2n, 2n+1} \#T_{2n, 2n+1}) \simeq \cC_n \otimes \cC_n$ displayed in Figure~\ref{fig:Wn}. This is similar to the upper half of Figure~\ref{fig:Yn}, but it is not quite the same: the second of the two rows has many of the tensor products occurring with transposed factors.

\begin{figure}[h]
	\[
	\begin{tikzcd}[column sep={1.05cm,between origins}, labels=description, row sep=1cm]
	y_{1-2n}y_{1-2n}&& y_{1-2n}y_{3-2n}&& y_{3-2n}y_{3-2n}&& y_{3-2n}y_{5-2n}&& \cdots&& y_{-3}y_{-1}&& y_{-1} y_{-1}&&\,\\
	& y_{1-2n} x_{2-2n} \ar[ul, "\scV"]\ar[ur, "\scU^{2n-1}"]&& x_{2-2n}y_{3-2n}\ar[ul, "\scV"]\ar[ur, "\scU^{2n-1}"]&& y_{3-2n}x_{4-2n} \ar[ul, "\scV^2"]\ar[ur, "\scU^{2n-2}"]&& \cdots \ar[ur] \ar[ul] && y_{-3}x_{-2}\ar[ul]\ar[ur, "\scU^{n+1}"] && x_{-2} y_{-1} \ar[ul, "\scV^{n-1}"]\ar[ur, "\scU^{n+1}"] && x_{0}y_{-1} \ar[ul, "\scV^n"] \ar[ur, "\scU^n"]
	\end{tikzcd} \cdots
	\]
	\[
	\begin{tikzcd}[column sep={1.05cm,between origins}, labels=description, row sep=1cm]
	y_{1}y_{-1}&& y_{1}y_{1}&& y_{1}y_{3}&& y_{3}y_{3}&& \cdots&& y_{2n-3}y_{2n-1}&& y_{2n-1}y_{2n-1}&&\,\\
	& y_{1}x_{0} \ar[ul, "\scV^n"]\ar[ur, "\scU^n"]&& y_{1}x_{2}\ar[ul, "\scV^{n+1}"]\ar[ur, "\scU^{n-1}"]&&x_{2}y_{3}  \ar[ul, "\scV^{n+1}"]\ar[ur, "\scU^{n-1}"]&& \cdots \ar[ul] \ar[ur] && y_{2n-3}x_{2n-2}\ar[ul]\ar[ur, "\scU"] && x_{2n-2}y_{2n-1} \ar[ul, "\scV^{2n-1}"] \ar[ur, "\scU"]
	\end{tikzcd}
	\]
	\caption{The subcomplex $\cW_n$. Note that the top two rows form a staircase complex, such that $\partial(x_0y_{-1})=\scV^n y_{-1}y_{-1}+\scU^n y_1y_{-1}$.}\label{fig:Wn}
\end{figure}

We claim that $\tau_{sw}$ preserves this subcomplex and acts as the obvious reflection map. To see this, consider the exchange involution $\tau_{exch}$ defined in Section~\ref{sec:4.2}. This sends
\[
\tau_{exch}(x_iy_j) = y_{-j}x_{-i}, \quad \tau_{exch}(y_ix_j) = x_{-j}y_{-i}, \quad \text{and} \quad \tau_{exch}(y_iy_j) = y_{-j}y_{-i}.
\]
Moreover, it is clear from the definition of $\partial$ on $\cC_n$ that $\Psi\otimes \Phi$ vanishes on generators of the form $x_i y_j, y_i x_j,$ and $y_i y_j$. Applying Theorem~\ref{thm:swapping} then gives the desired computation of $\tau_{sw}$. Mapping $(\cD_n, \tau)$ isomorphically onto $(\cW_n, \tau_{sw})$ completes the proof.
\end{proof}

\subsection{The involution on $K_n$}
We now finally turn to the $\tau_K$-class of $(K_n, \tau_n)$. Our first step will be to understand the complex $\cE_n\otimes \cD_n^{\vee}$. This follows \cite[Section 3.2.2]{HHSZ2}.

\begin{definition}
Let $\cB_n$ be the box complex displayed on the left in Figure~\ref{fig:Bn}. This has five generators $v$, $r_0$, $r_{-1}$, $r_1$, and $t$, with differential
\[
\partial v = 0, \quad \partial r_0=\scV^n r_{-1} +\scU^n r_1, \quad \partial r_{-1}=\scU^n t, \quad \partial r_1=\scV^n t ,\quad \text{and} \quad \partial t=0.
\]

\begin{figure}[h!]
	\begin{minipage}{.2\textwidth}
		\[
		v \quad \begin{tikzcd}[row sep=1.1cm, column sep=1.1cm, labels=description] r_{-1}\ar[r, "\scU^n"] & t\\
		r_0\ar[u,"\scV^n"] \ar[r, "\scU^n"]& r_1 \ar[u, "\scV^n"]
		\end{tikzcd}
		\]\end{minipage}\hspace{2cm}\begin{minipage}{.2\textwidth}
		\[
		v^\vee \quad \begin{tikzcd}[row sep=1.1cm, column sep=1.1cm, labels=description] r_{-1}^\vee \ar[d,"\scV^n"] & t^\vee\ar[l, "\scU^n"] \ar[d, "\scV^n"]\\
		r^\vee_0 & r^v_1 \ar[l, "\scU^n"] 
		\end{tikzcd}
		\]\end{minipage}
	\caption{The box complex $\cB_n$ and its dual $\cB_n^{\vee}$. See \cite[Figure 3.4]{HHSZ2}.}\label{fig:Bn}
\end{figure}
\noindent
The gradings of these generators are such that $\gr(v) = \gr(t) = (0,0)$. Define an involution $\tau$ on $\cB_n$ by setting
\begin{align*}
\tau(v)&=v+t\\
\tau(r_0)&=r_0\\
\tau(r_{-1})&=r_1\\
\tau(r_1)&=r_{-1}\\
\tau(t)&=t.
\end{align*}
Note that the action of $\tau$ sends the singleton generator $v$ to (the sum of $v$ with) a square complex generator. However, unlike in \cite[Section 3.2.2]{HHSZ2}, $\tau$ does not send the opposite corner of the square back to $v$. The reader should compare the complexes $\cB_n$ and $\cS_n$.
\end{definition}

The utility of $\cB_n$ is given by the following lemma:

\begin{lemma}\label{lem:Bn}
We have $(\cB_n, \tau) \sim (\cE_n, \tau) \otimes (\cD_n, \tau)^{\vee}$.
\end{lemma}
\begin{proof}
This is similar to \cite[Proposition 3.5]{HHSZ2}. We construct maps
\[
f \colon \cE_n\to \cD_n\otimes \cB_n
\]
and 
\[
g \colon \cE^\vee_n\to \cD_n^{\vee}\otimes \cB_n^{\vee}
\]
as follows. The map $f$ is given by	
\begin{equation*}
  \begin{split}
	f(w_i) &= w_{i}v \mbox{ for } i\leq -1\\
	f(w_1) &= w_1(v + t) + z_0r_1\\	
	f(w_i) &= w_i(v+t) \mbox{ for } i\geq 3\\
	f(z_i) &= z_iv \mbox{ for } i\leq 0\\
	f(z_i) &= z_i(v+t) \mbox{ for } i\geq 2
  \end{split}
\quad \quad \quad
  \begin{split}
	f(r_0) &= z_0r_0\\
	f(r_{-1})&= z_0r_{-1}\\
	f(r_1) &= z_0r_1\\
	f(t)&= z_0t.
  \end{split}
\end{equation*}
It is easily checked that $f$ is a grading-preserving chain map; the only subtlety is checking that $\partial f = f \partial$ on $w_1$. We have:
\begin{align*}
\partial f(w_1) &= \partial(w_1(v + t) + z_0r_1) = (\scV^n z_0 + \scU^n z_2)(v + t) + \scV^n z_0t \\
f (\partial w_1) &= f(\scV^n z_0 + \scU^n z_2) =  \scV^n z_0v + \scU^n z_2(v + t),
\end{align*}
which are equal to each other. Checking $\tau$-equivariance is likewise straightforward; the only subtle cases are for $w_{-1}$,  $w_1$, and $z_0$. For these, we have
\begin{align*}
\tau f(w_{-1}) &= \tau(w_{-1}v) = w_1(v + t) \\
f(\tau w_{-1}) &= f(w_1 + r_1) = w_1(v + t) + z_0r_1 + z_0r_1
\end{align*}
and
\begin{align*}
\tau f(w_1) &= \tau(w_1(v+t) + z_0r_1) = w_{-1}v + z_0r_{-1} \\
f(\tau w_1) &= f(w_{-1} + r_{-1}) = w_{-1}v + z_0 r_{-1}
\end{align*}
and
\begin{align*}
\tau f(z_0) &= \tau(z_0 v) = z_0(v+t) \\
f(\tau z_0) &= f(z_0 + t) = z_0v + z_0t.
\end{align*}
This completes the verification of $f$.

The map $g$ is given by
\begin{equation*}
  \begin{split}
	g(w_i^\vee) &= w_{i}^\vee v^\vee\\
	g(z_i^\vee) &= z_i^\vee v^\vee
  \end{split}
\quad \quad \quad
  \begin{split}
	g(r_0^\vee) &= z_0^\vee r_0^\vee + w_{-1}^\vee r_1^\vee + w_1^\vee r_{-1}^\vee \\
	g(r_{-1}^\vee)&= w_{-1}^\vee t^\vee + z_0^\vee r_{-1}^\vee \\
	g(r_1^\vee) &= w_1^\vee t^\vee + z_0^\vee r_1^\vee\\
	g(t^\vee)&= z_0^\vee t^\vee.
  \end{split}
\end{equation*}
An examination of the right-hand side of Figure~\ref{fig:Bn} shows that $g$ is a grading-preserving chain map. Checking $\tau$-equivariance is likewise straightforward; the only subtle cases are for $r_{-1}^\vee$, $r_1^\vee$, and $t^\vee$. For these, we have
\begin{align*}
\tau g(r_{-1}^\vee) &= \tau(w_{-1}^\vee t^\vee + z_0^\vee r_{-1}^\vee) = w_1^\vee(t^\vee + v^\vee) + z_0^\vee r_1^\vee \\
g(\tau r_{-1}^\vee) &= g( r_1^\vee + w_{1}^\vee ) = w_1^\vee t^\vee + z_0^\vee r_1^\vee + w_{1}^\vee v^\vee
\end{align*}
and
\begin{align*}
\tau g(r_{1}^\vee) &= \tau(w_1^\vee t^\vee + z_0^\vee r_1^\vee) = w_{-1}^\vee(t^\vee + v^\vee) + z_0^\vee r_{-1}^\vee \\
g(\tau r_{1}^\vee) &= g( r_{-1}^\vee + w_{-1}^\vee ) = w_{-1}^\vee t^\vee + z_0^\vee r_{-1}^\vee + w_{-1}^\vee v^\vee
\end{align*}	
and
\begin{align*}
\tau g(t^\vee) &= \tau(z_0^\vee t^\vee) = z_0^\vee(t^\vee + v^\vee) \\
g(\tau t^\vee) &= g(t^\vee + z_0^\vee) = z_0^\vee t^\vee + z_0^\vee v^\vee.
\end{align*}
This completes the verification for $g$.
\end{proof}

We are now finally in a position to state our fundamental computation:

\begin{lemma}\label{lem:fundamentalinequality}
We have $(\CFK(K_n), \tau_n) \leq (\cB_n, \tau)$.
\end{lemma}
\begin{proof}
By Lemma~\ref{lem:tau-sum-torus}, we have
\[
(\CFK(T_{2n, 2n+1} \# T_{2n, 2n+1}), \tau_\#) \sim (\cE_n, \tau).
\]
By Lemma~\ref{lem:tau-swap-torus}, we have
\[
(\CFK(T_{2n, 2n+1} \# T_{2n, 2n+1}), \tau_{sw})^\vee \leq (\cD_n, \tau)^\vee.
\]
Tensoring these together and utilizing Theorem~\ref{thm:connectedsum}, we thus have that
\[
(\CFK(K_n), \tau_n) \leq (\cE_n, \tau) \otimes (\cD_n, \tau)^\vee \sim (\cB_n, \tau),
\]
where the final local equivalence follows from Lemma~\ref{lem:Bn}.
\end{proof}

This immediately yields the proof of Theorem~\ref{thm:1.4}:

\begin{proof}[Proof of Theorem~\ref{thm:1.4}]
It is straightforward to check that an inequality as in Lemma~\ref{lem:fundamentalinequality} implies an inequality of the large-surgery numerical invariants defined in Section~\ref{sec:2.5}:
\[
\dl_\tau(\CFK(K_n)_0) \leq \dl_\tau((\cB_n)_0) \quad \text{and} \quad \du_\tau(\CFK(K_n)_0) \leq \du_\tau((\cB_n)_0).
\]
See Section~\ref{sec:7.2} for further discussion. A direct computation shows that
\[
\dl_\tau((\cB_n)_0) = -2n \quad \text{and} \quad \du_\tau((\cB_n)_0) = 0.
\]
Applying Definition~\ref{def:numericalinvariants} completes the proof.
\end{proof}

\begin{remark}\label{rem:suminsensitive}
Throughout this section, we have only worked with the $\tau_K$-complexes of our knots. These are insensitive to the choice of direction. Moreover, as discussed in Section~\ref{sec:2.3}, the (possible) non-abelian nature of $\K_{\tau, \iota}$ does not arise unless the action of $\iota_K$ is considered simultaneously. Thus, the computations of this section hold regardless of the way the equivariant connected sum is performed, not just following the conventions of Figure~\ref{fig:11}.
\end{remark}

\section{Relation to other invariants}\label{sec:7}

We now relate the present paper to the results of \cite{DHM} and \cite{JZgenus}. 

\subsection{Equivariant large surgery}\label{sec:7.1}
We begin with a brief review of \cite{DHM}. Let $Y$ be a $\Z/2\Z$-homology sphere and let $\tau$ be an involution on $Y$. Note that $Y$ has a single spin structure $\s$ which is necessarily sent to itself by $\tau$. In \cite[Section 4]{DHM}, it is shown that $\tau$ induces a well-defined automorphism of $\CFm(Y, \s)$, which we also denote by $\tau$.

Moreover, in \cite[Section 4]{DHM} it is shown that $\tau$ is a homotopy involution. The pair $(\CFm(Y, \s), \tau)$ thus constitutes an abstract $\iota$-complex in the sense of \cite[Definition 8.1]{HMZ}. Taking the local equivalence class of this $\iota$-complex gives an element
\[
h_\tau(Y) = [(\CFm(Y, \s)[-2], \tau)]
\]
in the local equivalence group $\Inv$ of \cite[Proposition 8.8]{HMZ}. (The grading shift is a convention due to the definition of the grading on $\CFm$.) This is an invariant of equivariant $\Z/2\Z$-homology bordism. In fact, one can construct a $\Z/2\Z$-homology bordism group of involutions and show that $h_\tau$ constitutes a homomorphism from this group into $\Inv$; see \cite[Theorem 1.2]{DHM} and \cite[Section 2]{DHM}. We may also consider the map $\iota \circ \tau$ in place of $\tau$, which is similarly a homotopy involution. This gives another $\iota$-complex whose local equivalence class 
\[
h_{\iota \circ \tau}(Y) = [(\CFm(Y, \s)[-2], \iota \circ \tau)]
\]
is another (generally different) element of $\Inv$. For each of these elements, one can extract the numerical invariants $\du$ and $\dl$ following the procedure described by Hendricks and Manolescu \cite{HM}. This yields numerical invariants $\du_{\tau}$ and $\dl_{\tau}$ associated to $h_\tau$, as well as invariants $\du_{\iota \tau}$ and $\dl_{\iota \tau}$ associated to $h_{\iota \circ \tau}$. 

\begin{remark} 
The discussion of \cite{DHM} is phrased in terms of integer homology spheres, but the extension to $\Z/2\Z$-homology spheres is straightforward. Note that in this more general situation, the gradings of our complexes take values in $\Q$, and so $\du_\circ$ and $\dl_\circ$ may be $\Q$-valued.
\end{remark}

If $(K, \tau)$ is an equivariant knot, then any surgery on $K$ inherits an involution; see for example \cite[Section 5]{DHM}. In \cite{Mallick}, the second author showed the following:

\begin{theorem}\cite[Theorem 1.1]{Mallick}\label{thm:largesurgery}
Let $(K, \tau)$ be an equivariant knot and let $p \geq g_3(K)$. Then there is an absolutely graded isomorphism
\[
\left(\CFm(S^3_p(K), [0])\left[\dfrac{p-1}{4} - 2\right], \tau\right) \simeq \left(\CFK(K)_0, \tau_K\right).
\]
On the left-hand side, $S^3_p(K)$ is large surgery on $K$ and $\tau$ is the automorphism on $\CFm(S^3_p(K), [0])$ induced by the inherited 3-manifold involution. On the right-hand side, $\CFK(K)_0$ is the large surgery subcomplex of $\CFK(K)$ and $\tau_K$ is the restriction of the action defined in Section~\ref{sec:3.2}. A similar statement holds replacing $\tau$ with $\iota \circ \tau$ and $\tau_K$ with $\iota_K \circ \tau_K$. 
\end{theorem}

\begin{remark}\label{rem:absolutegrading}
There are several confusing conventions regarding absolute gradings. For the sake of being explicit, we give an explanation of these for the reader.
\begin{enumerate}
\item Note that the ``trivial complex" $\CFm(S^3)$ consists of a single $\F[U]$-tower starting in Maslov grading $-2$. However, when discussing Floer complexes in the abstract, it is generally preferable to treat this complex as starting in Maslov grading zero. This explains the shift by $-2$ in the definition of $h_\tau$ and $h_{\iota \circ \tau}$.
\item Similarly, if $K$ is the unknot, then the large surgery complex $\Co$ (as defined in Section~\ref{sec:2.5}) consists of a single $\F[U]$-tower starting in Maslov grading zero. This explains the extra $-2$ in the isomorphism of Theorem~\ref{thm:largesurgery}.
\item In Section~\ref{sec:2.5}, we have defined $\du_\circ(\Co)$ and $\dl_\circ(\Co)$ in such a way so that the shift by $-2$ is already taken into account. Indeed, note that $\du_\circ(\Co) = \dl_\circ(\Co) = 0$ for the large surgery complex of the unknot. However, when defining $\du_\circ$ and $\dl_\circ$ in terms of an actual 3-manifold complex $\CFm(Y)$, it is necessary to add two to each of the definitions.
\end{enumerate}
Taking into account the grading shift, Definition~\ref{def:numericalinvariants} and Theorem~\ref{thm:largesurgery} immediately imply the relations referenced in Section~\ref{sec:1.4}:
\[
-2\Vsu(K) + \dfrac{p-1}{4} = \du_\circ(S^3_p(K), [0])
\]
and
\[
-2\Vsl(K) + \dfrac{p-1}{4} = \dl_\circ(S^3_p(K), [0]).
\]
for $\circ \in \{\tau, \ita\}$.
\end{remark}

\subsection{Inequalities}\label{sec:7.2}
A key property of the local equivalence group $\Inv$ is that it is partially ordered; see for example \cite[Definition 3.6]{DHM}. This partial order is consistent with the numerical invariants $\du$ and $\dl$, in the sense that if $[(C_1, \iota_1)] \leq [(C_2, \iota_2)]$, then
\[
\du(C_1) \leq \du(C_2) \quad \text{and} \quad \dl(C_1) \leq \dl(C_2).
\]
In \cite[Theorem 1.5]{DHM}, it was shown that if $(Y_1, \tau_1)$ and $(Y_2, \tau_2)$ are two homology spheres with involutions and $W$ is an equivariant negative-definite cobordism from $Y_1$ to $Y_2$, then under certain circumstances we obtain inequalities 
\[
[(\CFm(Y_1), \tau_1)] \leq [(\CFm(Y_2), \tau_2)]
\]
and/or
\[
[(\CFm(Y_1), \iota \circ \tau_1)] \leq [(\CFm(Y_2), \iota \circ \tau_2)].
\]
It is thus possible to bound the numerical invariants of $(Y_1, \tau_1)$ by topologically constructing equivariant negative-definite cobordisms into other manifolds $(Y_2, \tau_2)$. See \cite[Section 5]{DHM} and \cite[Section 7]{DHM} for further discussion and examples.

For convenience, we briefly review these results here, generalizing them slightly in the case of $\Z/2\Z$-homology spheres. Let $Y_1$ and $Y_2$ be two $\Z/2\Z$-homology spheres equipped with involutions $\tau_1$ and $\tau_2$. Let $W$ be a cobordism from $Y_1$ to $Y_2$ equipped with a self-diffeomorphism $f \colon W \rightarrow W$ that restricts to $\tau_i$ on $Y_i$. In what follows, we will be interested in $\spinc$-structures $\s$ on $W$ such that the Heegaard Floer grading shift
\[
\Delta(W, \s) = \dfrac{c_1(\s)^2 - 2 \chi(W) - 3\sigma(W)}{4}
\]
is zero. 

\begin{theorem}\cite[Proposition 4.10]{DHM}\label{thm:inequality}
Let $W$ be a negative-definite cobordism as above with $b_1(W) = 0$. Suppose $W$ admits a $\spinc$-structure $\s$ such that $\Delta(W, \s) = 0$ and $\s$ restricts to the unique spin structure on $\partial W$. Then:
\begin{enumerate}
\item If $f_* \s = \s$, we have $h_{\tau_1}(Y_1) \leq h_{\tau_2}(Y_2)$.
\item If $f_* \s = \bs$, we have $h_{\iota \circ \tau_1}(Y_1) \leq h_{\iota \circ \tau_2}(Y_2)$.
\end{enumerate}
\end{theorem}
\begin{proof}
The proof is the same as that of \cite[Proposition 4.10]{DHM} and proceeds by considering the Heegaard Floer cobordism map associated to $(W, \s)$.
\end{proof}

We consider a particularly important family of such cobordisms, constructed as follows. Let $(K, \tau)$ be an equivariant knot and let $Y_1 = S^3_n(K)$ be large, odd surgery on $K$. Define an equivariant cobordism $W$ by symmetrically attaching $(-1)$-framed 2-handles along unknots that have linking number zero with $K$ (as well as with each other):

\begin{definition}\label{def:eqcobordisms}
Three important instances of this construction are given in the top row of Figure~\ref{fig:71}. We categorize these as follows:
\begin{enumerate}
\item Type Ia: Attach a single $(-1)$-framed 2-handle along an equivariant unknot which has no fixed points along the axis of $\tau$.
\item Type Ib: Attach a single $(-1)$-framed 2-handle along an equivariant unknot which has two fixed points along the axis of $\tau$.
\item Type II: Attach a pair of $(-1)$-framed 2-handles which are interchanged by $\tau$.
\end{enumerate}
\end{definition}

\begin{figure}[h!]
\includegraphics[scale = 0.8]{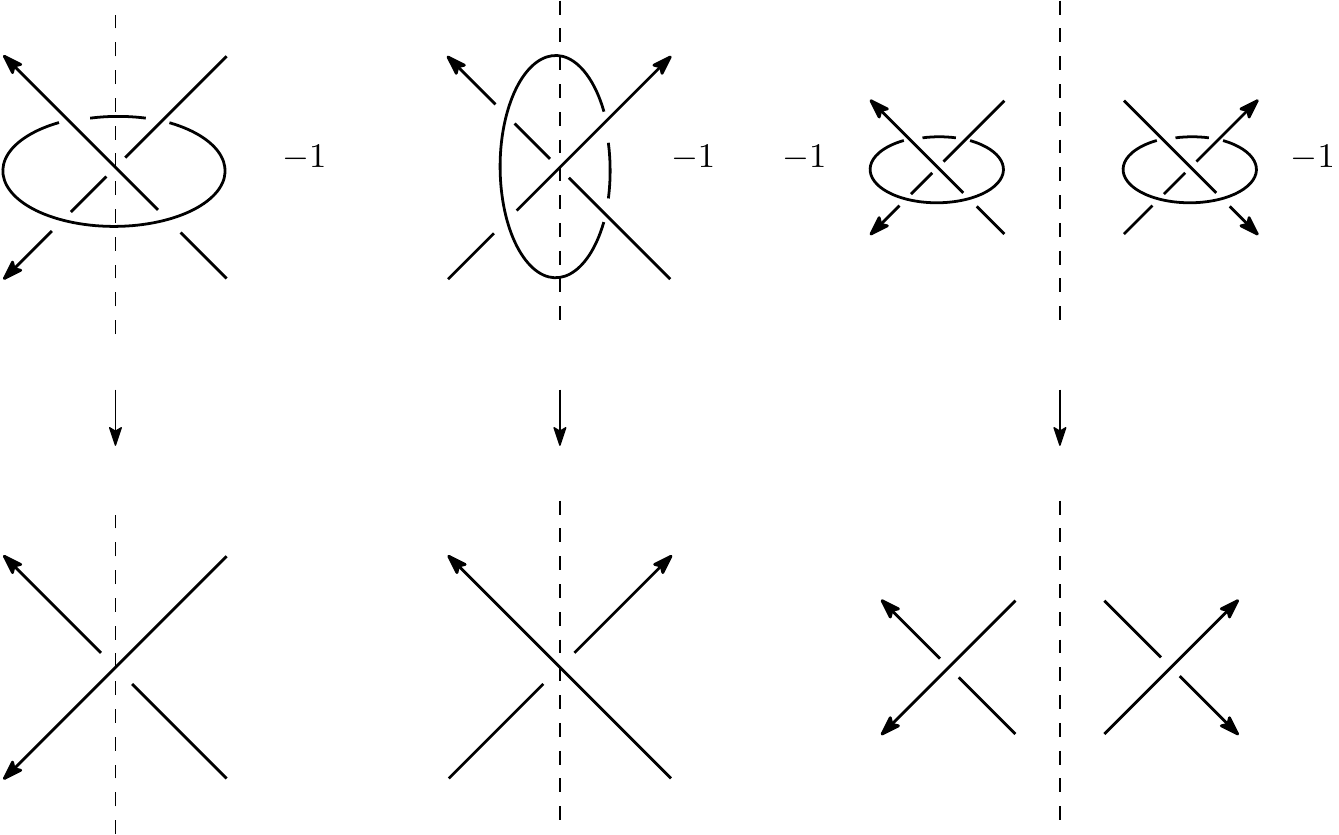}
\caption{Left: handle attachment/crossing change of Type Ia. Middle: handle attachment/crossing change of Type Ib. Right: handle attachment/crossing change of Type II.}\label{fig:71}
\end{figure}

Consider a handle attachment of Type Ia. Let $x$ be the element of $H_2(W, \partial W; \Z)$ corresponding to the core of the attached 2-handle and let $\s$ be the $\spinc$-structure on $W$ corresponding to the dual of $x$. This restricts to the unique spin structure on the ends of $W$, as can be seen from the fact that the map $H^2(W; \Z) \rightarrow H^2(\partial W; \Z)$ corresponds to the map $H_2(W, \partial W; \Z) \rightarrow H_1(\partial W; \Z)$ under Poincar\'e duality. Since $x$ has self-intersection $-1$, we moreover have $\Delta(\s) = 0$. Finally, it is easily checked that $f_* \s = \s$. Handle attachments of Type Ib are similar, except that $f_* \s = \bs$. To understand handle attachments of Type II, let $x$ and $y$ be the elements of $H_2(W, \partial W; \Z)$ represented by the cores of the attached 2-handles. Then $W$ admits both a $\spinc$-structure with $f_*\s = \s$ (corresponding to the dual of $x + y$) and a $\spinc$-structure with $f_*\s = \bs$ (corresponding to the dual of $x - y$). Once again, these both restrict to the unique spin structure on the ends of $W$ and have $\Delta(\s) = 0$. See \cite[Section 5.2]{DHM} for further discussion.

In the context of equivariant knots, this immediately gives a set of crossing change inequalities for $\Vsu$ and $\Vsl$. As in Figure~\ref{fig:71}, define:

\begin{definition}\label{def:eqcrossingchanges}
Let $K$ be an equivariant knot. We categorize equivariant positive-to-negative crossing changes as follows:
\begin{enumerate}
\item Type Ia: The crossing change occurs along the axis of symmetry and the two strands of the crossing point in opposite directions along the axis. (Figure~\ref{fig:71}, top left.)
\item Type Ib: The crossing change occurs along the axis of symmetry and the two strands of the crossing point in the same direction along the axis. (Figure~\ref{fig:71}, top middle.)
\item Type II: We perform a symmetric pair of crossing changes. (Figure~\ref{fig:71}, top right.)
\end{enumerate}
\end{definition}

\begin{proof}[Proof of Theorem~\ref{thm:1.10}]
Passing to large surgery, each of the possible crossing changes is mediated by a 2-handle attachment of the corresponding type. By Theorem~\ref{thm:inequality}, we thus have
\[
\du_\circ (S^3_p(K),[0]) \leq \du_\circ(S^3_p(K'),[0]) \quad \text{and} \quad \dl_\circ (S^3_p(K),[0]) \leq \dl_\circ(S^3_p(K'),[0])
\]
for $\circ \in \{\tau, \ita\}$. Applying the relation (\ref{eq:knotto3manifold}) immediately gives the desired conclusion.
\end{proof}

\begin{remark}
In Definition~\ref{def:eqcobordisms}, attaching our 2-handles along unknots is not essential. However, we emphasize the unknot case due to its ease of use and connection with the current paper.
\end{remark}

\subsection{Exotic slice disks}\label{sec:7.3}

We now turn to the proofs of Theorems~\ref{thm:1.7} and \ref{thm:knotfloermaps}. We establish the former by showing that at least one of $\Vtl(J)$ and $\Vitl(J)$ is greater than zero.

\begin{proof}[Proof of Theorem~\ref{thm:1.7}]
We exhibit an equivariant cobordism (of Type II) from $(+1)$-surgery on $J$ to $(-1)$-surgery on the knot $6_2$, where the latter is equipped with a certain involution $\tau$. This is done in Figures~\ref{fig:72} and \ref{fig:73}; compare \cite[Section 7.5]{DHM}. Note that the first picture in Figure~\ref{fig:72} certainly constitutes an equivariant cobordism from $S^3_{+1}(J)$ to \textit{some} homology sphere with involution; we identify this homology sphere as $S^3_{-1}(6_2)$. However, we will not bother to make this identification equivariant, although this can be done. That is, it turns out (rather surprisingly) that we will not need to explicitly identify the involution $\tau$ on $S^3_{-1}(6_2)$.

\begin{figure}[h!]
\includegraphics[scale = 0.62]{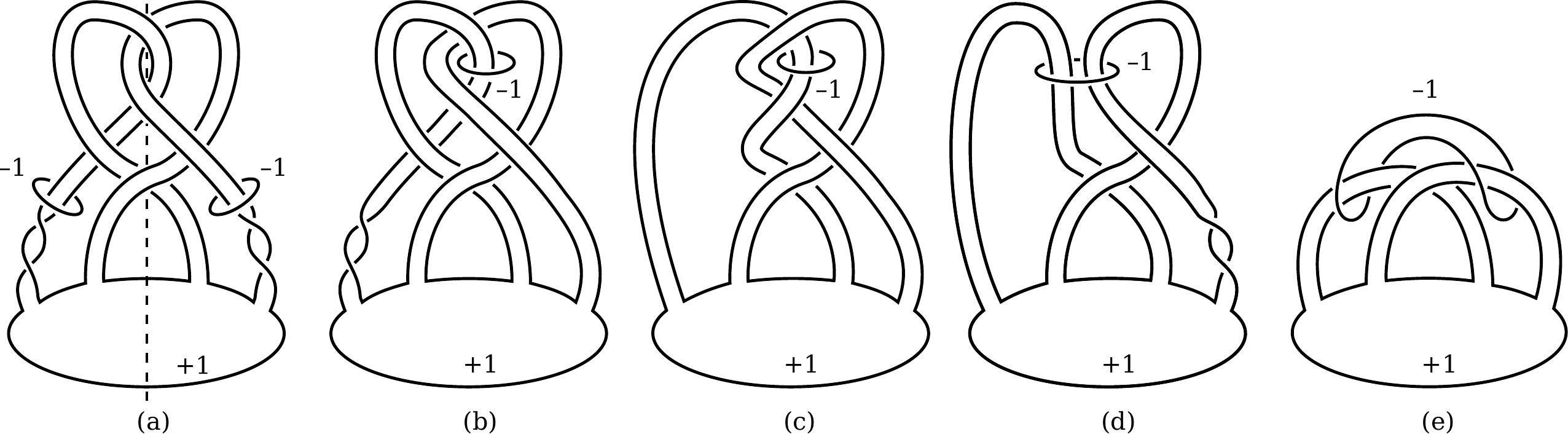}
\caption{Attaching a pair of equivariant 2-handles to $(+1)$-surgery on $J$. In (a), we perform the equivariant handle attachment. In (b), we blow down the right-hand unknot and slide the left-hand unknot partway along its band. In (c), we untwist the left-hand band. In (d), we slide the right-hand band over the unknot. In (e), we untwist the right-hand band. Manipulations are continued in Figure~\ref{fig:73}.}\label{fig:72}
\end{figure}

\begin{figure}[h!]
\includegraphics[scale = 0.8]{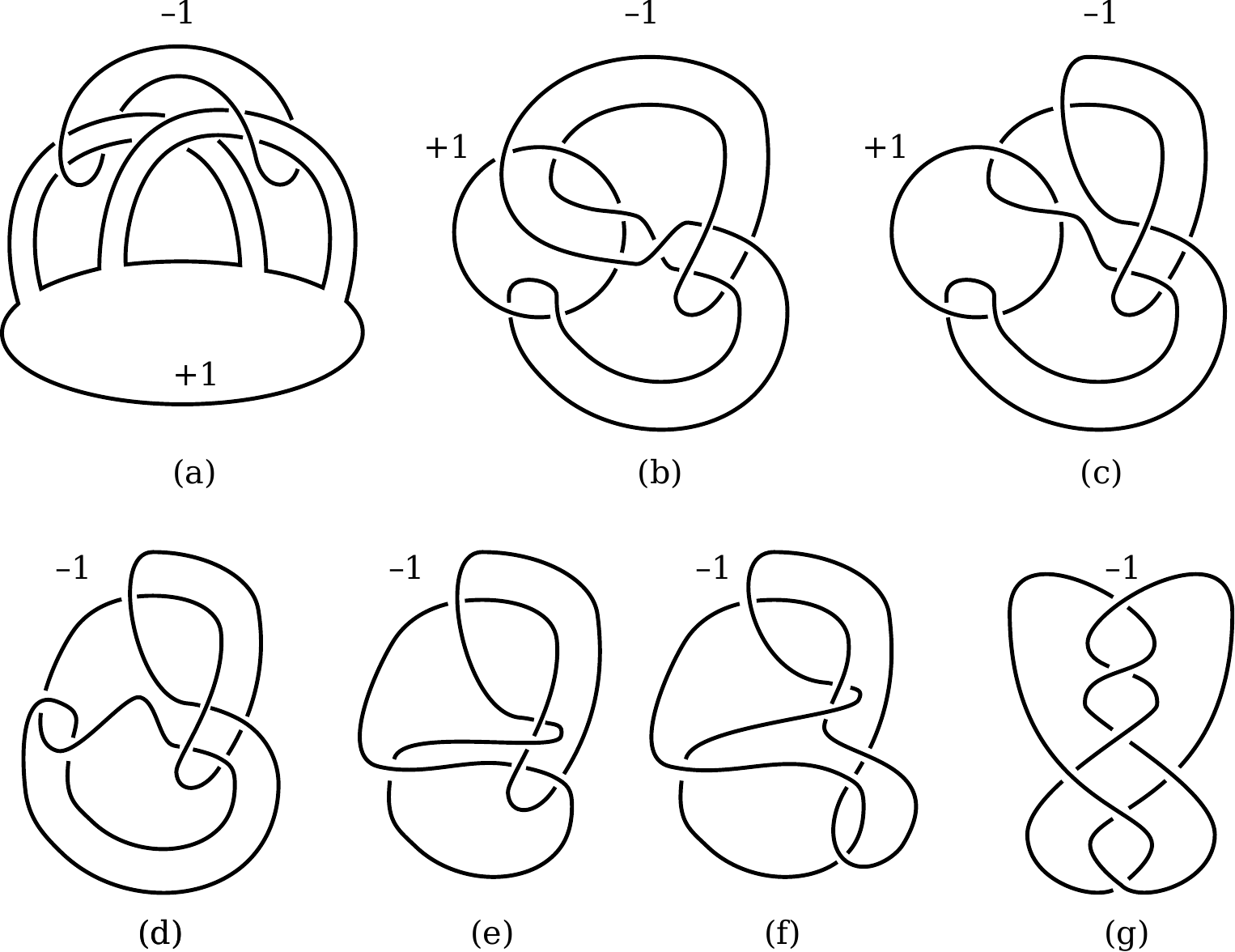}
\caption{Identifying the result of the equivariant handle attachment. Figure (a) is a copy of (e) from Figure~\ref{fig:72}; note that both knots in the figure are unknots. In (b), we retract the right-hand band of the $(+1)$-curve along itself to clearly make it into an unknot. In (c), we move a strand of the $(-1)$-curve slightly to make the blowdown more apparent. In (d) we blow down the $(+1)$-curve. In (e) through (g), we isotope the result to look like the end result of \cite[Figure 39]{DHM}. The reader may check that this is $6_2$.}\label{fig:73}
\end{figure}

Since the cobordism is of Type II, by Theorem~\ref{thm:inequality} it suffices to show that at least one of 
\[
\dl_\tau(S^3_{-1}(6_2)) \quad \text{and} \quad \dl_{\iota \tau}(S^3_{-1}(6_2))
\]
is strictly less than zero. Indeed, we would then have that at least one of $\dl_\tau(S^3_{+1}(J))$ and $\dl_{\iota \tau}(S^3_{+1}(J))$ is strictly less than zero. Since $J$ is genus one, $(+1)$-surgery constitutes large surgery, so this implies that one of $\Vtl(J)$ and $\Vitl(J)$ is strictly greater than zero. (Note that the grading shift $(p-1)/4$ in the relation (\ref{eq:knotto3manifold}) in this case is zero.)

The desired claim is established in \cite[Section 7.5]{DHM}, but for the sake of completeness we outline the argument here. Firstly, the knot Floer homology of $6_2$ is easily calculated from the Alexander polynomial of $6_2$. The Floer homology $\HFm(S^3_{-1}(6_2))$ can then be calculated using the surgery formula; the action of the Hendricks-Manolescu involution $\iota$ on $\HFm(S^3_{-1}(6_2))$ can also be calculated (see for example \cite[Section 7.5]{DHM}). The result is displayed in Figure~\ref{fig:74}. 

\begin{figure}[h!]
\includegraphics[scale = 1]{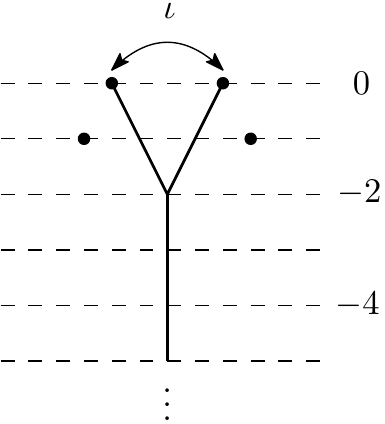}
\caption{Floer homology $\HFm(S^3_{-1}(6_2))$ with action of $\iota$.}\label{fig:74}
\end{figure}

Now, either $\tau$ acts on $\HFm(S^3_{-1}(6_2))$ by fixing the central $Y$-shape, or it acts on the central $Y$-shape by reflection. By direct calculation, in the former case we have 
\[
\dl_\tau(S^3_{-1}(6_2)) = 0 \quad \text{and} \quad \dl_{\iota \tau}(S^3_{-1}(6_2)) = -2
\]
while in the latter, we have
\[
\dl_\tau(S^3_{-1}(6_2)) = -2 \quad \text{and} \quad \dl_{\iota \tau}(S^3_{-1}(6_2)) = 0.
\]
This completes the proof.
\end{proof}

\begin{remark}\label{rem:immediateproof}
It possible to provide an immediate proof of Theorem~\ref{thm:1.7} as a topological corollary to \cite[Theorem 1.15]{DHM}, as follows. Let $\Sigma$ and $\tau_W(\Sigma)$ be any pair of symmetric slice disks for $J$ in some homology ball $W$. One can show that $(+1)$-surgery on $J$ is diffeomorphic to $Y = \partial W_0$, where $W_0$ is the positron cork of Akbulut-Matveyev \cite{AkbulutMat}. Moreover, under this diffeomorphism, the induced action of $\tau$ on $S^3_{+1}(J)$ is the usual cork involution on $Y$. Extend the $(+1)$-surgery on $J$ along the disks $\Sigma$ and $\tau_W(\Sigma)$ to obtain two homology balls $B_1$ and $B_2$, each with boundary $Y$. (Here, $\partial B_1$ and $\partial B_2$ are identified via the obvious identity map.) Using $\tau_W$, we obtain a diffeomorphism $f \colon B_1 \rightarrow B_2$ which restricts to the cork involution on $Y$. (Note that $f$ restricts to $\tau_W$ on the complement of a tubular neighborhood of $\Sigma$.) If $\Sigma$ and $\tau_W(\Sigma)$ were isotopic rel boundary, then we would have that $(W, \Sigma)$ and $(W, \tau_W(\Sigma))$ were diffeomorphic rel boundary. This would imply the existence of a diffeomorphism $g \colon B_1 \rightarrow B_2$ restricting to the identity on $Y$. Then $g^{-1} \circ f$ is a self-diffeomorphism of $B_1$ restricting to the cork involution on $Y = \partial B_1$. However, in \cite[Theorem 1.15]{DHM}, it is shown that no such extension exists.
\end{remark}

We now explain why the proof of Theorem~\ref{thm:1.7} implies Theorem~\ref{thm:knotfloermaps}. We begin with the following:


\begin{lemma}\label{lem:tautological}
Let $(K, \tau)$ be a strongly invertible knot in $S^3$. Let $W$ be any (smooth) homology ball with $\partial W = S^3$ and let $\tau_W$ be any extension of $\tau$ over $W$. 
Let $\Sigma$ and $\tau_W(\Sigma)$ be a pair of symmetric slice disks for $K$. Then $\tau_K([F_{W, \Sigma}(1)]) = [F_{W, \tau_W(\Sigma)}(1)]$ as elements of $H_*(\CFK(K))$. 
\end{lemma}
\begin{proof}
The proof is the same as that of Theorem~\ref{thm:1.2}. Let $\mathcal{F}$ be a decoration on $\Sigma$. We again have the commutative diagram 
\[\begin{tikzcd}
	{\mathcal{CFK}(U) } && {\mathcal{CFK}(K)} \\
	\\
	{\mathcal{CFK}(U^r) } && {\mathcal{CFK}(K^r)} \\
	\\
	{\mathcal{CFK}(U) } && {\mathcal{CFK}(K)}
	\arrow["{F_{W,\mathcal{F}}}", from=1-1, to=1-3]
	\arrow["{t}"', from=1-1, to=3-1]
	\arrow["{t}", from=1-3, to=3-3]
	\arrow["{F_{W,\tau_W(\mathcal{F})}}", from=3-1, to=3-3]
	\arrow["sw"', from=3-1, to=5-1]
	\arrow["sw", from=3-3, to=5-3]
	\arrow["{F_{W,sw(\tau_W(\mathcal{F}))}}", from=5-1, to=5-3]
\end{tikzcd}\]
Note that the decoration $sw(\tau_W(\mathcal{F}))$ is associated to the disk $\tau_W(\Sigma)$, which is not necessarily isotopic to $\Sigma$. Since $\Sigma$ and $\tau_W(\Sigma)$ are disks, no additional subtlety involving the decoration arises, and we may instead write $F_{W, \Sigma}$ and $F_{W, \tau_W(\Sigma)}$ along the top and bottom rows in the above diagram, respectively, to represent that these maps are unique up to chain homotopy. Using the fact that $\tau_K$ acts trivially on $\CFK(U)$ immediately gives the claim.
\end{proof}

The nontriviality of our numerical invariants then easily obstructs $F_{W, \Sigma}(1)$ and $F_{W, \tau(\Sigma)}(1)$ from being homologous:

\begin{proof}[Proof of Theorem~\ref{thm:knotfloermaps}]
Let $\Sigma$ and $\tau_W(\Sigma)$ be a pair of symmetric slice disks for $J$ and suppose that $[F_{W, \Sigma}(1)] = [F_{W, \tau_W(\Sigma)}(1)]$ as elements of $H_*(\CFK(J))$. Lemma~\ref{lem:tautological} then implies that $[F_{W, \Sigma}(1)]$ is a $\tau_K$-invariant element in $H_*(\CFK(J))$. Using this (and the fact that $F_{W, \Sigma}$ has zero grading shift), it is straightforward to construct an absolutely graded, $\tau_K$-equivariant local map from the trivial complex into $\CFK(J)$. This shows that $0 \leq \dl_\tau(\CFK(J)_0)$ and thus that $\Vtl(J) \leq 0$. Moreover, since $F_{W, \Sigma}$ homotopy commutes with $\iota_K$, we also know that $[F_{W, \Sigma}(1)]$ is $\iota_K$-equivariant. Hence $[F_{W, \Sigma}(1)]$ is in fact $\iota_K \circ \tau_K$-equivariant. This likewise shows that $0 \leq \dl_{\ita}(\CFK(J)_0)$ and thus that $\Vitl(J) \leq 0$, contradicting the proof of Theorem~\ref{thm:1.7}.

We now verify that $[F_{W, \Sigma}(1)] \neq [F_{W, \tau_W(\Sigma)}(1)]$ as elements of $\HFKhat(J) \cong H_*(\CFK(J)/(U,V))$. This follows algebraically from the previous paragraph and an analysis of $\CFK(J)$. We first calculate the ranks of $\smash{\HFKhat(J)}$ in each Alexander and Maslov grading. This can be done using the knot Floer calculator implemented in SnapPy \cite{SnapPy}; the results are displayed in Figure~\ref{fig:hfkj}.

\begin{figure}[h!]
\begin{tabular}{ c | c | c }
Alexander & Maslov & Rank of $\HFKhat(J)$ \\
\hline
$-1$ & $-2$ & $2$ \\
$-1$ & $-1$ & $2$ \\
\hline
$0$ & $-1$ & $4$ \\
$0$ & $0$ & $5$ \\
\hline
$1$ & $0$ & $2$ \\ 
$1$ & $1$ & $2$
\end{tabular}
\caption{Rank of $\HFKhat(J)$ in each Alexander and Maslov grading.}\label{fig:hfkj}
\end{figure}

Note that this computation of $\HFKhat(J)$ uses the conventions of Ozsv\'ath-Szab\'o. Although $J$ is not thin, a similar analysis as in (for example) \cite{Petkovacables} allows us to determine the full knot Floer complex from the hat version. Translating into the conventions used by Zemke gives the complex displayed in Figure~\ref{fig:cfkj}. (For a discussion of this procedure, see for example \cite[Section 2]{DHSTconcord}.) This consists of a singleton generator $v$, together with four squares. Two of these squares are spanned by $\scU$- or $\scV$-powers of $\{a_i, b_i, c_i, d_i\}$ $(i = 1, 2)$ and have a corner in $(\grU, \grV)$-bigrading $(0, 0)$. The other two are spanned by  by $\scU$- or $\scV$-powers of $\{e_i, f_i, g_i, h_i\}$ $(i = 1, 2)$ and have a corner in $(\grU, \grV)$-bigrading $(-1, -1)$.
\begin{figure}[h!]
\begin{minipage}{.2\textwidth}
\[
v \quad \ \begin{tikzcd}[row sep=1.1cm, column sep=1.1cm, labels=description] c_i\ar[r, "\scU"] & d_i\\
a_i\ar[u,"\scV"] \ar[r, "\scU"]& b_i \ar[u, "\scV"]
\end{tikzcd}
\]
\end{minipage}
\begin{minipage}{.2\textwidth}
\[
\begin{tikzcd}[row sep=1.1cm, column sep=1.1cm, labels=description] g_i\ar[r, "\scU"] & h_i\\
e_i\ar[u,"\scV"] \ar[r, "\scU"]& f_i \ar[u, "\scV"]
\end{tikzcd}
\]
\end{minipage}
\begin{minipage}{.2\textwidth}
\[
\begin{tabular}{ c | c | c }
\text{Generator} & $\grU$ & $\grV$ \\
\hline
$v$ & 0 & 0 \\
\hline
$a_i$ & 0 & 0 \\
$b_i$ & $1$ & $-1$ \\
$c_i$ & $-1$ & $1$ \\
$d_i$ & 0 & 0 \\
\hline
$e_i$ & $-1$ & $-1$ \\
$f_i$ & 0 & $-2$ \\
$g_i$ & $-2$ & 0 \\
$h_i$ & $-1$ & $-1$
\end{tabular}
\]
\end{minipage}

\caption{The complex $\CFK(J)$, spanned by $v$ together with $\{a_i, b_i, c_i, d_i\}$ and $\{e_i, f_i, g_i, h_i\}$ for $i = 1, 2$. Bigradings of generators are given on the right.}\label{fig:cfkj}
\end{figure}
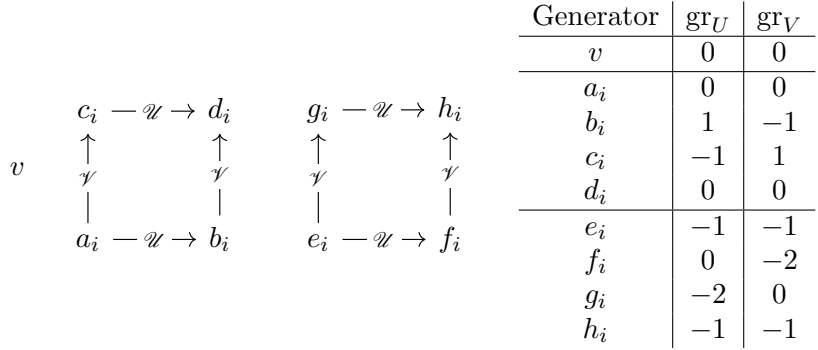

Now, $x = F_{W, \Sigma}(1)$ and $\tau_K x = F_{W, \tau_W(\Sigma)}(1)$ are cycles in $\CFK(J)$ which we know are not homologous. In the quotient $\CFK(J)/(U,V)$, the images of $x$ and $\tau_K x$ remain cycles. The only way for these images to become homologous in $\CFK(J)/(U,V)$ would be for $x - \tau_K x$ to be (homologous to) a nonzero element of $\CFK(J)$ lying in the image of $(\scU, \scV)$. However, an examination of Figure~\ref{fig:cfkj} shows that there are no elements of $\CFK(J)$ with $\grU = \grV = 0$ which lie in the image of $(\scU, \scV)$, a contradiction.
\end{proof}

We now show that taking the $n$-fold connected sum of $J$ with itself gives a slice knot with $2^n$ distinct exotic slice disks, distinguished by their concordance maps on $\smash{\HFKhat}$. For a similar construction, see \cite[Corollary 6.6]{SundbergSwann}.

\begin{theorem}\label{thm:2ndisks}
The (equivariant) connected sum $\#_n J$ admits $2^n$ distinct exotic slice disks, distinguished by their concordance maps on $\smash{\HFKhat}$.
\end{theorem}
\begin{proof}
As in Figure~\ref{fig:12}, let $D$ and $D'$ be the pair of exotic slice disks for $J$ from \cite[Section 2.1]{Hayden}. For each binary string $s$ of length $n$, there is an obvious slice disk $D_s$ for $\#_n J$ constructed by taking the boundary sum of copies of $D$ and $D'$. Explicitly, each index in $s$ with a $0$ contributes a copy of $D$, while each index with a $1$ contributes a copy of $D'$; see Figure~\ref{fig:2ndisks}. The fact that $D$ and $D'$ are topologically isotopic easily shows that the $2^n$ disks constructed in this manner are topologically isotopic rel boundary. 
\begin{figure}[h!]
\includegraphics[scale = 3.5]{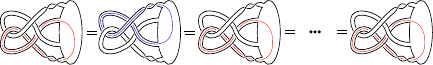}
\caption{Schematic depiction of $\#_nJ$. Compressing along the indicated curves gives the slice disk $D_s$ corresponding to the binary string $s = 010 \cdots 0$.}\label{fig:2ndisks}
\end{figure}

Now, we may identify
\[
\HFKhat(\#_nJ) = \bigotimes_n \HFKhat(J).
\]
Under this identification, $[F_{B^4, D_s}(1)]$ is the tensor product of copies of $[F_{B^4, D}(1)]$ and $[F_{B^4, D'}(1)]$, each in the appropriate index. But $[F_{B^4, D}(1)] \neq [F_{B^4, D'}(1)]$ as elements of the vector space $\smash{\HFKhat(J)}$. It follows that the $[F_{B^4, D_s}(1)]$ are different for different strings $s$. This completes the proof.
\end{proof}

Finally, we generalize Theorem~\ref{thm:1.7} to an infinite family of knots $J_n$ with exotic pairs of slice disks, considered by Hayden in \cite[Section 2.3]{Hayden}. These are displayed on the left in Figure~\ref{fig:exoticfamily} and are obtained from the knot $J$ of Theorem~\ref{thm:1.7} by adding pairs of (negative) full twists, as indicated. We have an obvious pair of slice disks for $J_n$ given by compressing along the displayed red and blue curves. In \cite[Figure 9]{Hayden}, Hayden constructs a handle diagram for the complement of these disks; it is immediate from \cite[Figure 9]{Hayden} that the disk exteriors have fundamental group $\Z$ and thus that the disks are topologically isotopic. Here, we show that knot Floer homology obstructs any two symmetric pair of disks for $J_n$ from being smoothly isotopic.

\begin{figure}[h!]
\includegraphics[scale = 0.7]{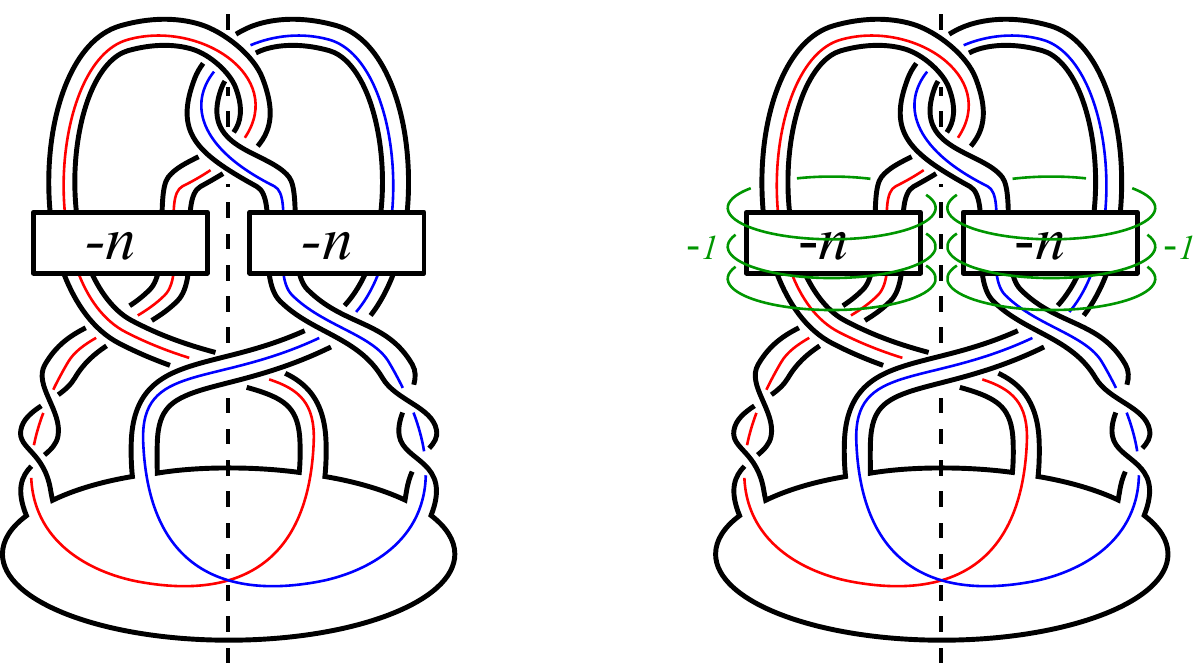}
\caption{An infinite family of knots $J_n$ admitting exotic pairs of slice disks. See \cite[Theorem B]{Hayden} and \cite[Section 2.3]{Hayden}.}\label{fig:exoticfamily}
\end{figure}

\begin{theorem}\label{thm:exoticfamily}
Let $J_n$ (for $n \geq 0$) be as in Figure~\ref{fig:exoticfamily}. Then $\ieg(J_n) > 0$. In particular, no pair of symmetric slice disks $\Sigma$ and $\tau_W(\Sigma)$ are (smoothly) isotopic rel $J_n$. This holds for any (smooth) homology ball $W$ with $\partial W = S^3$ and any extension $\tau_W$ of $\tau$ over $W$.
\end{theorem}
\begin{proof}
Clearly, $(+1)$-surgery on $J_n$ admits a negative-definite equivariant cobordism to $(+1)$-surgery on $J$, given by attaching $(-1)$-framed $2$-handles along the green curves indicated on the right in Figure~\ref{fig:exoticfamily}. Noting that each $J_n$ has Seifert genus one, it follows from Theorem~\ref{thm:inequality} that
\[
\Vtu(J_n) \geq \Vtu(J) \quad \text{and} \quad \Vtl(J_n) \geq \Vtl(J)
\]
and
\[
\Vitu(J_n) \geq \Vitu(J) \quad \text{and} \quad \Vitl(J_n) \geq \Vitl(J).
\]
The claim then immediately follows from our bounds on the invariants of $J$.
\end{proof}

\subsection{Secondary invariants}\label{sec:7.4}
We now discuss the secondary invariant $V_0(\Sigma, \Sigma')$ of \cite[Section 4.5]{JZstabilization}. This is defined as follows.  Let $K$ be a knot in $S^3$. For simplicity, let $W$ be a homology ball with boundary $S^3$ and let $\Sigma$ and $\Sigma'$ be two slice disks for $K$ in $W$. Consider the elements $F_{W, \Sigma}(1)$ and $F_{W, \Sigma'}(1)$ in $\CFK(K)$. (Note that since $\Sigma$ and $\Sigma'$ are disks, no choice of decoration is needed; the more general definition of $V_0(\Sigma, \Sigma')$ in \cite[Section 4.5]{JZstabilization} requires a discussion of a specific set of dividing curves.) Viewing $F_{W, \Sigma}(1)$ and $F_{W, \Sigma'}(1)$ as elements of the large surgery subcomplex $\CFK(K)_0$, define
\[
V_0(\Sigma, \Sigma') = \min \{n\in \Z^{\geq 0} \ \vert \ U^n \cdot [F_{W, \Sigma}(1)] = U^n \cdot [F_{W, \Sigma'}(1)] \text{ in } H_*(\CFK(K)_0)\}.
\]
In \cite[Theorem 1.1]{JZstabilization}, it is shown that $V_0(\Sigma, \Sigma')$ bounds the stabilization distance between $\Sigma$ and $\Sigma'$ from below:
\[
V_0(\Sigma, \Sigma') \leq \left \lceil \dfrac{\must(\Sigma, \Sigma')}{2} \right \rceil.
\]

The following is straightforward:

\begin{proof}[Proof of Theorem~\ref{thm:relativeV0}]
Let $\Sigma$ be any slice disk for $K$ in $W$, and suppose that $V_0(\Sigma, \tau_W(\Sigma)) = n$. By Lemma~\ref{lem:tautological}, we have 
\[
\tau_K ([F_{W, \Sigma}(1)]) = [F_{W, \tau_W(\Sigma)}(1)]. 
\]
Multiplying both sides by $U^n$, we obtain a $\tau_K$-invariant element in the homology of $\CFK(K)_0$ with grading $-2n$. This implies $\dl_\tau(\CFK(K)_0) \geq -2n$ and thus that $\Vtl(K) \leq n$. Moreover, since $F_{W, \Sigma}$ homotopy commutes with $\iota_K$, we also know that $[F_{W, \Sigma}(1)]$ is $\iota_K$-equivariant. Hence we obtain an $\iota_K \circ \tau_K$-invariant element in the homology of $\CFK(K)_0$ with grading $-2n$. This similarly shows that $\Vitl(K) \leq n$.
\end{proof}

\section{Periodic knots}\label{sec:8}
We close this paper by discussing a similar family of results in the periodic setting. As in the strongly invertible case, it is possible to define an action of $\tau_K$ associated to a $2$-periodic knot $K$ and consider the notion of a \textit{periodic $(\tau_K, \iota_K)$-complex}. The same subtlety as in Section~\ref{sec:3.4} arises, in that this is only an invariant of equivariant concordance in the decorated category. Nevertheless, once again we may define numerical invariants $\Vsu$ and $\Vsl$, and these give bounds for the equivariant slice genus. While much of the formalism is thus the same, the authors have not yet been able to find many interesting calculations of periodic invariants. One key difference is that in the periodic case, there is no natural notion of equivariant connected sum. Correspondingly, it turns out that the set of periodic $(\tau_K, \iota_K)$-complexes (up to local equivalence) does not seem to admit a natural group structure.

\subsection{Construction of $\tau_K$}\label{sec:8.1}
We begin with the construction of $\tau_K$. Let $(K, \tau)$ be a $2$-periodic knot. In contrast to the strongly invertible case, it is natural to assume that $K$ is oriented (since $K$ may not come with an orientation-reversing diffeomorphism). Let $w$ and $z$ be a pair of basepoints on $K$ which are interchanged by $\tau$, and let $\mathcal{H}$ be any choice of compatible Heegaard data for $(K, w, z)$. Taking the pushforward under $\tau$ gives a tautological isomorphism
 \[
 t \colon \CFK(\mathcal{H}) \rightarrow \CFK(\tau \mathcal{H}).
  \]
The latter complex represents the same knot, but now the roles of the basepoints have been interchanged. We now apply a half-Dehn twist which moves $w$ into $z$ and $z$ into $w$:
\[
\rho \colon \CFK(\mathcal{\tau H}) \rightarrow \CFK(\rho \tau \mathcal{H}).
\]
Finally, we have the naturality map
 \[
 \Phi(\rho \tau \mathcal{H}, \mathcal{H}) \colon \CFK(\rho \tau \mathcal{H}) \rightarrow \CFK(\H).
 \]
We thus define $\tau_\H$ to be
\[
\tau_{\H}: \CFK(\mathcal{H}) \xrightarrow{t} \CFK(\tau \mathcal{H}) \xrightarrow{\rho} \CFK(\rho \tau \mathcal{H}) \xrightarrow{\Phi} \CFK(\mathcal{H}).
\]
As before, $\tau_\H$ is independent of the choice of Heegaard data for $(K, w, z)$. 

The proof of the following is analogous to that of Theorem~\ref{thm:1.8}:
\begin{theorem}\label{thm:8.1}
Let $(K, \tau)$ be an (oriented) $2$-periodic knot and fix a pair of symmetric basepoints $(w, z)$ on $K$. Let $\H$ be any choice of Heegaard data compatible with $(K, w, z)$. Then $\tau$ induces an automorphism
\[
\tau_\H \colon \CFK(\H) \rightarrow \CFK(\H)
\]
with the following properties:
\begin{enumerate}
\item $\tau_\H$ is filtered and $\F[\cU, \cV]$-equivariant
\item $\tau_\H^2 \simeq \varsigma_\H$
\item $\tau_\H \circ \iota_\H \simeq \iota_\H \circ \tau_\H$
\end{enumerate}
Here, $\iota_\H$ is the Hendricks-Manolescu knot Floer involution on $\CFK(\H)$ and $\varsigma_\H$ is the Sarkar map. Moreover, the homotopy type of the triple $(\CFK(\H), \tau_\H, \iota_\H)$ is independent of the choice of Heegaard data $\H$ for the doubly-based knot $(K, w, z)$.
\end{theorem}
\begin{proof}
Left to the reader; analogous to Theorem~\ref{thm:1.8}.
\end{proof}

Note the difference in all three properties with Theorem~\ref{thm:1.8}.

\begin{remark}
It turns out that the analogous subtlety to Section~\ref{sec:3.3} does not arise at this stage: the homotopy class of $(\CFK(\H), \tau_\H, \iota_\H)$ is independent of the choice of symmetric basepoints $w$ and $z$. This is because any two pairs $(w, z)$ and $(w', z')$ on $K$ are related by a $\tau$-equivariant basepoint-pushing diffeomorphism along $K$. Unlike in the strongly invertible case, the associated pushforward map commutes with all the components of $\tau_\H$. Combined with the fact that $K$ comes with an orientation, this shows that we may unambiguously refer to the $(\tau_K, \iota_K)$-complex of $(K, \tau)$, without specifying any additional data. 
\end{remark}

\subsection{Periodic $(\tau_K, \iota_K)$-complexes} \label{sec:8.2}
Given Theorem~\ref{thm:8.1}, it is natural to define a $(\tau_K, \iota_K)$-complex formalism in the periodic setting:

\begin{definition}\label{def:periodicticomplex}
A \textit{periodic $(\tau_K, \iota_K)$-complex} is a triple $(C, \tau_K, \iota_K)$ such that:
\begin{enumerate}
\item $C$ is an abstract knot complex
\item $\iota_K \colon C \rightarrow C$ is a skew-graded, $\mathcal{R}$-skew-equivariant chain map such that
\[
\iota_K^2 \simeq \varsigma_K
\]
\item $\tau_K$ is a graded, $\mathcal{R}$-equivariant chain map such that
\[
\tau_K^2 \simeq \varsigma_K \quad \text{and} \quad \tau_K \circ \iota_K \simeq \iota_K \circ \tau_K.
\]
\end{enumerate}
\end{definition}
\noindent
The notions of homotopy equivalence and local equivalence carry over without change. We may also define the notion of a twist by $\varsigma_K$ as before. However, it should be noted that there is no analogue of Lemma~\ref{lem:onetwist} in the periodic setting.

\begin{remark}
The principal difference between local equivalence in the periodic and strongly invertible settings is the absence of a natural group structure in the former. Indeed, the reader can check that trying a product law such as
\[
\tau_\otimes = \tau_1 \otimes \tau_2 
\]
or even
\[
\tau_\otimes = (\id \otimes \id + \Phi \otimes \Psi) \circ (\tau_1 \otimes \tau_2)
\]
does not satisfy $\tau_\otimes^2 \simeq \varsigma_\otimes$.
\end{remark}

The algebraic procedure of Section~\ref{sec:2.5} also carries over without change to define numerical invariants $\Vtu(K), \Vtl(K), \Vitu(K)$, and $\Vitl(K)$. In \cite{Mallick}, the second author established a large surgery formula for periodic knots. Thus, $\Vsu$ and $\Vsl$ again have the interpretation as invariants associated to large surgeries.

\subsection{Equivariant concordance and cobordism} \label{sec:8.3}
As in Section~\ref{sec:2.1}, we may define the notion of an isotopy-equivariant homology concordance between two periodic knots $(K_1, \tau_1)$ and $(K_2, \tau_2)$. The subtlety of Section~\ref{sec:3.4} again arises: even if $\Sigma$ is equivariant or isotopy-equivariant, it is unclear whether an equivariant or isotopy-equivariant pair of arcs on $\Sigma$ can be chosen. We thus have:

\begin{theorem}
Let $(K_1, \tau_1)$ and $(K_2, \tau_2)$ be two periodic knots in $S^3$. Suppose that we have an isotopy-equivariant homology concordance between $(K_1, \tau_1)$ and $(K_2, \tau_2)$. Then $(\CFK(K_1), \tau_{K_1}, \iota_{K_1})$ is locally equivalent to either $(\CFK(K_2), \tau_{K_2}, \iota_{K_2})$ or $(\CFK(K_2), \varsigma_{K_2} \circ \tau_{K_2}, \iota_{K_2})$. Hence $\Vsu(K)$ and $\Vsl(K)$ are invariant under isotopy-equivariant homology concordance.
\end{theorem}
\begin{proof}
Left to the reader; analogous to Theorem~\ref{thm:3.3B}.
\end{proof}

Finally, we formally record that the results of Theorem~\ref{thm:1.2} and Theorem~\ref{thm:relativeV0} hold in the periodic setting:

\begin{theorem}
Let $(K, \tau)$ be a $2$-periodic knot in $S^3$. Then for $\circ \in \{\tau, \iota\tau\}$,
\[
- \left \lceil{\frac{1+ \ieg(K)}{2}} \right \rceil \leq \Vsu(K) \leq \Vsl(K) \leq \left \lceil{\frac{1+ \ieg(K)}{2}} \right \rceil.
\]
\end{theorem}
\begin{proof}
Left to the reader; analogous to Theorem~\ref{thm:1.2}.
\end{proof}

\begin{theorem}
Let $(K, \tau)$ be any $2$-periodic knot in $S^3$. Let $W$ be any (smooth) homology ball with boundary $S^3$, and let $\tau_W$ be any extension of $\tau$ over $W$. If $\Sigma$ is any slice disk for $K$ in $W$, then
\[
\max\{\Vtl(K), \Vitl(K)\} \leq V_0(\Sigma, \tau_W(\Sigma)).
\]
\end{theorem}
\begin{proof}
Left to the reader; analogous to Theorem~\ref{thm:relativeV0}.
\end{proof}

\bibliographystyle{amsalpha}
\bibliography{bib}

\end{document}